\documentclass{article} 

\usepackage{amsmath,amssymb,amsthm}
\usepackage{bm}
\usepackage{hyperref} 
\usepackage{comment}
\usepackage{microtype}
\usepackage{graphicx}
\usepackage{subfigure}
\usepackage{booktabs}
\usepackage{multirow}
\usepackage[page]{appendix}

\usepackage[flushleft]{threeparttable}

\usepackage[accepted]{icml2022}

\icmltitlerunning{K-Subspaces Method for Subspace Clustering}

\newcommand{\be}{\begin{equation}}
\newcommand{\ee}{\end{equation}}
\newcommand{\st}{\mbox{s.t.}}
\newcommand{\argmin}{\mathop{\mathrm{arg\,min}}}
\newcommand{\argmax}{\mathop{\mathrm{arg\,max}}}

\newtheorem{lemma}{Lemma}
\newtheorem{thm}{Theorem}
\newtheorem{coro}{Corollary}
\newtheorem{assumption}{Assumption}
\newtheorem{defi}{Definition}
\newtheorem{prop}{Proposition}

\def\bA{\bm{A}}
\def\bB{\bm{B}}
\def\bC{\bm{C}}

\def\bE{\bm{E}}
\def\bG{\bm{G}}
\def\bH{\bm{H}}

\def\bHs{\bm{H}^*}
\def\bI{\bm{I}}

\def\bO{\bm{O}}
\def\bP{\bm{P}}
\def\bQ{\bm{Q}}
\def\bU{\bm{U}}
\def\bV{\bm{V}}

\def\bX{\bm{X}}

\def\b0{\bm{0}}
\def\ba{\bm{a}}
\def\bba{\bar{\ba}}

\def\bc{\bm{c}}

\def\be{\bm{e}}
\def\bg{\bm{g}}
\def\bh{\bm{h}}
\def\bo{\bm{1}}

\def\bu{\bm{u}}
\def\bv{\bm{v}}
\def\tbv{\tilde{\bv}}
\def\bw{\bm{w}}
\def\bx{\bm{x}}

\def\by{\bm{y}}
\def\bz{\bm{z}}
\def\mC{\mathcal{C}}

\def\mM{\mathcal{M}}
\def\mN{\mathcal{N}}
\def\mO{\mathcal{O}}

\def\mT{\mathcal{T}}

\def\bb{\bm{b}}
\def\E{\mathbb{E}}

\def\P{\mathbb{P}}
\def\R{\mathbb{R}}
\def\S{\mathbb{S}}

\begin{document}

\twocolumn[
\icmltitle{Convergence and Recovery Guarantees of the K-Subspaces Method \\ for Subspace Clustering}


\begin{icmlauthorlist}
	\icmlauthor{Peng Wang}{ed}	
	\icmlauthor{Huikang Liu}{goo}
	\icmlauthor{Anthony Man-Cho So}{to}
	\icmlauthor{Laura Balzano}{ed}
\end{icmlauthorlist}

\icmlaffiliation{to}{Department of Systems Engineering and Engineering Management, The Chinese University of Hong Kong, Hong Kong}
\icmlaffiliation{goo}{Research Institute for Interdisciplinary Sciences, School of Information Management and Engineering, Shanghai University of Finance and Economics, Shanghai}
\icmlaffiliation{ed}{Department of Electrical Engineering and Computer Science, University of Michigan, Ann Arbor}

\icmlcorrespondingauthor{Peng Wang}{peng8wang@gmail.com}

\icmlkeywords{Machine Learning, ICML}

\vskip 0.3in
]



\printAffiliationsAndNotice{}  

\begin{abstract}
 
The K-subspaces (KSS) method is a generalization of the K-means method for subspace clustering. In this work, we present local convergence analysis and a recovery guarantee for KSS, assuming data are generated by the semi-random union of subspaces model, where $N$ points are randomly sampled from $K \ge 2$ overlapping subspaces. We show that if the initial assignment of the KSS method lies within a neighborhood of a true clustering, it converges at a superlinear rate and finds the correct clustering within $\Theta(\log\log N)$ iterations with high probability. Moreover, we propose a thresholding inner-product based spectral 
method for initialization and prove that it produces a point in this 
neighborhood. We also present numerical results of the studied method to support our theoretical developments. 
\end{abstract}

\section{Introduction}\label{sec:intro}

Subspace clustering (SC) is a fundamental problem in unsupervised learning, which can be applied to do  dimensionality reduction and data analysis. It has found wide applications in diverse fields, such as computer vision \cite{ho2003clustering,vidal2008multiframe}, gene expression analysis \cite{jiang2004cluster,ucar2011combinatorial}, and image segmentation \cite{hong2006multiscale}, to name a few. In research on SC, the union of subspace (UoS) model, which assumes that data points lie in one of multiple underlying subspaces, is a typical model for studying SC. In particular, substantial advances have been made recently on designing algorithms for solving the SC problem and on establishing theoretical foundations in the UoS model; see, e.g., \citet{vidal2011subspace,vidal2016generalized,meng2018general} and the references therein. 

In the UoS model, the goal of SC is to recover the underlying subspaces and cluster the  unlabeled data points into the corresponding subspaces. To achieve this goal, many algorithms have been proposed in the past two decades, such as sparse subspace clustering methods \cite{elhamifar2013sparse,wang2013noisy}, low-rank representation-based methods \cite{liu2012robust}, thresholding-based methods \cite{heckel2015robust,pmlr-v139-li21f}, and K-subspaces (KSS) method \cite{bradley2000k}. In these methods, the KSS method, which is known as a generalization of the K-means method, can handle clusters in subspaces. In particular, it is conceptually simple and has linear complexity per iteration. This computational benefits render it suitable to handle large-scale datasets in practice. However, a complete theoretical understanding of its convergence behavior and recovery performance is not found in the literature, to the best of our knowledge. This is due in part to its alternating and discrete nature, as well as the fact that as with the K-means, KSS can easily get stuck in bad local minima without a good initialization. Consequently, it remains a major challenge to provide the theoretical foundations for KSS. In this work, we provide guarantees for the convergence behavior and recovery performance of the KSS method. We also develop a simple initialization method with provable guarantees for the KSS method. 
It is worth mentioning that our results improve on state-of-the-art theory with respect to allowable affinity between subspaces, and support the algorithm's competitive performance in our numerical evaluation.

\subsection{Related Works}\label{subsec:rw}

Over the past years, a substantial body of literature explores algorithmic development and theoretical analysis of SC. One of the most well-studied methods is arguably sparse subspace clustering (SSC), which is motivated by representing each data point as a sparse linear combination of the remaining ones. A seminal work by \citet{elhamifar2013sparse} proposed and studied this method. The algorithm proceeds by solving a convex sparse optimization problem, followed by applying spectral clustering to the graph constructed by a solution of this convex problem. In particular, they showed that when the data points are drawn from the \emph{disjoint} subspaces in the noiseless setting, the solution is non-trivial and no edges in the constructed graph connect two points in different subspaces. This is referred to as the \emph{subspace detection property} (SDP) in literature; see, e.g., \citet{wang2013noisy,soltanolkotabi2012geometric,soltanolkotabi2014robust}. We should point out that SDP does not imply \emph{correct clustering} (CC) of data points as mentioned in  \citet{wang2016graph,pmlr-v139-li21f}. Following this line of work,  
theoretical results on the SSC method in various contexts have been established, and  many variants and extensions of the SSC method have been proposed. For example, \citet{soltanolkotabi2012geometric} developed a unified analysis framework of the SSC method, which showed that the SDP holds even when subspaces can be \emph{overlapping} in the noiseless setting. 
Later, \citet{soltanolkotabi2014robust} extended their analysis and results to the noisy setting. Meanwhile, an independent work by \citet{wang2013noisy} also studied the behavior of SSC based on the SDP in the noisy setting. In spite of the solid theoretical guarantees and great empirical performance, SSC suffers from high computational cost. To tackle this issue, \citet{dyer2013greedy} applied an orthogonal matching pursuit (OMP) method to SSC. Then, \citet{tschannen2018noisy} analyzed the performance of this method in the noisy setting and  also introduced and studied the matching pursuit (MP) method for SSC. Recently, more and more variants and extensions for solving SSC have been proposed; see, e.g., \citet{ding2021dual,wang2019provable,wu2020sparse,chen2020stochastic,matsushima2019selective,traganitis2017sketched,you2016scalable}. 

\begin{table}[t]
\caption{Comparison of affinity requirement and recovery results of the surveyed methods in the noiseless semi-random UoS model with \emph{overlapping} subspaces ($K\ge 2$).}\vspace{-0.1in}
\label{table-1}
\begin{center}
\begin{footnotesize}
\begin{tabular}{ccccc}
\toprule
{\bf\scriptsize References} & {\bf\scriptsize Methods} &  {\bf\scriptsize Affinity$^a$}  &  {\bf\scriptsize Results} \\
\midrule
\citet{soltanolkotabi2012geometric}  & SSC & $O(\frac{1}{\sqrt{\log N}})$   & SDP  \\
\citet{wang2013noisy} & SSC &  $O(\frac{1}{\sqrt{\log N}})$  &   SDP \\
\citet{tschannen2018noisy}  & (O)MP & $O(\frac{1}{\sqrt{\log N}})$  & SDP \\
\midrule
\citet{wang2016graph} & SSC & $O(\frac{1}{\sqrt{\log N}})$ & CC \\
\citet{heckel2015robust} & TSC &  $O(\frac{1}{\sqrt{\log N}})$  & CC  \\ 
\citet{park2014greedy} & GSR  &  $O(\frac{1}{\sqrt{\log N}})^b$ & CC \\
\citet{lipor2021subspace} &  EKSS & $O(\frac{1}{\sqrt{\log N}})$ & CC \\ 
\midrule
{\bf Ours} & {\bf KSS} &  $O(1)$ & CC \\
\bottomrule
\end{tabular}
\end{footnotesize}
\end{center}
\footnotesize{$^a$We use the notion \eqref{affi:re} to measure the subspace affinity.}
\footnotesize{$^b$This is obtained by taking $\delta=1/N$ in \citet[Theorem 3.2]{park2014greedy}. }
\vspace{-0.2in}
\end{table}	

As for other methods for SC, \citet{liu2012robust} proposed a low-rank representation (LRR) method by minimizing a nuclear norm regularized problem. In particular, they showed that the proposed method can recover the row space of the data points. Later, \citet{shen2016online} developed an online version of the LLR method, which reduces its computational cost significantly. Another notable approach for SC is thresholding-based methods, which exploit the correlation between data points. 
For example, \citet{heckel2015robust} proposed a thresholding-based subspace clustering (TSC) method, which applies spectral clustering to a weight matrix with entries depending on spherical distances of each data point to its nearest neighbors. They showed that TSC can achieve correct clustering by proving that the formed graph has no false connection and $K$ connected subgraphs. \citet{pmlr-v139-li21f} proposed a thresholding inner-product (TIP) method for SC, which constructs an adjacency matrix by thresholding magnitudes of inner products between data points. In particular, they provided an explicit
bound on the error rate of the TIP method when there are only two subspaces of the same dimension. Moreover, \citet{park2014greedy} proposed a greedy subspace clustering (GSC) method that constructs a neighborhood matrix using a nearest subspace neighbor method and then recovers subspaces by a greedy algorithm. They showed that their approach can guarantee correct clustering. However, they assumed that the dimension of each subspace is known and same and the number of data points in each subspace is also  same. Actually, there are still numerous other popular methods using different techniques for SC, such as matrix factorization-based method \cite{boult1991factorization,pimentel2016group,fan2021large} and principal component analysis type methods \cite{vidal2005generalized,mcwilliams2014subspace}. 


In contrast to the above methods, the KSS method \cite{bradley2000k,agarwal2004k,tseng2000nearest} is essentially a generalization of the \emph{k-means} method for SC, which minimizes the sum of distances of each point to its projection onto the assigned subspace, i.e., 
\begin{align}\label{SC-1}
\min_{\mC,\bU}\  \sum_{k=1}^K \sum_{i \in \mC_k} \|\bz_i - \bU_k\bU_k^T\bz_i\|^2, 
\end{align}
where $\{\bm{z}_i\}_{i=1}^N \subseteq \R^n$ denotes the set of $N$ data points, $\mC = \{\mC_k\}_{k=1}^K$ denotes the set of $K\ge 2$ estimated clusters, and $\bU=\begin{bmatrix}
\bU_1 & \dots & \bU_K
\end{bmatrix}$ with $\bU_k $ being an orthonormal basis of the corresponding cluster. Similar to the k-means method, the KSS method proceeds by alternating between the \emph{subspace update step} and the \emph{cluster assignment step}. As a local search algorithm, it is conceptually simple and has linear complexity as a function of the number of data points, while many popular methods based on self-expression property, such as the surveyed SSC, OMP-based SSC, and LLR, have at least quadratic complexity. This computational advantage renders it more suitable to handle large-scale datasets than these self-expression property based-methods. 
However, due to its non-convex nature, it suffers from sensitivity to initialization and lack of theoretical understanding. To fix the former issue, some heuristics for good initialization have been proposed; see, e.g.,  \citet{he2016robust,zhang2009median}. To improve the performance of the KSS method, \citet{gitlin2018improving} employed a coherence pursuit algorithm. Recently, \citet{lipor2021subspace} applied an ensembles approach to the KSS method with random initialization and showed that it achieves correct clustering based on the argument in \citet{heckel2015robust}. However, their analysis can only tackle one KSS iteration. Generally, it remains open to propose a provable initialization scheme for the KSS method and fully understand its convergence behavior and recovery performance. Due to this, the KSS method has mostly been superseded by convex methods based on self-expression property, which are widely studied and have solid theoretical results. Moreover, establishing theoretical foundations for the KSS method may open the door for the study of various non-convex methods for SC. 

\subsection{Our Contributions}\label{subsec:oc}

In this work, we study the KSS method for SC in the semi-random UoS model. First, we provide theoretical guarantees for the convergence behavior and recovery performance of the KSS method. Specifically, we prove the existence of a \emph{basin of attraction}, whose radius is as large as $O(\sqrt{N})$ around the true clustering of the data points, when the cluster sample sizes are on the same order and the subspace dimensions are also on the same order. 
If the initial assignment of the KSS method lies within this basin, the algorithm is guaranteed to converge to the true clustering at a \emph{superlinear} rate. In particular, once the number of iterations reaches $\Theta(\log\log N)$, the KSS method yields the true clustering with the corresponding orthonormal bases exactly. It is worth emphasizing that these results are obtained under the condition that the normalized affinity between pairwise subspaces is $O(1)$, which is generally milder than those in the existing literature; see the comparison in Table \ref{table-1}. Second,  
we propose a thresholding inner-product based spectral method for initialization of the KSS method. We show that it can generate a point lying in the basin of attraction of the KSS method by deriving its clustering error rate. Our core argument is to derive a spectral bound for a random adjacency matrix without independence structure, which could be of independent interest. In conclusion, our work demystifies the computational efficiency of the KSS method and provides a provable initialization scheme for it, thus bridging the gap between theory and practice. From a broader perspective, our work also contributes to the literature on simple and scalable non-convex methods with provable guarantees; see, e.g., \citet{wang2021optimal,wang2021non,zhang2020symmetry,gao2019iterative,boumal2016nonconvex,ling2022improved}. 

\emph{Notation.} Let $\R^n$ be the $n$-dimensional Euclidean space and $\|\cdot\|$ be the Euclidean norm.
Given a matrix $\bA$, we use $\|\bA\|$ to denote its spectral norm, $\sigma_i(\bA)$ its $i$-th largest singular value, $\|\bA\|_F$ its Frobenius norm, and $a_{ij}$ its $(i,j)$-th element. Given a vector $\ba \in \R^n$, we denote by $a_i$ its $i$-th element. Given a positive integer $n$, we denote by $[n]$ the set $\{1,\dots,n\}$. Given $d_1,\dots,d_K$, let $d_{\min} = \min\{d_k:k\in [K]\}$ and $d_{\max} = \max\{d_k:k\in [K]\}$. Given a discrete set $S$, we denote by $|S|$ its cardinality. Given two sets $A,B \subseteq [n]$, the set difference between $A$ and $B$ denoted by $A\setminus B$ is defined by $A\setminus B=\{x \in A: x\notin B\}$. We use $\mO^{n\times d}$ to denote the set of all $n\times d$ matrices that have orthonormal columns (in particular, $\mO^d$ denotes the set of all $d\times d$ orthogonal matrices) and $\Pi_K$ to denote the set of all $K\times K$ permutation matrices. Let $\pi:[K] \rightarrow [K]$ denote a permutation of the elements in $[K]$. Each $\pi$ corresponds to a $\bQ_{\pi} = \{q_{ij}\}_{1\le i,j\le K} \in \Pi_K$ such that $q_{ij}=1$ if $j=\pi(i)$ and $q_{ij}=0$ otherwise for all $i \in [K]$. The converse also holds. Moreover, for any $\bU,\bV \in \mO^{n\times d}$, we denote by $d(\bU,\bV) =  \left\| \bU\bU^T-\bV\bV^T \right\|$ the distance between the subspaces spanned by $\bU$ and $\bV$. We define $\S^{d-1}=\left\{ \ba \in \R^d: \|\ba\|=1 \right\}$ and denote by $ \mathrm{Unif}(\S^{d-1})$ a uniform distribution over the sphere in $\R^d$. For  non-negative sequences $\{a_k\}$ and $\{b_k\}$, we write $a_k \gtrsim b_k$ if there exists a universal constant $C>0$ such that $a_k \ge Cb_k$ for all $k$.

\section{Preliminaries and Main Results}\label{sec:preli}

In this section, we formally set up the SC problem in the semi-random UoS model, introduce the KSS method for tackling the SC problem, propose an initialization scheme for the KSS method, and give a summary of our main results. 

\subsection{Semi-Random UoS Model}

\begin{defi}\label{memship-matrix}
Suppose that a family of sets $\{\mC_k\}_{k=1}^K$ is a partition of $[N]$. 
We say that $\bH \in \R^{N\times K}$ is a membership matrix if $h_{ik} = 1$ if $i\in \mC_k$ and $h_{ik}=0$ otherwise. For simplicity, we use $\mM^{N \times K}$ to denote the collections of all such $N\times K$ membership matrices.  
\end{defi}
Given an $\bH \in \mM^{N\times K}$, each row of it has exactly one $1$ and $(K-1)$ 0's. 
Besides, $\bH\bQ$ for any $\bQ \in \Pi_K$ represents the same partition as $\bH$ up to a permutation of the cluster labels. We define the distance between two membership matrices $\bH,\bH^\prime \in \mM^{N\times K}$ by
\begin{align*}
d_F(\bH,\bH^\prime) = \min_{\bQ \in \Pi_K} \|\bH-\bH^\prime\bQ\|_F.
\end{align*} 
Then, one can verify that the number of misclassified points in $\bH$ with respect to   $\bH^\prime$ is $d^2_F(\bH,\bH^\prime)/2$. 

\begin{defi}[Semi-Random UoS Model\footnote{This model is called semi-random UoS because the subspaces are arbitrary but data points are randomly generated.}]\label{UoS}
Let $S_k^*$ denote a subspace of $\R^n$ of dimension $d_k$ with $\bU_k^* \in \mO^{n\times d_k}$ being its orthonormal basis for all $k\in [K]$. Let $\bH^* \in \mM^{N\times K}$ represent a partition of $[N]$ into $K$ clusters, each of which is of size $N_k$ for all $k \in [K]$. Then, we say that a collection of $N\ge 2$ points $\{\bz_i\}_{i=1}^N$ is generated according to the semi-random UoS model with parameters $\left(N,K,\{\bU_k^*\}_{k=1}^K,\bHs\right)$ if 
\begin{align}\label{z=Ua}
\bz_i = \bU_{k}^*\ba_i,
\end{align}
where $k \in [K]$ satisfies $h_{ik}^*=1$ and $\ba_i \overset{i.i.d.}{\sim} \mathrm{Unif}(\S^{d_k-1})$ for all $i \in [N]$. 
\end{defi}
Intuitively, given an unknown  partition encoded by $\bHs$, this model generates a collection of unlabeled observations. Given these observations, the goal  of SC is to design an algorithm that finds the true partition, i.e., $\bHs\bQ$ for some $\bQ \in \Pi_K$. We should point out that the subspace dimensions $d_1,\dots,d_K$ are also all unknown. 

\subsection{The KSS Method}


In this subsection, we introduce the KSS method by interpreting it as an application of the alternating minimization method to Problem \eqref{SC-1}. Such an interpretation is similar to that in \citet{bradley2000k}. 
By introducing $\bH \in \mM^{N\times K}$, we can reformulate Problem \eqref{SC-1} as
\begin{align}\label{SC-2}
\min\ &\  \sum_{i=1}^N\sum_{k=1}^K h_{ik} \left(  \|\bz_i\|^2 - \|\bU_k^T\bz_i\|^2 \right)\\
\st\ &\ \bH \in \mM^{N\times K},\ \bU_k \in  \mO^{n \times \hat{d}_k},\ \text{for\ all}\ k \in [K], \notag
\end{align}
where $\hat{d}_k$ for all $k \in [K]$ are candidate subspace dimensions. Observe that this problem is in a form that is amenable to the alternating minimization method (see, e.g., \citet{ghosh2020alternating,hardt2014understanding, zhang2020phase}). Given the current iterate $(\bH^t,\bU_1^t,\dots,\bU_K^t) \in \mM^{N \times K}\times \mO^{n\times \hat{d}_1} \times \dots \times \mO^{n\times \hat{d}_K}  $, the method generates the next iterate via
\begin{align}
\bU_k^{t+1} & \in \argmin_{\bU_k \in \mO^{n\times \hat{d}_k}} \sum_{i=1}^N h_{ik}^t \left(  \|\bz_i\|^2 - \|\bU_k^T\bz_i\|^2 \right) \label{update:Uk}
\end{align}
for all $k \in [K]$ and
\begin{align}\label{update:H}
\bH^{t+1} & \in \mT\left(\bG_{\bH}(\bU^{t+1})\right),
\end{align}
where the $(i,k)$-th element of $\bG_{\bH}(\bU) \in \R^{N\times K}$ is $\|\bz_i\|^2-\|\bU_k^T\bz_i\|^2$ and $\mT$ denotes the operator that for any $\bG \in \R^{N \times K}$,
\begin{align}\label{T-H}
\mT(\bG) = \argmin \left\{ \langle \bG, \bH \rangle: \bH \in \mM^{N\times K} \right\}. 
\end{align}
It is worth noting that the updates \eqref{update:Uk} and \eqref{update:H} both admit closed-form solutions, which respectively correspond to the \emph{subspace update step} and the \emph{cluster assignment step} of the KSS method. Indeed, the update \eqref{update:Uk} is typically a PCA problem and its solution is given by
\begin{align}\label{sol:Uk}
\bU_k^{t+1}  = \mathrm{PCA}\left(\sum_{i=1}^Nh_{ik}^t\bz_i\bz_i^T, \hat{d}_k\right),
\end{align}
 where $\mathrm{PCA}(\bA,d):\S^n \times \R \rightarrow \R^{n\times d}$ is the operator that computes the eigenvectors associated with the $d$ leading eigenvalues of $\bA$. 
Moreover, the update \eqref{update:H} is a special assignment problem, whose solution is given by 
\begin{align}\label{sol:H}
h_{ik}^{t+1} = \begin{cases}
1, & \text{if}\ k = I_i, \\
0, & \text{otherwise}, 
\end{cases}
\end{align}
where $I_i \in [K]$ satisfies $\|\bU^{{t+1}^T}_{I_i}\bz_i\| \ge \|\bU^{{t+1}^T}_k\bz_i\|$ for all $k \neq I_i$; see Lemma \ref{lem:solution-T-h}.      

A natural question arising in the update \eqref{sol:Uk} is how to choose $\hat{d}_k$ for all $k \in [K]$. Generally, the KSS method assumes that the subspace dimensions $d_1,\dots,d_K$ are known beforehand \cite{vidal2011subspace}, which is not practical in applications. Even if $d_1,\dots,d_K$ are known but unequal, it is still unknown how to find an one-to-one mapping between $\{d_k\}_{k=1}^K$ and $\{\hat{d}_k\}_{k=1}^K$ due to the fact that the permutation of clusters is unknown. To fix this issue, we propose an adaptive strategy to choose $\hat{d}_k$ for all $k \in [K]$. Specifically, let $\lambda_{k1}^t \ge \dots \ge \lambda_{kd}^t$ be the $d$ leading eigenvalues of $\sum_{i=1}^Nh_{ik}^t\bz_i\bz_i^T$ for all $k \in [K]$, where $d$ is an input parameter satisfying $d > d_{\max}$. Then for all $k \in [K]$, we set
\begin{align}\label{esti:dk}
\hat{d}^{t+1}_k  = \argmax_{i\in [d-1]} \left(\lambda_{ki}^t - \lambda_{k(i+1)}^t\right)
\end{align}
and replace \eqref{sol:Uk} by
\begin{align}\label{sol:Uk-1}
\bU_k^{t+1}  = \mathrm{PCA}\left(\sum_{i=1}^Nh_{ik}^t\bz_i\bz_i^T, \hat{d}^{t+1}_k\right).
\end{align}

\subsection{Initialization Method}
A key ingredient in our approach is to identify a proper initial point $\bH^0$ that may guarantee rapid convergence of the KSS method for solving Problem \eqref{SC-2}. Motivated by the thresholding inner-product based scheme in \citet{pmlr-v139-li21f}, we propose a thresholding  inner-product based spectral method (TIPS) for initialization. Specifically, given a thresholding parameter $\tau > 0$, a graph $G$ with adjacency matrix $\bA \in \R^{N\times N}$ is generated by 
\begin{align}\label{step:thresholding}
a_{ij}  = \begin{cases}
1, &\ \text{if}\ |\langle \bz_i,\bz_j \rangle| \ge \tau\ \text{and}\ i\neq j, \\
0, &\ \text{otherwise},
\end{cases}
\end{align}
for all $1 \le i \le j \le N$. Then, the initial cluster assignment $\bH^0$ is obtained by applying the \emph{k-means} to the matrix $\bV$ formed by the eigenvectors associated with the $K$ leading eigenvalues of $\bA$. Although it is NP-hard in the worst case to compute a global minimizer of the k-means problem (see, e.g., \citet{aloise2009np}), some polynomial-time algorithms have been proposed for finding an approximate solution whose value is within a constant fraction of the optimal value (see, e.g., \citet{kumar2010clustering}), i.e., 
\begin{align}\label{k-means}
&\ (\bH^0,\hat{\bm{X}}) \in \mM^{N\times K} \times \R^{K\times K}\quad \st\ \ \|\bm{H}^0\hat{\bm{X}}-\bV\|_F^2  \notag\\ 
&\quad\ \le (1+\theta) \min_{(\bm{H},\bm{X}) \in \mM^{N \times K} \times \R^{K\times K}}\|\bm{H}\bm{X}-\bV\|_F^2,
\end{align} 
where $\theta > 0$ is a constant. We assume that we can find such an approximate solution. We now summarize the proposed method in Algorithm \ref{alg-1}. 
\begin{algorithm}[!hbtp]
	\caption{The TIPS initialized KSS method}   
	\begin{algorithmic}[1]  
		\STATE \textbf{Input}: samples $\{\bz_i\}_{i=1}^N$, $\tau > 0$,  $\theta > 0$, $d,T,K\in \mathbb{Z}_+ $ \\
		\texttt{/*\quad The TIPS initialization \quad */} 
		\STATE construct an adjacency matrix $\bA \in \R^{N\times N}$ by \eqref{step:thresholding} 
		\STATE calculate $\bV \in \R^{N\times K}$ formed by the eigenvectors associated with the $K$ leading  eigenvalues of $\bA$
		\STATE let $(\bm{H}^0, \hat{\bm{X}}) $ be a $(1+\theta)$-approximate solution to the k-means problem \eqref{k-means} with $K$ clusters and input  $\bV$\\	
		\texttt{/*\quad The KSS method  \quad */} 
		\FOR{$t=0,1,\dots,T$} 
		\STATE \texttt{/*\quad subspace update step  \quad */} 
		\FOR{$k=1,\dots,K$}
	    \STATE Compute $\hat{d}_k^{t+1}$ via \eqref{esti:dk} and $\bU_k^{t+1}$ via \eqref{sol:Uk-1} 
 		\ENDFOR 
		\STATE \texttt{/*\quad cluster assignment step  \quad */} 
		\STATE compute $\bH^{t+1}$ via \eqref{sol:H}  
		\ENDFOR
%
	\end{algorithmic}
	\label{alg-1}
\end{algorithm}

\subsection{Main Theorems}

Before we proceed, we introduce a definition to capture notions of affinity between pairwise subspaces and impose an assumption on the affinity. 
\begin{defi}\label{def:affi}
The affinity between two subspaces $S_k$ and $S_\ell$ is defined by
\begin{align}\label{affi:two}
\mathrm{aff}(S_k,S_\ell) = \sqrt{\sum_{i=1}^{\min\{d_k,d_\ell\}} \left(\sigma_{k\ell}^{(i)}\right)^2},
\end{align}
where $\sigma_{k\ell}^{(1)} \ge \dots \ge \sigma_{k\ell}^{(\min\{d_k,d_\ell\})} \ge 0$ are the singular vaules of $\bU_k^{T}\bU_\ell \in \R^{d_k \times d_\ell}$ with $\bU_k,\bU_\ell$ being respectively orthonormal bases of $S_k$ and $S_\ell$. The normalized affinity between two subspaces $S_k$ and $S_\ell$ is defined by
\begin{align}\label{affi:ave}
\overline{\mathrm{aff}}(S_k,S_\ell) = \frac{\mathrm{aff}(S_k,S_\ell)}{\min\{\sqrt{d_k},\sqrt{d_\ell}\}}.
\end{align}
\end{defi}
 For ease of exposition, we define the maximum of the normalized affinities as
\begin{align}\label{affi:re}
\kappa = \max_{1 \le k \neq \ell \le K} \overline{\mathrm{aff}}(S_k^*,S_\ell^*)
\end{align}
and define
\begin{align}\label{ratio:d-N}
\kappa_d = \frac{d_{\max}}{d_{\min}},\ \kappa_N = \frac{N_{\max}}{N_{\min}}. 
\end{align}
\begin{assumption}\label{AS:1}
The affinity between pairwise subspaces in the UoS model satisfies $\kappa \in (0, 1/2]$. 
\end{assumption}
We remark that this affinity condition is milder than those in the related literature. Because this assumption allows  the affinity $\mathrm{aff}(S_k,S_{\ell})$ to be as large as $\min\{\sqrt{d_k},\sqrt{d_\ell}\}/2$, while those in the literature require $\mathrm{aff}(S_k,S_{\ell}) \lesssim \min\{\sqrt{d_k},\sqrt{d_\ell}\}/\sqrt{\log N}$ for all $1\le k \neq \ell \le K$. Please see the comparison in Table \ref{table-1}. We next present a main theorem of this work, which shows that the KSS method converges superlinearly and achieves the correct clustering under Assumption \ref{AS:1}.

\begin{thm}\label{thm:KSS}
Let $\{\bz_i\}_{i=1}^N$ be data points generated according to the semi-random UoS model with parameters $(N,K,\{\bU_k^*\}_{k=1}^K,\bHs)$. Suppose that Assumption \ref{AS:1} holds, $N_{\min} \gtrsim  d_k \gtrsim \log N$ for all $k \in [K]$, and the initial point $\bH^0 \in \mM^{N\times K}$ satisfies 
\begin{align}\label{H:thm:init}
 d_F(\bH^0,\bHs) \le \frac{(1-\kappa)N_{\min}}{5\kappa_d\sqrt{N}}.
\end{align}
Set $T=\Theta\left(\log\log N \right)$ and $d \in \mathbb{Z}_+$ satisfying $d>d_{\max}$ in Algorithm \ref{alg-1}. Then, the following statements hold with probability at least $1-N^{-\Omega(1)}$: \\
(i) For all $t \in [T]$, it holds that 
\begin{align}\label{H:rate}
d_F(\bH^{t},\bHs) \le \kappa_1^{2^{t}-1} d_F\left(\bH^{0}, \bH^*\right), 
\end{align}
where $\kappa_1 \in (0,1)$ is an absolute constant. \\
(ii) It holds for a permutation $\pi :[K]\rightarrow [K]$ that 
\begin{align}\label{HT:thm}
\bH^{T} =  \bHs\bQ_{\pi}
\end{align}
and $\hat{d}^{T+1}_{\pi(k)} = d_k$ for all $k \in [K]$,
\begin{align}\label{UT:thm}
\bU^{T+1}_{\pi (k)}  = \bU_k^*\bO_k,\ \bO_k \in \mO^{d_k}\ \text{for all}\ k \in[K].
\end{align}
\end{thm}
Before we proceed, some remarks are in order. First, while Problem \eqref{SC-1} is NP-hard in the worst case \cite{gitlin2018improving}, the assumption that the data points are generated by the semi-random UoS model allows us to conduct an average-case analysis of the KSS method.
Second, a neighborhood of size $O\left(\frac{N_{\min}}{\kappa_d\sqrt{N}}\right)$ around each true cluster  forms a \emph{basin of attraction} in the UoS model, in which the KSS method converges superlinearly. In particular, if $\kappa_d,\kappa_N$ are both constant, we can see that the size of this basin is $O(\sqrt{N})$, which is rather large. Provided that the initial point $\bH^0$ lies within this basin, the subsequent iterates are guaranteed to converge to ground truth at a \emph{superlinear} rate. Third, if the number of iterations reaches $\Theta(\log\log N)$, the KSS method can not only find correct clustering, but also exactly recovers the orthonormal basis of each subspace. This demonstrates the efficacy of the KSS method. 
Finally, any method that can return a point satisfying \eqref{H:thm:init} is qualified as an initialization scheme for the KSS method. In this work, we design a simple initialization
scheme in the first stage of Algorithm \ref{alg-1} that can provably generate a point in the basin of attraction under the following assumption. 
Before we proceed, let $\bB \in \R^{K\times K}$ be a symmetric matrix whose elements are given by
\begin{align}\label{bkl}
b_{k\ell} =  2 - 2\Phi\left(\frac{\tau\sqrt{d_kd_\ell}}{\mathrm{aff}(S_k^*,S_\ell^*)}\right),\ \forall\ 1\le k, \ell \le K, 
\end{align}
where $\Phi(x) = \frac{1}{\sqrt{2\pi}} \int_{-\infty}^x \exp\left(-\frac{t^2}{2}\right) dt$ denotes the cumulative distribution function of the standard normal distribution. It is worth noting that $b_{k\ell}$ is an approximation of the probability of $a_{ij}=1$ if $\bz_i \in S_k^*$ and $\bz_j \in S_\ell^*$ for all $1\le k, \ell \le K$, where $a_{ij}$ is given in \eqref{step:thresholding}; see Lemma \ref{lem:p-q}.  
\begin{assumption}\label{AS:2}
The thresholding parameter is set as 
\begin{align}\label{tau}
\tau =  \frac{\sqrt{c}}{\sqrt{d_{\max}}},
\end{align}
where $c>0$ is a constant. The parameter $\kappa_d$ is a constant and the maximum of the normalized affinities satisfying
\begin{align}\label{kappa:AS2}
\kappa \in \left(0, \frac{\sqrt{c}}{\sqrt{\kappa_d}\Phi^{-1}\left( 1-\frac{1-\Phi(\sqrt{c})}{2(K-1)} \right)} \right)
\end{align}
is also a constant. Moreover, the affinity between pairwise subspaces satisfies
\begin{align}\label{affi:AS2}
\mathrm{aff}(S_k^*,S_\ell^*) \gtrsim \log N,\ \forall\ 1 \le k \neq \ell \le K
\end{align}
and the subspace dimension satisfies
\begin{align}\label{subdim:AS2}
d_{\min} \gtrsim \log^3N \;.
\end{align}
\end{assumption}
We will use this assumption in the following theorem, restricting our result to the high affinity case. In general, the clustering becomes harder as the affinity increases; see, e.g., \citet[Section 1.3.1]{soltanolkotabi2014robust}. Then, it is natural to assume that $\kappa$ is a constant and \eqref{affi:AS2} holds. We want to also highlight that \eqref{affi:AS2} implies that our subspaces are of generally moderate dimension, which is made precise in \eqref{subdim:AS2} of the assumption. While this is slightly restrictive, it is in line with theoretical results in other subspace clustering literature, and it also simplifies our theoretical analysis. We leave an analysis of the low-to-moderate affinity settings and low-rank subspaces to future work.
\begin{thm}\label{thm:init}
Let $\{\bz_i\}_{i=1}^N$ be data points generated according to the semi-random UoS model with parameters $(N,K,\{\bU_k^*\}_{k=1}^K,\bHs)$. Suppose that Assumption \ref{AS:2} holds, $\kappa_d \le \sqrt{\log N}$, and $\kappa_d\kappa_N^2 \lesssim {\sqrt{\log N}}$.
It holds with probability at least $1-N^{-\Omega(1)}$ that
\begin{align}\label{eq:H0}
d_F(\bH^0,\bHs) \lesssim \sqrt{\kappa_d}\kappa_N\frac{\sqrt{N_{\max}}}{\sqrt[4]{\log N}}. 
\end{align}
In particular, if both $\kappa_d$ and $\kappa_N$ are constants and $N$ is sufficiently large, $\bH^0$ satisfies \eqref{H:thm:init}  with probability at least $1-N^{-\Omega(1)}$.  
\end{thm}
To put the above results in perspective, we make some remarks. First, according to the fact that $d_F^2(\bH^0,\bHs)/2$ denotes the number of misclassified data points, the bound \eqref{eq:H0} implies that the TIPS method only misclassifies $O(N/\sqrt{\log N})$ points when $\kappa_d,\kappa_N$ are constants and the normalized subspace affinity is $O(1)$. This automatically satisfies \eqref{H:thm:init}, which requires the number of misclassified points to be $O(N)$ when $\kappa_d,\kappa_N$ are constants. Second, we believe that the recovery error bound \eqref{eq:H0} can be improved by enhancing the spectral bound in Proposition \ref{prop:spec-gap}. This is left for future research.

\section{Proofs of Main Results}\label{sec:pf-main}

In this section, we sketch the proofs of the theorems in Section \ref{sec:preli}. The complete proofs can be found in Sections \ref{appenSec:pf-init}, \ref{appenSec:pfUpdate}, and \ref{appenSec:pfAssign} of the appendix. 

\subsection{Analysis of Initialization Method}\label{subsec:init}

In this subsection, our goal is to establish a recovery error bound of the TIPS method.
To begin, we estimate the connection probability of data points (i.e., the probability of $a_{ij}=1$ in \eqref{step:thresholding}) according to their memberships after the thresholding procedure \eqref{step:thresholding}. Moreover, we show that the connection probability of data points in the same subspace is larger than that of data points in different subspaces. 
\begin{lemma}\label{lem:p-q}
Consider the setting in Theorem \ref{thm:init}.  Let $p_{k\ell} \in \R$ denote the connection probability between any pair of data points that respectively belong to the subspaces $S_k^*$ and $S_\ell^*$ for all $1 \le k, \ell \le K$. 
Then, it holds for all $1 \le k, \ell \le K$ that
\begin{align}\label{pkl-bkl}
\left| p_{k\ell} - b_{k\ell} \right| \lesssim \kappa_d/{\sqrt{\log N}},
\end{align}
where $b_{k\ell}$ is defined in \eqref{bkl}, and
\begin{align}\label{pq:gap}
p_{kk} - p_{k\ell} \gtrsim 1/{\sqrt{\kappa_d}}.
\end{align}
\end{lemma}

Under Assumption \ref{AS:2}, we can show that the approximate connection matrix $\bB$ is non-degenerate, which is crucial for the analysis of the k-means error bound. 
\begin{lemma}\label{lem:sing-B}
Consider the setting in Theorem \ref{thm:init}. The matrix $\bB$ defined in \eqref{bkl} is of full rank and its smallest singular value $\gamma$ satisfies $\gamma \ge 1 - \Phi(\sqrt{c})$, where $c$ is the constant in Assumption \ref{AS:2}.   
\end{lemma}

Next, we present a spectral bound on the deviation of $\bA$ from its mean.    
\begin{prop}\label{prop:spec-gap}
Consider the setting in Theorem \ref{thm:init}. Then, it holds with probability at least  $1-6K^2N^{-1}$  that
\begin{align}\label{specgap}
\|\bA - \E[\bA]\| \lesssim \frac{\sqrt{\kappa_d} N}{\sqrt[4]{\log N}}.
\end{align}
\end{prop}
Despite that this bound seems large, it is sufficient for proving \eqref{eq:H0}. A key observation is that the entries in the $i$-th column of $\bA$ are independent conditioned on $\bz_i$, while they are dependent. This plays an important role in our analysis. Compared to the results in \citet[Theorem 5.2]{lei2015consistency}, this lemma provides a spectral bound for an adjacency matrix without independence structure, which could be of independent interest. 

Equipped with Proposition \ref{prop:spec-gap}, Assumption \ref{AS:2}, Lemma \ref{lem:sing-B}, and \citet[Lemmas 5.1, 5.3]{lei2015consistency}, we can prove Theorem \ref{thm:init}. The complete proof is provided in Section \ref{appenSubsec:pfThmInit} of the appendix. 

\subsection{Analysis of Subspace Update Step}\label{subsec:update}

In this subsection, we analyze convergence behavior of the subspace update step in the KSS iterations.
For ease of exposition, let us introduce some further notation. Given an $\bH \in \mM^{N\times K}$, let $\mC_k=\{i\in [N]:h_{ik} = 1\}$ and $n_k=|\mC_k|$ for all $k \in [K]$. Given $\mC_1,\dots,\mC_K$, let
\begin{align}\label{def:Bkl}
n_{k\ell}=|\mC_k \cap \mC_\ell^*|,\quad \bm{\Psi}_{k\ell} = \frac{1}{n_{k\ell}} \sum_{i\in \mC_k \cap \mC_\ell^*} \ba_i \ba_i^T 
\end{align}
for all $k,\ell \in [K]$, where $\ba_i$ for all $i\in [N]$ are given in the UoS model. Given a permutation $\pi:[K] \rightarrow [K]$ and a partition  $\{\mC_1,\dots,\mC_K\}$ of $[N]$ represented by $\bH \in \mM^{N\times K}$, we define the maximum of the number of misclassified points in $\mC_{k}$ w.r.t. $\mC_{\pi^{-1}(k)}^*$ and that in $\mC^*_{k}$ w.r.t. $\mC_{\pi(k)}$ as
\begin{align}\label{def:Gk}
W_k(\bH) = \max\left\{ |\mC_{k} \setminus \mC_{\pi^{-1}(k)}^* |, |\mC_{k}^*  \setminus \mC_{\pi(k)} | \right\}.
\end{align}
To begin, we present a lemma that estimates the singular values of $\bm{\Psi}_{k\ell}$ for all $1 \le k \neq \ell \le K$. 

\begin{lemma}\label{lem:spec-B-kl}
Suppose that $\pi:[K]\rightarrow [K]$ is a permutation, $N_k \gtrsim  d_k \gtrsim \log N$ for all $k\in [K]$, and $\bH \in \mM^{N\times K}$ satisfies
\begin{align}\label{n-1}
W_k(\bH) \le \frac{1}{8}N_{\min}\ \text{for all}\ k \in [K].
\end{align}
It holds with probability at least $1-2K /({d_{\min}^2 N})$ for all $1 \le k \neq \ell \le K$ that 
\begin{align}
& \left|\sigma_i\left(\bm{\Psi}_{\pi(k)k}\right) - \frac{1}{ d_k} \right| \le \frac{1}{32d_k}\ \ \text{for all}\ i \in [d_k], \label{rst1:lem-spec-B-kl}\\
& \sigma_1\left(\bm{\Psi}_{\pi(k)\ell}\right) \le  \frac{1}{d_\ell} + \frac{5c_1}{4d_\ell}\left( \sqrt{\frac{d_\ell}{n_{\pi(k)\ell}}} + \frac{d_\ell}{n_{\pi(k)\ell}}\right),\label{rst2:lem-spec-B-kl}
\end{align}
where $c_1 > 0$ is an absolute constant. 
\end{lemma}

Armed with this lemma, we are now ready to show that the distance from the subspaces generated by the update steps to the true ones can be bounded by the number of misclassfied data points.   

\begin{lemma}\label{lem:contra-U-less}  
Let $\bG_{\bU_k}(\bH)=\sum_{i=1}^N h_{ik} \bz_i\bz_i^T$ for some $\bH \in \mM^{N\times K}$ and $\lambda_{k1} \ge \dots \ge \lambda_{kd}$ be the $d$ leading eigenvalues of $\bG_{\bU_k}(\bH)$ for all $k \in [K]$. Suppose that for all $k \in [K]$,
\begin{align}\label{dk}
\hat{d}_k  = \argmax_{i\in [d-1]} \left(\lambda_{ki}  - \lambda_{k(i+1)} \right)
\end{align}
and
\begin{align}\label{Uk}
\bU_{k} = \mathrm{PCA}(\bG_{\bU_k}(\bH), \hat{d}_k). 
\end{align}
Suppose in addition that $\pi:[K]\rightarrow [K]$ is a permutation, $N_{\min} \gtrsim  d_k \gtrsim \log N$ for all $k \in [K]$, and $\varepsilon \in 	\left(0, 1/({8\kappa_d}) \right]$ is a constant such that
\begin{align}\label{init:U-noiseless}
W_k(\bH) \le \varepsilon N_{\min}\ \text{for all}\ k \in [K].
\end{align}
Then, it holds with probability at least $1-2K/(d_{\min}^2N)$ that 
\begin{align}\label{rst:dk}
\hat{d}_{\pi(k)} = d_k\ \text{for all}\ k \in [K],
\end{align}
\begin{align}\label{rst:lem-contra-U}
\sum_{k=1}^K d(\bU_{\pi(k)},\bU_{k}^*)  \le  \frac{2d_{\max}}{N_{\min}} & \max\left\{\frac{1}{d_{\min}} \|\bH-\bHs\bQ_{\pi}\|_F^2, \notag \right. \\  & \quad \left.  2K^2(c_1+1)c_1 \right\},
\end{align}
where $c_1$ is the constant in Lemma \ref{lem:spec-B-kl}. 
\end{lemma}

\subsection{Analysis of Cluster Assignment Step}\label{subsec:assign}

\begin{figure*}[t]
\begin{center}
	\begin{minipage}[b]{0.33\linewidth}
		\centering
		\centerline{\includegraphics[width=\linewidth]{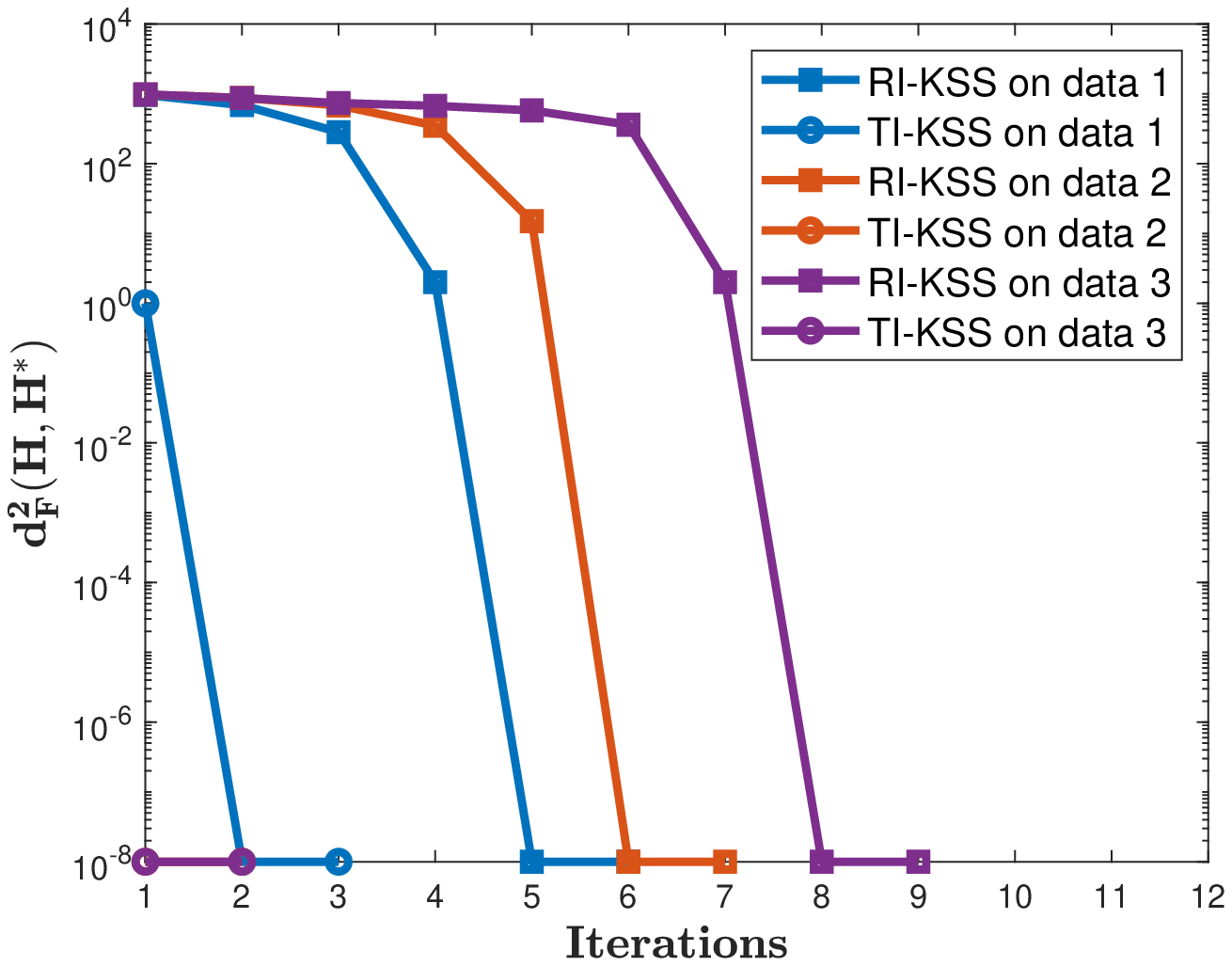}}
		\centerline{(a) {\footnotesize $K=3$}}\medskip
	\end{minipage}
	\begin{minipage}[b]{0.33\linewidth}
		\centering
		\centerline{\includegraphics[width=\linewidth]{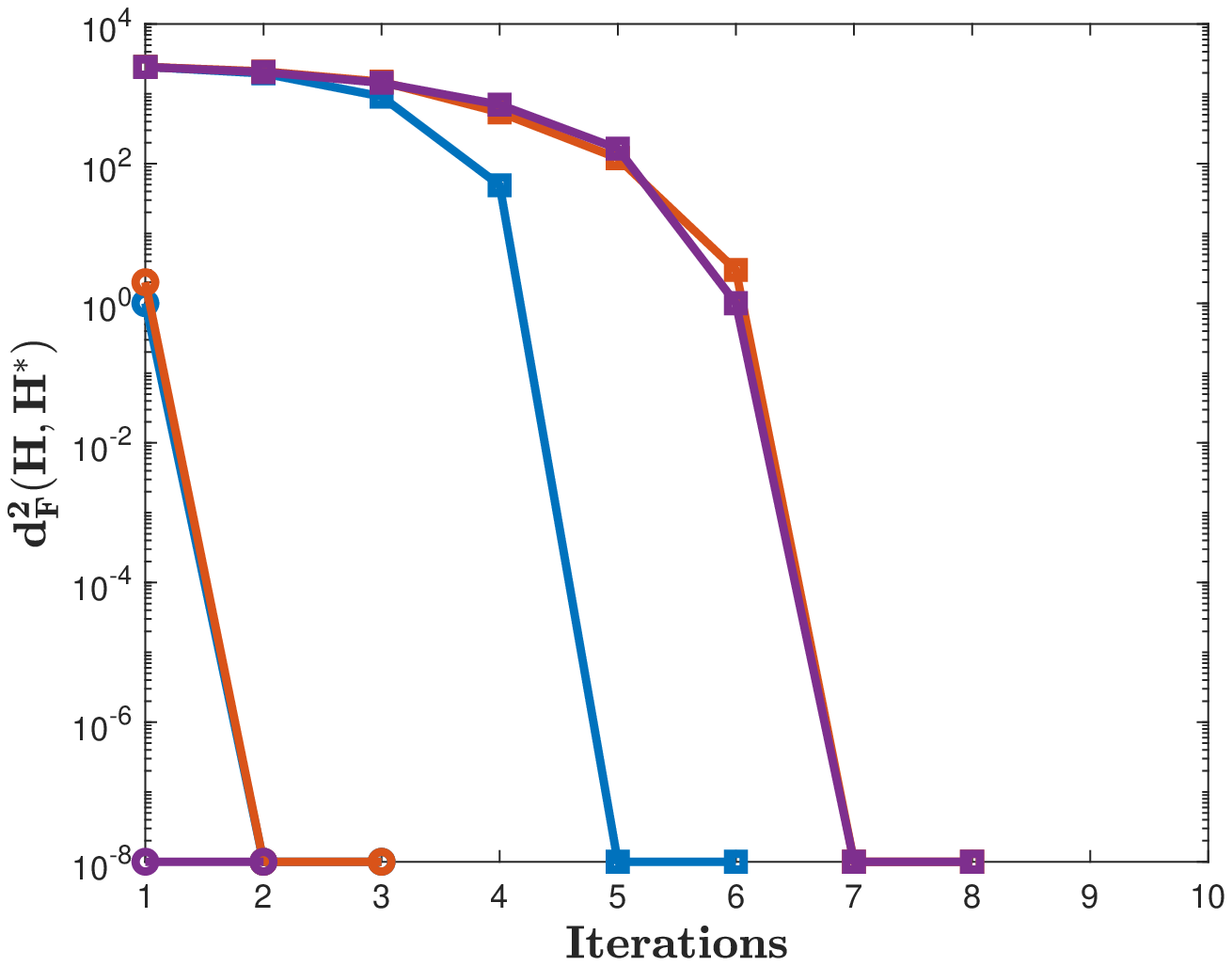}}
		\centerline{(b) {\footnotesize $K=6$}} \medskip
	\end{minipage}
	\begin{minipage}[b]{0.33\linewidth}
		\centering
		\centerline{\includegraphics[width=\linewidth]{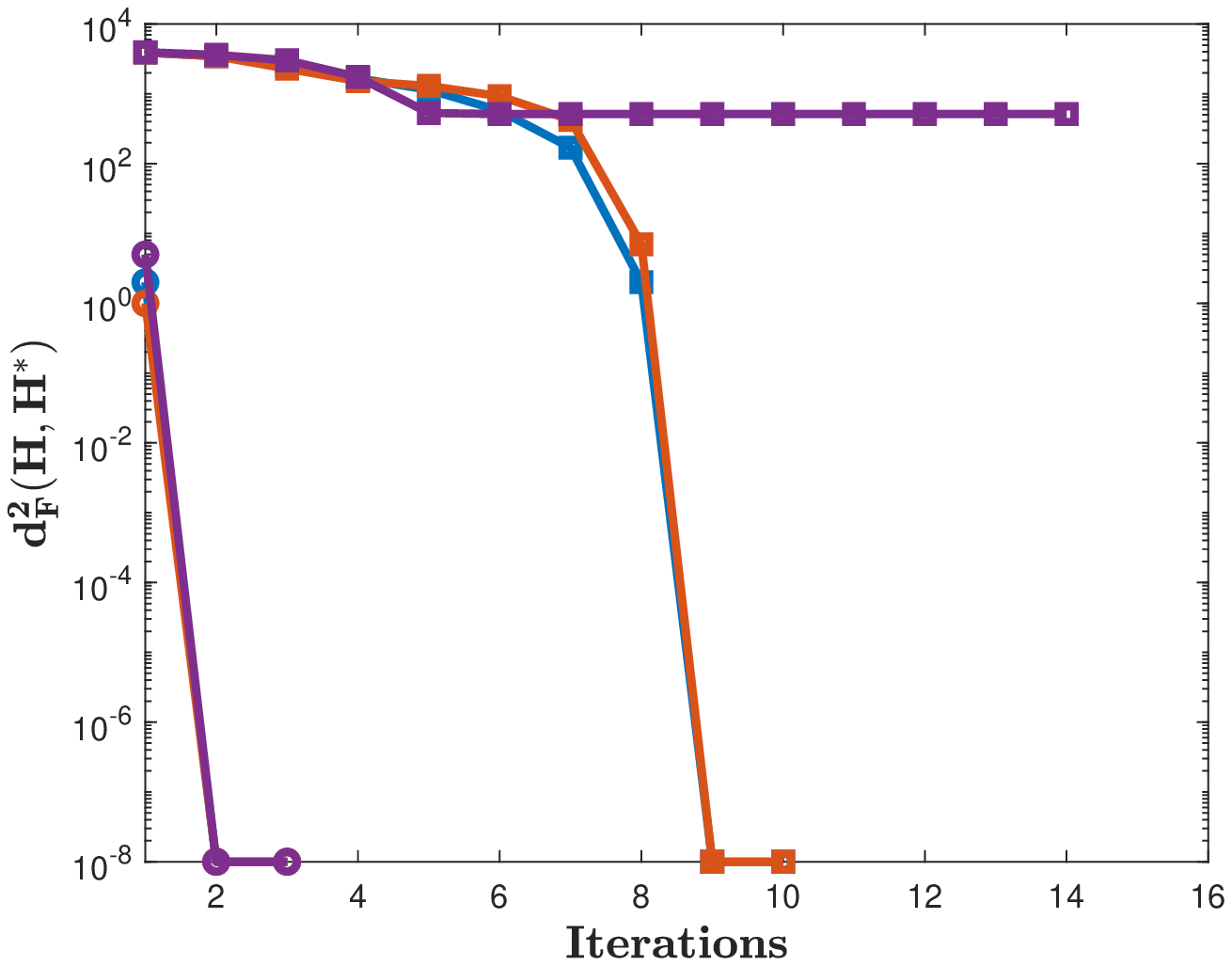}}
		\centerline{(c) {\footnotesize $K=9$}} \medskip 
	\end{minipage}
\end{center}
	\vskip -0.25in
	\caption{Convergence performance of KSS: The $x$-axis is number of iterations and the $y$-axis is the distance from an iterate to a ground truth, i.e., $d_F^2(\bH^k,\bHs)+10^{-8}$, where $\bH^k$ is the $k$-th iterate generated by KSS.}
	\label{fig:1}
\end{figure*}

In this subsection, we turn to study convergence behavior of the cluster assignment step in the KSS iterations. 
Observe that Problem \eqref{T-H} is row-separable, and thus we can solve it by dividing it into $N$ subproblems. Specifically, for a row of $\bG$ denoted by $\bg \in \R^K$, it suffices to consider
\begin{align*}
\mT(\bg) = \argmin \left\{ \langle \bg,\bh \rangle: \bh^T\bo_K = 1,\ \bh \in \{0,1\}^K \right\}.
\end{align*}Then, we can show that this problem admits a closed-form solution, which may be not unique. 
\begin{lemma}\label{lem:solution-T-h}
For any $\bg \in \R^K$, it holds that $\bv \in \mT(\bg)$ if and only if $v_k = 1$ and $v_\ell = 0$ for all $\ell \neq k$, where $k\in [K]$ satisfies $g_k \le g_\ell$ for all $\ell \neq k$. Moreover, $\bv \in \mT(\bg)$ if and only if $\bQ\bv \in \mT(\bQ\bg)$ for $\bQ \in \Pi_K$. 
\end{lemma} 
Based on the above result, we can prove that the operator $\mT$ possesses a Lipschitz-like property. 
\begin{lemma}\label{lem:contra-T-h-less}
Suppose that $\bg \in \R^K$ is arbitrary and $\delta >0$ is a constant such that $g_\ell - g_k \ge  \delta $ for some $k\in [K]$ and all $\ell \neq k$. Then, for any $\bv \in \mT(\bg)$, $\bg^\prime \in \R^K$, and $\bv^\prime \in \mT(\bg^\prime)$,  it holds that
\begin{align}\label{rst:contra-T-h}
\|\bv - \bv^\prime\| \le \frac{2\|\bg-\bg^\prime\|}{\delta}. 
\end{align}
\end{lemma}
We are now ready to show that the number of misclassified points is bounded by the subspace distance.
\begin{lemma}\label{lem-less:contra-H}
Let $\pi:[K]\rightarrow [K]$ be a permutation such that $\bU = \left(\bU_1,\dots,\bU_K\right)$ with $\bU_{\pi(k)} \in \mO^{n\times d_k}$ for all $k\in [K]$. Suppose that $\bar{\bH} \in \mT\left(\bG_{\bH}(\bU)\right)$, where the $(i,k)$-th element of $\bG_{\bH}(\bU) \in \R^{N\times K}$ is $\|\bz_i\|^2-\|\bU_k^T\bz_i\|^2$. Then, it holds for all $i\in [N]$ that
\begin{align}\label{rst-1:lem-less-contra-H}
\|\bar{\bh}_i - \bh_i^*\bQ_{\pi}\|   \le   \frac{ 2 \sqrt{\sum_{k=1}^K d(\bU_{\pi(k)},\bU_{k}^*)} }{1 - \max_{\ell \neq I_i} \|\bU_\ell^{*^T}\bz_i\|^2},
\end{align}
where the row vectors $\bar{\bh}_i, \bh_i^* \in \R^K$ respectively denote the $i$-th row of $\bar{\bH}$ and $\bHs$, and $I_i \in [K]$ satisfies $h_{iI_i}^*=1$ for all $i \in [N]$.
\end{lemma}

The following lemma indicates that the KSS iterations directly converge to ground truth once the distance from the current iterate to ground truth is small enough. This implies finite termination of the KKS method. 
\begin{lemma}\label{lem:one-step}
Suppose that Assumption \ref{AS:1} holds,  $N_{\min} \gtrsim  d_k \gtrsim \log N$ for all $k \in [K]$, and $\bH^t \in \mM^{N\times K}$ satisfies
\begin{align}\label{init:lem:one-step}
d^2_F(\bH^t,\bHs) \le 2K^2(c_1+1)c_1d_{\min},
\end{align}
where $c_1$ is the constant in Lemma \ref{lem:spec-B-kl}. Then, it holds with probability at least $1-2K/(d_{\min}^2N)-5K^2/N$ that   
\begin{align*}
\bH^{t+1} = \bH^*\bQ_\pi
\end{align*}
for some $\bQ_\pi \in \Pi_K$.  
\end{lemma}

Equipped with the results in Sections \ref{subsec:update} and \ref{subsec:assign}, we can prove Theorem \ref{thm:KSS}. 
The complete proof can be found in Section \ref{appenSubsec:pfThmKSS} of the appendix.

\section{Experiment Results}

In this section, we report the convergence behavior, recovery performance, and numerical efficiency of the KSS method for SC on both synthetic and real datasets. 
All of our experiments are implemented in MATLAB R2020a on the Great Lakes HPC Cluster of the University of Michigan with 180GB memory and 16 cores. Our code is available at \href{https://github.com/peng8wang/ICML2022-K-Subspaces}{https://github.com/peng8wang/ICML2022-K-Subspaces}. 

\subsection{Convergence Behavior and Recovery Performance}

We first conduct 3 sets of numerical tests, which correspond to $K \in \{3,6,9\}$, to examine the convergence behavior and recovery performance of the KSS method  in the semi-random UoS model (see Definition \ref{UoS}). We generate $K$ overlapping subspaces as follows. First, we set $n=300$, $\overline{d}=30$, $\underline{d}=25$, and uniformly at random select $d_k \in [\underline{d},\overline{d}]$ for all $k \in [K]$. Second, we arbitrarily generate an orthogonal matrix $\bU= \left[\bu_1,\dots,\bu_n\right] \in \mO^n$  and set the shared basis as $\bar{\bU}=\left[\bu_{n-s+1},\dots,\bu_n\right]$ for an integer $s \in [0,\underline{d}]$. Next, we generate $\bV_k$ by randomly picking up $(d_k-s)$ columns, which are not repeated, from the first $n-s$ columns of $\bU$. Finally, we form $\bU_k^*=\left[\bV_k\ \bar{\bU} \right]$ for all $k \in [K]$, which ensures that the intersection between $S_k$ and $S_\ell$ is at least of dimension $s$ for all $1\le k \neq \ell \le K$. 
In each test, we generate 3 datasets by setting $s=6$ and $N_k=500$ for all $k\in [K]$ and respectively run the KSS method with random initialization (denoted by RI-KSS) and TIPS initialization (denoted by TI-KSS) by setting $\tau=2/\sqrt{\overline{d}}$ on them. Then, we plot the distance of the iterates to ground truth, i.e., $d_F^2(\bH^k,\bH^*)+10^{-8}$, against the iteration numbers in Figure \ref{fig:1}. It can be observed that with a proper initialization, the KSS method converges so quickly that it finds the correct clustering within 10 iterations. This supports the result in Theorem \ref{thm:KSS}. Additionally, it exhibits a finite termination phenomenon that corroborates the result in Lemma \ref{lem:one-step}. Moreover, it is observed in Figure \ref{fig:1}(c) that RI-KSS gets stuck at a local minimum while TI-KSS does not on data 3. 

\begin{table}[!htbp]
\caption{Average CPU time (in seconds) and the best clustering accuracy of the tested methods on real datasets.}
\label{table-2}
\begin{center}
\begin{tabular}{lccccccc}
\toprule
 Accuracy & \emph{COIL100} & \emph{YaleB} & \emph{USPS} & \emph{MNIST}  \\
\midrule
KSS & {\bf 0.8117} & 0.7154 & {\bf 0.8172} & {\bf 0.9780}\\ 
SSC & 0.6732 & {\bf 0.8277} &  0.6583 & --\\
OMP & 0.3393 & {\bf 0.8268} & 0.2109 & 0.5749 \\
TSC & {\bf 0.7343} & 0.4878 & 0.6693 & {\bf 0.8514}\\
GSC & 0.6550 & 0.7071 & {\bf 0.9522} & 0.6306\\
LRR & 0.5500 & 0.6828 & 0.7129 & --\\
LRSSC & 0.5200 & 0.7088 & 0.6443 & -- \\
\midrule
Time (s) & \emph{COIL100} & \emph{YaleB} & \emph{USPS} & \emph{MNIST}   \\
\midrule
KSS & 53.53 & 6.90 & {\bf 8.85} & {\bf 30.53}\\ 
SSC & 912.25 & 136.36 & 1217.88 & -- \\
OMP & {\bf 12.12} & {\bf 1.02} & 31.12 & 398.37\\
TSC & {\bf 29.78} & {\bf 3.06} & {\bf 2.66} & {\bf 154.46}\\
GSC & 178.15 & 24.22 & 105.59 & 1800.00\\
LRR & 144.25 & 63.31 & 112.56 & -- \\
LRSSC & 1800.00 & 444.28 & 1800.00 & -- \\
\bottomrule
\multicolumn{4}{l}{\footnotesize ``--'' denotes out of memory.}
\end{tabular}
\end{center}
\vskip -0.15in
\end{table}

\subsection{Numerical Efficiency and Accuracy on Real Data}\label{subsec:test-real}

We now conduct experiments to examine the computational efficiency and recovery accuracy of the KSS method on real datasets. We also compare it with several state-of-the-art methods: SSC in \citet{elhamifar2013sparse}, SSC solved by OMP in \citet{you2016scalable}, TSC in \citet{heckel2015robust}, GSC in \citet{park2014greedy}, LRR in \citet{liu2012robust}, and LRSSC in \citet{wang2019provable}. In the implementations of SSC, OMP, LRR, and LRSSC, we use the source codes provided by their authors. We use the real datasets  \emph{COIL-100}~\cite{COIL100}, the cropped extended \emph{Yale B}~\cite{GeBeKr01}, \emph{USPS}~\cite{hull1994database}, and \emph{MNIST} \cite{lecun1998mnist}.\footnote{The datasets \emph{COIL-100}, \emph{Yale B}, and \emph{USPS} are downloaded from \href{http://www.cad.zju.edu.cn/home/dengcai/Data/data.html}{http://www.cad.zju.edu.cn/home/dengcai/Data/data.html}. The dataset \emph{MNIST} is downloaded from LIBSVM \cite{chang2011libsvm} at \href{https://www.csie.ntu.edu.tw/~cjlin/libsvmtools/datasets/}{https://www.csie.ntu.edu.tw/~cjlin/libsvmtools/datasets/}.} The stopping criteria for the tested methods are given as follows. For KSS, we terminate it when the norm of two consecutive iterates is less than $10^{-2}$. For SSC, LRR, and LRSSC, we use the stopping criteria in their source codes. No stopping criterion is needed for TSC and GSC due to their one-shot nature. We set the maximum iteration number of KSS, SSC, LRR, and LRSSC as 200. We set the maximum running time of all tested algorithms as 1800 seconds. 
For the implementation of KSS, we used the TIPS initialization except for on \emph{MNIST}, where we use random initialization in Algorithm \ref{alg-1}.
More details, including data processing, parameter settings, and test results, can be found in Section \ref{appen:expe} of the appendix. Then, we run each method 10 times. Note that if the algorithms are initialized deterministically, the only randomness is from the initialization for k-means in spectral clustering. To compare the computational efficiency and recovery accuracy of the tested methods, we report the average running time and best clustering accuracy for all runs of each method in Table \ref{table-2}. More experiment results can be also found in Section \ref{appen:expe} of the appendix. 
It can be observed that the KSS method is in the top three in terms of \emph{both} accuracy and computational efficiency for every dataset. This demonstrates the efficiency and efficacy of the KSS method for SC.

\section{Concluding Remarks}\label{sec:con}

In this work, we analyzed the KSS method for subspace clustering and provided a TIPS method for its initialization in the semi-random UoS model. We showed that provided an initial assignment satisfying a partial recovery condition, the KSS method converges superlinearly  and achieves correct clustering within $\Theta(\log\log N)$ iterations, even when the normalized affinity between pairwise subspaces is $O(1)$. Moreover, we proved that the proposed initialization method can return a qualified initial point. All these results are demonstrated by the numerical results. A natural future direction is to study the convergence behavior and recovery performance of the KSS method in the noisy UoS model; see, e.g., \citet{heckel2015robust,soltanolkotabi2014robust,tschannen2018noisy,wang2013noisy}. 

 
\section*{Acknowledgements}


The first and last authors are supported in part by ARO YIP award W911NF1910027, in part by AFOSR YIP award FA9550-19-1-0026, and in party by NSF CAREER award CCF-1845076. The second author is supported by the National Natural Science Foundation of China (NSFC) Grant 72192832. The third author is supported in part by the Hong Kong Research Grants Council (RGC) General Research Fund (GRF) Project CUHK 14205421. The authors thank the reviewers for their insightful comments. They also thank John Lipor for sharing some data and code. 
 
\bibliographystyle{icml2022}
\bibliography{subspace-clustering.bib}

\newpage
\onecolumn
\begin{appendix}
\begin{center}
{\Large \bf Appendix}
\end{center}
\vspace{-0.1in}
\par\noindent\rule{\textwidth}{1pt}
\setcounter{section}{0}

\renewcommand\thesection{\Alph{section}}

In the appendix, we provide proofs of the technical results presented in Sections \ref{sec:preli} and \ref{sec:pf-main}. To proceed, we introduce some further notations.  Given a vector $\ba \in \R^n$, we denote by $\mathrm{diag}(\ba) \in \R^{n\times n}$ the diagonal matrix with $\ba$ on its diagonal. Given a symmetric matrix $\bA$, we use $\lambda_{\min}(\bA)$ to denote its smallest eigenvalue. We respectively use $\bo_n$, $\bE_n$, and $\bI_d$ to denote the $n$-dimensional all-one vector, $n\times n$ all-one matrix, and $d\times d$ identity matrix, and simply write $\bo$, $\bE$, and $\bI$ when their dimension can be inferred from the context. Given two random variables $X$ and $Y$, we write $X\overset{d}{=}Y$ if $X$ and $Y$ are equal in distribution. We use $\be_i$ to denote a standard basis with a 1 in the $i$-th coordinate and $0$'s elsewhere. For a vector $\bx \in \R^n$, we denote by $\bx_S$ its subvector consisting of the elements indexed by the set $S$. We denote the cumulative distribution function of the standard normal distribution by
\begin{align*}
\Phi(x) = \frac{1}{\sqrt{2\pi}} \int_{-\infty}^x e^{-\frac{t^2}{2}} dt. 
\end{align*} 
For any random vector $\ba \sim \mathrm{Unif}(\S^{d-1})$, it is known that there exists a standard normal random vector such that $\ba$ is its normalization. We denote such vector by $\bba$. Thus, it holds that
\begin{align}\label{eq:bar-a}
\bba \sim \mN(\b0,\bI_d),\ \ba=\frac{\bba}{\|\bba\|}.
\end{align}
Moreover, let 
\begin{align}\label{svd}
\bU^{*^T}_k\bU^*_\ell = \bU^*_{k\ell}\bm{\Sigma}^*_{k\ell}\bV_{k\ell}^{*^T}
\end{align}
be a singular value decomposition (SVD) of $\bU^{*^T}_k\bU^*_\ell$, where $\sigma_{k\ell}^{(1)} \ge \dots \ge \sigma_{k\ell}^{(\min\{d_k,d_\ell\})} \ge 0$ are the singular values of $\bU^{*^T}_k\bU^*_\ell$ and $\bU^*_{k\ell} \in \mO^{d_k},\bV^*_{k\ell} \in \mO^{d_\ell}$. Suppose that $d_k \ge d_\ell$. We have
\begin{align}\label{d:k>l}
\bU^{*^T}_k\bU^*_\ell = \bU^*_{k\ell} \begin{bmatrix}
\bar{\bm{\Sigma}}_{k\ell}^* \\
\bm{0}
\end{bmatrix}\bV_{k\ell}^{*^T},
\end{align}
where $\bar{\bm{\Sigma}}_{k\ell}^* = \mathrm{diag}\left(\sigma_{k\ell}^{(1)},\dots,\sigma_{k\ell}^{(d_\ell)}  \right)$.  
Suppose to the contrary that $d_k < d_\ell$. Then, we have
\begin{align}\label{d:k<l}
\bU^{*^T}_k\bU^*_\ell = \bU^*_{k\ell} \begin{bmatrix}
\bar{\bm{\Sigma}}_{k\ell}^* & 
\bm{0}
\end{bmatrix}\bV_{k\ell}^{*^T},
\end{align}
where $\bar{\bm{\Sigma}}_{k\ell}^* = \mathrm{diag}\left(\sigma_{k\ell}^{(1)},\dots,\sigma_{k\ell}^{(d_k)}  \right)$. According to \eqref{affi:two} and \eqref{svd}, one can verify that
\begin{align}\label{eq:aff=F}
\mathrm{aff}(S_k^*,S_\ell^*) =   \|\bm{\Sigma}_{k\ell}^*\|_F.
\end{align}
Recall that for any $\bU,\bV \in \mO^{n\times d}$, we use
$ d(\bU,\bV) =  \left\| \bU\bU^T-\bV\bV^T \right\|$ to denote the distance between the subspaces respectively spanned by $\bU$ and $\bV$.  
Then, one can verify 
\begin{align}\label{dist:subsp}
d(\bU,\bV) = \left\| (\bI-\bU\bU^T)\bV \right\| =  \left\| (\bI-\bV\bV^T)\bU\right\| = \sqrt{1-\sigma^2_{\min}(\bU^T\bV)}. 
\end{align}

\section{Concentration Inequalities}

In this section, we present some concentration inequalities for random vectors. These inequalities play an important role in the analysis of the proposed method. We first introduce a spectral bound on the covariance estimation for random vectors generated by a uniform distribution over the sphere. It is a direct consequence of \citet[Theorem 4.7.1]{vershynin2018high} and thus we omit its proof.

\begin{lemma}\label{lem:cov-esti}
Suppose that $\ba_1,\dots,\ba_m \in \R^d$  are i.i.d. uniformly distributed over the unit sphere. Then, it holds with probability at least $1-2e^{-u}$ that 
\begin{align*}
\left\| \frac{1}{m} \sum_{i=1}^m \ba_i\ba_i^T  - \frac{1}{d}\bI_d \right\| \le \frac{c_1}{d} \left( \sqrt{\frac{d+u}{m}}  + \frac{d+u}{m}\right),
\end{align*}
where $c_1 > 0$ is an absolute constant. 
\end{lemma}

We next present a bound on the deviation of the weighted sum of standard normal random variables from its mean. This is an extension of \citet[Lemma 7]{li2019theory}. 
\begin{lemma}\label{lem:norm-gaus}
Let $\bx \in \R^d$ be a normal random vector such that $\bx \sim \mN(\b0,\sigma^2\bI_d)$.
It holds for $\lambda_1,\dots,\lambda_d \in [0,1]$ with $\sum_{i=1}^d \lambda_i^2 \ge 4$ and $t > 0$ that  
\begin{align*}
\P\left(  \left|\sqrt{\sum_{i=1}^d \lambda_i^2x_i^2} - \sigma\sqrt{\sum_{i=1}^d \lambda_i^2}\right| \ge t + 2\sigma\right) \le 2\exp\left( -\frac{t^2}{2\sigma^2} \right). 
\end{align*}
\end{lemma}
\begin{proof}[Proof of Lemma \ref{lem:norm-gaus}] 
We define
\begin{align*}
f(\bx) = \sqrt{\sum_{i=1}^d \lambda_i^2x_i^2}. 
\end{align*}
By calculation, we obtain
\begin{align*}
\|\nabla f(\bx)\| = \sqrt{\frac{\sum_{i=1}^d\lambda_i^4x_i^2}{\sum_{i=1}^d\lambda_i^2x_i^2}} \le 1. 
\end{align*}
Applying the concentration inequality for Lipschitz functions (see, e.g., \citet[Lemma 6]{li2019theory})) to $f(\bx)$ yields that 
\begin{align}\label{eq1:lem-norm-gaus}
\P\left( \left| f(\bx) - \E[f(\bx)] \right| \ge t \right) \le 2\exp\left( -\frac{t^2}{2\sigma^2} \right).  
\end{align}
We first note that
\begin{align}\label{eq2:lem-norm-gaus}
\E[f(\bx)] \le \sqrt{\E[f^2(\bx)]} = \sqrt{\E\left[ \sum_{i=1}^d \lambda_i^2x_i^2 \right]}  = \sigma\sqrt{\sum_{i=1}^d\lambda_i^2}.
\end{align}
By letting $X = f(\bx) \ge 0$ and $\mu = \E[X]$, we can compute
\begin{align*}
\mathrm{Var}\left( X \right) & = \E\left[(X-\mu)^2 \right]  =  \int_0^\infty t^2 d\P\left(  \left|X-\mu  \right|\le t \right)   = -\int_0^\infty t^2 d\P\left(  \left|X-\mu  \right| > t \right)\\
&  = \int_0^\infty 2t \P\left(  \left|X-\mu  \right| > t \right) dt  \le  \int_0^\infty 4t \exp\left( -\frac{t^2}{2\sigma^2} \right) dt = 4\sigma^2,
\end{align*} 
where the forth equality and the last one follow from integration by parts and the inequality is due to \eqref{eq1:lem-norm-gaus}. Thus, we have
\begin{align*}
\E^2[f(\bx)] & = \E[f^2(\bx)] - \mathrm{Var}\left( f(\bx) \right) = \E\left[ \sum_{i=1}^d \lambda_i^2x_i^2 \right] - \mathrm{Var}\left( f(\bx) \right) \ge \sigma^2 \left(\sum_{i=1}^d\lambda_i^2 - 4 \right).
\end{align*}
This, together with $\sum_{i=1}^d \lambda_i^2 \ge 4$, implies
\begin{align}\label{eq3:lem-norm-gaus}
\E[f(\bx)] \ge \sigma\sqrt{ \sum_{i=1}^d\lambda_i^2 - 4 } \ge \sigma\left( \sqrt{\sum_{i=1}^d\lambda_i^2}-2 \right). 
\end{align}
Plugging \eqref{eq2:lem-norm-gaus} and \eqref{eq3:lem-norm-gaus} into \eqref{eq1:lem-norm-gaus} yields the desired result. 
\end{proof}

Equipped with the above results, we are ready to present a lemma that characterizes the properties of a uniform distribution over the sphere. This plays an important role in the subsequent analysis. 
\begin{lemma}\label{lem:norm-unit} 
Suppose that $\|\bm{\Sigma}_{k\ell}^*\|_F \ge 2$ for all $1 \le k\neq \ell \le K$ and $\ba_i \overset{i.i.d.}{\sim} \mathrm{Unif}(\S^{d_\ell-1})$ for all $i\in [N]$. Then, it holds with probability at least $1-4N^{-2}$ that for some $i\in [N]$ and $1 \le k\neq \ell \le K$,
\begin{align}\label{rst1:lem-norm-unit}
\left|\|\bar{\ba}_i\| - \sqrt{d_\ell}\right| \le  \alpha,\quad \left|\|\bm{\Sigma}^*_{k\ell}\bar{\ba}_i\| - \|\bm{\Sigma}^*_{k\ell}\|_F\right| \le \alpha,
\end{align}
and 
\begin{align}\label{rst2:lem-norm-unit}
  \frac{\|\bm{\Sigma}^*_{k\ell}\|_F-\alpha}{\sqrt{d_\ell}+\alpha} \le  \|\bm{\Sigma}^*_{k\ell}\ba_i\| \le \frac{\|\bm{\Sigma}^*_{k\ell}\|_F+\alpha}{\sqrt{d_\ell}-\alpha},
\end{align}
where $\alpha=2\sqrt{\log N}+2$.
\end{lemma}
\begin{proof}[Proof of Lemma \ref{lem:norm-unit}] 
We first prove \eqref{rst1:lem-norm-unit}. Applying Lemma \ref{lem:norm-gaus} with $t=2\sqrt{\log N}$ and $\lambda_j=1$ for all $j\in [d_\ell]$ to $\bar{\ba}_i \sim \mN(\b0,\bI_{d_\ell})$ yields that 
\begin{align}\label{eq1:lem:norm-unit}
\P\left(\left|\|\bar{\ba}_i\| - \sqrt{d_\ell}\right| \ge \alpha \right) \le 2N^{-2}. 
\end{align}
Suppose that $d_k \ge d_\ell$. According to \eqref{d:k>l}, we have 
$\bm{\Sigma}_{k\ell}^* = \begin{bmatrix}
\bar{\bm{\Sigma}}_{k\ell}^* \\
\bm{0}
\end{bmatrix}$. 
Applying Lemma \ref{lem:norm-gaus} with $t=2\sqrt{\log N}$ and $\lambda_j=\sigma_{k\ell}^{(j)}$ for all $j\in [d_\ell]$ to $\bar{\ba}_i \sim \mN(\b0,\bI_{d_\ell})$ yields
\begin{align*}
\P\left( \left|\|\bar{\bm{\Sigma}}^*_{k\ell}\bar{\ba}_i\| - \|\bar{\bm{\Sigma}}^*_{k\ell}\|_F\right| \ge \alpha \right) \le 2N^{-2}.
\end{align*}
This, together with $\|\bm{\Sigma}^*_{k\ell}\bar{\ba}_i\| = \|\bar{\bm{\Sigma}}^*_{k\ell}\bar{\ba}_i\|$ and $\|\bm{\Sigma}^*_{k\ell}\|_F = \|\bar{\bm{\Sigma}}^*_{k\ell}\|_F$, implies
\begin{align}\label{eq2:lem:norm-unit}
\P\left( \left|\|\bm{\Sigma}^*_{k\ell}\bar{\ba}_i\| - \|\bm{\Sigma}^*_{k\ell}\|_F\right| \ge \alpha \right) \le 2N^{-2}.
\end{align}
Suppose to the contrary that $d_k < d_\ell$. According to \eqref{d:k<l}, we have 
$\bm{\Sigma}_{k\ell}^*\bar{\ba}_i = \left(
\sigma_{k\ell}^{(1)}\bar{a}_{i1}, \dots,  \sigma_{k\ell}^{(d_k)}\bar{a}_{id_k} \right)$. Applying Lemma \ref{lem:norm-gaus} with $t=2\sqrt{\log N}$ and $\lambda_j=\sigma_{k\ell}^{(j)}$ for all $j\in [d_k]$ to $[\bar{\ba}_i]_{S} \sim \mN(\b0,\bI_{d_k})$ with $S=[d_k]$ yields
\begin{align*}
\P\left( \left|\|\bm{\Sigma}^*_{k\ell}\bar{\ba}_i\| - \|\bm{\Sigma}^*_{k\ell}\|_F\right| \ge \alpha \right) \le 2N^{-2}.
\end{align*}
This, together with \eqref{eq1:lem:norm-unit}, \eqref{eq2:lem:norm-unit}, and the union bound, implies \eqref{rst1:lem-norm-unit}. 

We next prove \eqref{rst2:lem-norm-unit} using \eqref{rst1:lem-norm-unit}. Using $\ba_i=\bar{\ba}_i/\|\bba_i\|$ and \eqref{rst1:lem-norm-unit}, we have
\begin{align*} 
 \|\bm{\Sigma}^*_{k\ell}\ba_i\| = \frac{\|\bm{\Sigma}^*_{k\ell}\bar{\ba}_i\|}{\|\bar{\ba}_i\|} \le   \frac{\|\bm{\Sigma}^*_{k\ell}\|_F+\alpha}{\sqrt{d_\ell}-\alpha}
\end{align*}
and 
\begin{align*} 
\|\bm{\Sigma}^*_{k\ell}\ba_i\| = \frac{\|\bm{\Sigma}^*_{k\ell}\bar{\ba}_i\|}{\|\bar{\ba}_i\|}  \ge   \frac{\|\bm{\Sigma}^*_{k\ell}\|_F-\alpha}{\sqrt{d_\ell}+\alpha}.
\end{align*}
Then, we complete the proof. 
\end{proof}

Then, we present a lemma that estimates the magnitudes of some crucial parameters in our analysis. 
\begin{lemma}\label{lem:norm-Uz}
Suppose that $\bz_i  \in \R^N$ are generated according to the semi-random UoS model such that $\bz_i \in S_{\ell}^*$. Then, for any $1 \le i \neq j \le N$ and $k \in [K]$, it holds  with probability at least $1 - 5K^2/N $ that 
\begin{align}\label{rst1:eq-norm-Uz}
 \frac{\mathrm{aff}(S_k^*,S_{\ell}^*)-\alpha}{\sqrt{d_\ell}+\alpha} \le  \|\bU_k^{*^T}\bz_i\| \le \frac{\mathrm{aff}(S_k^*,S_\ell^*)+\alpha}{\sqrt{d_\ell}-\alpha} 
\end{align}
and
\begin{align}\label{rst2:eq-norm-Uz}
\left|\frac{\langle \bU_k^{*^T}\bz_i,  \bU_k^{*^T}\bz_j\rangle}{\|\bU_k^{*^T}\bz_j\|}\right| \le \frac{2\sqrt{\log N}}{\sqrt{d_\ell}-\alpha},
\end{align}
where $\alpha=2\sqrt{\log N}+2$.
\end{lemma}
\begin{proof}[Proof of Lemma \ref{lem:norm-Uz}] 
Suppose that \eqref{rst1:lem-norm-unit} and \eqref{rst2:lem-norm-unit} hold for all $i\in [N]$ and $k,\ell \in [K]$, which happens with probability $1-4K^2N^{-1}$ according to Lemma \ref{lem:norm-unit} and the union bound.  We first show \eqref{rst1:eq-norm-Uz}. Since $\bz_i \in S_\ell^*$ and a uniform distribution over the sphere is rotationally invariant, we have $\|\bU_k^{*^T}\bz_i\|=\|\bU_k^{*^T}\bU_\ell^*\ba_i\| = \|\bU_{k\ell}^*\bm{\Sigma}^*_{k\ell}\bV_{k\ell}^{*^T}\ba_i\|\sim \| \bm{\Sigma}^*_{k\ell} \ba_i\|$. This, together with \eqref{rst2:lem-norm-unit} and \eqref{eq:aff=F}, implies that for any $j\in [N]$ and $\ell \in [K]$,
\begin{align}\label{eq1:lem-norm-Uz}
 \frac{\mathrm{aff}(S_k^*,S_\ell^*)-\alpha}{\sqrt{d_\ell}+\alpha} \le  \|\bU_k^{*^T}\bz_i\| \le \frac{\mathrm{aff}(S_k^*,S^*_\ell)+\alpha}{\sqrt{d_\ell}-\alpha}. 
\end{align}
We next show \eqref{rst2:eq-norm-Uz}. According to $\bz_i \in S_{\ell}^*$ and \eqref{eq:bar-a}, we have
\begin{align}\label{eq2:lem-norm-Uz}
\frac{\langle \bU_k^{*^T}\bz_i,  \bU_k^{*^T}\bz_j\rangle}{\|\bU_k^{*^T}\bz_j\|}  = \frac{\langle \bU_k^{*^T}\bU_\ell^*\ba_i,  \bU_k^{*^T}\bz_j\rangle}{\|\bU_k^{*^T}\bz_j\|}   =  \frac{\langle \bU_k^{*^T}\bU_{\ell}^*\bba_i,  \bU_k^{*^T}\bz_j\rangle}{\|\bba_i\|\|\bU_k^{*^T}\bz_j\|}.
\end{align}
By letting $X$ be a standard normal random variable, i.e., $X\sim N(0,1)$, we compute
\begin{align*}
\P\left( \left|\frac{\langle \bU_k^{*^T}\bU_\ell^*\bba_i,  \bU_k^{*^T}\bz_j\rangle}{ \|\bU_k^{*^T}\bz_j\|}\right| \le 2\sqrt{\log N} \Bigm\vert  \bz_j\right) &= \P\left( \left|X\right| \le \frac{2\|\bU_k^{*^T}\bz_j\|\sqrt{\log N} }{\|\bU_{\ell}^{*^T}\bU_k^*\bU_k^{*^T}\bz_j\|} \right) \\
&  \ge \P\left( |X| \le 2\sqrt{\log N} \right) \ge 1- \sqrt{\frac{2}{\pi}}\frac{N^{-2}}{\sqrt{\log N}},
\end{align*}
where the equality is due to $ \langle \bU_k^{*^T}\bU_\ell^*\bba_i,  \bU_k^{*^T}\bz_j \rangle/{ \|\bU_k^{*^T}\bz_j\|} \sim \mN\left(\b0,\|\bU_{\ell}^{*^T}\bU_k^*\bU_k^{*^T}\bz_j\|^2/\| \bU_k^{*^T}\bz_j\|^2 \right)$  and the first inequality is due to $\|\bU_{\ell}^{*^T}\bU_k^*\bU_k^{*^T}\bz_j\| \le \| \bU_k^{*^T}\bz_j\|$. This, together with \eqref{rst1:lem-norm-unit}, \eqref{eq2:lem-norm-Uz}, and the union bound, implies that it holds with probability at least $1- K^2N^{-1}/{\sqrt{\log N}}$ that for all $1 \le i\neq j \le N$ and $\ell \in [K]$, 
\begin{align*}
\frac{|\langle \bU_k^{*^T}\bz_i,  \bU_k^{*^T}\bz_j\rangle|}{\|\bU_k^{*^T}\bz_j\|} \le \frac{2\sqrt{\log N}}{\sqrt{d_\ell}-\alpha}.
\end{align*}
This, together with \eqref{eq1:lem-norm-Uz} and the union bound, implies  the desired results.
\end{proof}

\section{Proofs in Section \ref{subsec:init}}\label{appenSec:pf-init}

According to Assumption \ref{AS:2}, $d_{\min} \gtrsim \log^3N$, and $\alpha=2\sqrt{\log N}+2$, there exists an $\varepsilon\lesssim 1/\sqrt{\log N}$ such that 
\begin{align}\label{eq:epsilon}
\alpha \le \varepsilon \mathrm{aff}(S_k^*,S_\ell^*)\ \text{for all}\ 1 \le k \neq \ell \le K,\quad \alpha \le \varepsilon  \sqrt{d_k}\ \text{for all}\ k\in [K].
\end{align}
This result shall be used in the subsequent proofs again and again. 

\subsection{Proof of Lemma \ref{lem:p-q} } 

Before we prove Lemma \ref{lem:p-q}, we need the following lemma to estimate the probability of the event that $\left|\langle \bz_i,\bz_j \rangle \right| \ge  \tau$ conditioned on $\bz_j$. Recall that we denote the cumulative distribution function of the standard normal distribution by
\begin{align*}
\Phi(x) = \frac{1}{\sqrt{2\pi}} \int_{-\infty}^x e^{-\frac{t^2}{2}} dt. 
\end{align*} 

\begin{lemma}\label{lem:prob-cond}
Suppose that $\bz_i \in S_k^*$ for some $k \in [K]$. Then, it holds for any $1 \le i \neq j \le N$ that
\begin{align}\label{rst:lem-prob-cond}
2 - 2 \Phi\left( \frac{\tau(\sqrt{d_k}+\alpha)}{\|\bU_k^{*^T}\bz_j\|} \right) - 2N^{-2} \le \P\left( \left|\langle \bz_i,\bz_j \rangle \right| \ge  \tau   \Bigm\vert  \bz_j \right) 
\le 2 - 2 \Phi\left( \frac{\tau(\sqrt{d_k}-\alpha)}{\|\bU_k^{*^T}\bz_j\|} \right) + 2N^{-2},
\end{align}
where $\alpha=2\sqrt{\log N}+2$. In particular, we have
\begin{align}\label{rst1:lem-prob-cond}
\P\left( \left|\langle \bz_i,\bz_j \rangle\right| \ge \frac{\|\bU_k^{*^T}\bz_j\|\sqrt{ \log N}}{\sqrt{d_k}}   \Bigm\vert  \bz_j \right) \le \sqrt{\frac{2}{\pi}} \frac{2\sqrt{d_k}}{(\sqrt{d_k}-\alpha)\sqrt{\log N}} N^{-\frac{(\sqrt{d_k}-\alpha)^2}{2d_k}} + 2N^{-2}.
\end{align}
\end{lemma}
\begin{proof}[Proof of Lemma \ref{lem:prob-cond}] 
According to $\bz_i \in S_k^*$ and \eqref{eq:bar-a}, we have
\begin{align*}
\P\left( \left|\langle \bz_i,\bz_j \rangle\right| \ge \tau  \Bigm\vert \bz_j \right)  &= 2\P\left(  \langle \bz_i,\bz_j \rangle  \ge \tau  \Bigm\vert \bz_j \right) = 2\P\left(  \langle \bU_k^*\ba_i,\bz_j \rangle  \ge \tau  \Bigm\vert \bz_j \right)  =  2\P\left(  \langle \bba_i,\bU_k^{*^T}\bz_j \rangle  \ge \tau\|\bba_i\|  \Bigm\vert \bz_j \right) \\
& \ge 2\P\left( \langle \bba_i,\bU_k^{*^T}\bz_j \rangle \ge \tau (\sqrt{d_k}+\alpha)  \Bigm\vert \bz_j\right) - 2\P\left( \|\bba_i\| \ge \sqrt{d_k}+ \alpha \right) \\
& \ge 2 - 2 \Phi\left( \frac{\tau(\sqrt{d_k}+\alpha)}{\|\bU_k^{*^T}\bz_j\|} \right) - 2N^{-2},
\end{align*}
where the first inequality is due to the union bound and the fact that $\ba_i$ is independent of $\ba_j$ and the second inequality follows from $\langle \bba_i,\bU_k^{*^T}\bz_j \rangle \sim \mN(0,\|\bU_k^{*^T}\bz_j\|^2)$ and uses \eqref{eq1:lem:norm-unit} for $\bba_i \in \mathrm{Unif}\left(\S^{d_k-1}\right)$. By the same argument, we obtain
\begin{align*}
\P\left( \left|\langle \bz_i,\bz_j \rangle\right| \ge \tau  \Bigm\vert \bz_j \right)  & =  2\P\left(  \langle \bba_i,\bU_k^{*^T}\bz_j \rangle  \ge \tau\|\bba_i\|  \Bigm\vert \bz_j \right) \\
& \le 2\P\left( \langle \bba_i,\bU_k^{*^T}\bz_j \rangle \ge \tau (\sqrt{d_k} - \alpha)  \Bigm\vert \bz_j\right) + 2\P\left( \|\bba_i\| \le \sqrt{d_k} - \alpha \right) \\
& \le 2 - 2 \Phi\left( \frac{\tau(\sqrt{d_k}-\alpha)}{\|\bU_k^{*^T}\bz_j\|} \right) + 2N^{-2}.
\end{align*}
This, together with $\tau=\|\bU_k^{*^T}\bz_j\| \sqrt{\log N}/\sqrt{d_k}$, yields
\begin{align*}
\P\left( \left|\langle \bz_i,\bz_j \rangle\right| \ge  \frac{\|\bU_k^{*^T}\bz_j\|\sqrt{ \log N}}{\sqrt{d_k}}   \Bigm\vert  \bz_j \right) & \le 2 - 2 \Phi\left( \frac{\sqrt{ \log N}(\sqrt{d_k}-\alpha) }{\sqrt{d_k}} \right) + 2N^{-2} \\
& \le \sqrt{\frac{2}{\pi}}\int_{\frac{\sqrt{ \log N}(\sqrt{d_k}-\alpha) }{\sqrt{d_k}}}^\infty \exp\left( -\frac{t}{2}\frac{\sqrt{ \log N}(\sqrt{d_k}-\alpha) }{\sqrt{d_k}} \right)dt + 2N^{-2} \\
& = \sqrt{\frac{2}{\pi}} \frac{2\sqrt{d_k}}{(\sqrt{d_k}-\alpha)\sqrt{\log N}} N^{-\frac{(\sqrt{d_k}-\alpha)^2}{2d_k}} + 2N^{-2}.
\end{align*}
\end{proof}

\begin{proof}[Proof of Lemma \ref{lem:p-q}] 
  
First, we prove \eqref{pkl-bkl}.  Suppose that a pair of data points $\bz_i, \bz_j \in  S_k^*$ for some $k\in[K]$. According to \eqref{rst:lem-prob-cond} in Lemma \ref{lem:prob-cond} and $\|\bU_k^{*^T}\bz_j\|=\| \ba_j\|=1$ due to the UoS model, we obtain
\begin{align*}
2 - 2 \Phi\left(  \tau(\sqrt{d_k}+\alpha)  \right) - 2N^{-2} \le \P\left( \left|\langle \bz_i,\bz_j \rangle\right| \ge \tau   \Bigm\vert  \bz_j \right) \le 2 - 2 \Phi\left(  \tau(\sqrt{d_k}-\alpha) \right) + 2N^{-2}.
\end{align*}
This, together with $p_{kk} = \E\left[ \P\left( |\langle \bz_i,\bz_j \rangle| \ge \tau \Bigm\vert \bz_j \right) \right]$, implies 
\begin{align}\label{p:value}
 2 - 2 \Phi\left(  \tau(\sqrt{d_k}+\alpha)  \right) - 2N^{-2} \le p_{kk} \le  2 - 2 \Phi\left(  \tau(\sqrt{d_k}-\alpha) \right) + 2N^{-2}.
\end{align} 
This, together with \eqref{bkl}, yields that for all $k \in [K]$,
\begin{align}\label{eq1:lem:p-q}
p_{kk}-b_{kk} & \le 2\Phi\left( \tau\sqrt{d_k}\right) - 2\Phi\left(\tau(\sqrt{d_k}-\alpha) \right) + 2N^{-2}  \le \sqrt{\frac{2}{\pi}}\exp\left( -\frac{\tau^2(\sqrt{d_k}-\alpha)^2}{2} \right)\tau\alpha + 2N^{-2}  \lesssim \frac{1}{\log N},
\end{align}
where the last inequality is due to \eqref{tau} and $d_{\min} \gtrsim \log^3N$. By the same argument, we have $p_{kk}-b_{kk} \gtrsim - \frac{1}{\log N}.$ This, together with \eqref{eq1:lem:p-q}, implies \eqref{pkl-bkl} for all $k=\ell$. Suppose that a pair of data points $\bz_i \in S_k^*, \bz_j \in S_\ell^*$ for some $1 \le k \neq \ell \le K$. Since a uniform distribution over the sphere is rotationally invariant, we have
\begin{align*}
\|\bU_k^{*^T}\bz_j\| = \|\bU_k^{*^T}\bU_\ell^*\ba_j\| = \|\bU_{k\ell}^*\bm{\Sigma}_{k\ell}^*\bV_{k\ell}^{*^T}\ba_j\| \sim  \|\bm{\Sigma}_{k\ell}^*\ba_j\|,
\end{align*}
where the second equality is due to \eqref{svd}. This, together with \eqref{rst2:lem-norm-unit} in Lemma \ref{lem:norm-unit} and \eqref{rst:lem-prob-cond} in Lemma \ref{lem:prob-cond}, implies it holds with probability at least $1-2N^{-2}$ that
\begin{align*}
 2 - 2 \Phi\left( \frac{\tau(\sqrt{d_k}+\alpha)(\sqrt{d_\ell}+\alpha)}{\|\bm{\Sigma}^*_{k\ell}\|_F-\alpha} \right) - 2N^{-2} \le \P\left( \left|\langle \bz_i,\bz_j \rangle\right| \ge \tau   \Bigm\vert  \bz_j \right) \le 2 - 2 \Phi\left( \frac{\tau(\sqrt{d_k}-\alpha)(\sqrt{d_\ell}-\alpha)}{\|\bm{\Sigma}^*_{k\ell}\|_F+\alpha} \right) + 2N^{-2}.
\end{align*}
This, together with $\mathrm{aff}(S_k^*,S_\ell^*)=\|\bm{\Sigma}^*_{k\ell}\|_F$ and $p_{k\ell} = \E\left[\P\left( |\langle \bz_i,\bz_j \rangle| \ge \tau \Bigm\vert \bz_j \right)\right]$, further implies
\begin{align}\label{q:value}
2 - 2 \Phi\left( \frac{\tau(\sqrt{d_k}+\alpha)(\sqrt{d_\ell}+\alpha)}{\mathrm{aff}(S_k^*,S_\ell^*)-\alpha} \right) - 4N^{-2} \le  p_{k\ell}  \le 2 -  2 \Phi\left( \frac{\tau(\sqrt{d_k}-\alpha)(\sqrt{d_\ell}-\alpha)}{\mathrm{aff}(S_k^*,S_\ell^*)+\alpha} \right) + 4N^{-2}.
\end{align}  
This, together with \eqref{bkl}, yields that for all $1 \le k \neq \ell \le K$,
\begin{align}\label{eq2:lem:p-q}
p_{k\ell} - b_{k\ell} &\le  2 \Phi\left( \frac{\tau\sqrt{d_kd_\ell}}{\mathrm{aff}(S_k^*,S_\ell^*)} \right)  - 2 \Phi\left( \frac{\tau(\sqrt{d_k}-\alpha)(\sqrt{d_\ell}-\alpha)}{\mathrm{aff}(S_k^*,S_\ell^*)+\alpha} \right) + 4N^{-2} \notag\\
& \le \sqrt{\frac{2}{\pi}}\exp\left( -\frac{\tau^2(\sqrt{d_k}-\alpha)^2(\sqrt{d_\ell}-\alpha)^2}{2(\mathrm{aff}(S_k^*,S_\ell^*)+\alpha)^2} \right)\left( \frac{\tau\sqrt{d_kd_\ell}}{\mathrm{aff}(S_k^*,S_\ell^*)}  - \frac{\tau(\sqrt{d_k}-\alpha)(\sqrt{d_\ell}-\alpha)}{\mathrm{aff}(S_k^*,S_\ell^*)+\alpha} \right)\notag\\
&  \le \sqrt{\frac{2}{\pi}}\exp\left( -\frac{(1-\varepsilon)^4\tau^2d_{\min}^2}{2(1+\varepsilon)^2\mathrm{aff}^2(S_k^*,S_\ell^*)} \right)  \left( 1- \frac{(1-\varepsilon)^2}{1+\varepsilon}\right) \frac{\tau\sqrt{d_kd_\ell}}{\mathrm{aff}(S_k^*,S_\ell^*)} \notag\\
& =  \sqrt{\frac{2}{\pi}} \frac{\varepsilon(3-\varepsilon)\tau\sqrt{d_kd_\ell}}{1+\varepsilon}\exp\left( -\frac{(1-\varepsilon)^4\tau^2d_{\min}^2}{2 (1+\varepsilon)^2  \mathrm{aff}^2(S_k^*,S_\ell^*)} \right) \frac{1}{\mathrm{aff}(S_k^*,S_\ell^*)} \notag\\
& = \sqrt{\frac{2}{\pi}}\exp\left(-\frac{1}{2} \right) \frac{\varepsilon(3-\varepsilon) \sqrt{d_kd_\ell}}{(1-\varepsilon)^2d_{\min}} \lesssim \frac{d_{\max} }{d_{\min}\sqrt{\log N}},
\end{align}
where the third inequality is due to \eqref{eq:epsilon} and the last inequality follows from $\varepsilon \lesssim 1/\sqrt{\log N}$ and the fact that  $ \exp\left( -\eta x^2/2 \right)x$ attains the maximum  at $x=1/\sqrt{\eta}$ when $x \in (0, \infty )$. By the same argument, we have $p_{k\ell}-b_{k\ell} \gtrsim - \frac{d_{\max}^{3/2}}{d_{\min}^{3/2}\sqrt{\log N}}.$ This, together with \eqref{eq1:lem:p-q}, implies \eqref{pkl-bkl} for all $1 \le k \neq \ell \le K$. 

Next, we prove \eqref{pq:gap}. Note that for any  $1 \le k \neq \ell \le K$,
\begin{align}\label{eq0:lem-p-q-gap}
(\sqrt{d_k}-\alpha)(\sqrt{d_\ell}-\alpha) - (\sqrt{d_k}+\alpha)\left( \mathrm{aff}(S_k^*,S^*_\ell)+\alpha \right) 
& =  \sqrt{d_k}\left(\sqrt{d_\ell}- \mathrm{aff}(S_k^*,S^*_\ell)\right)-\alpha\left(2\sqrt{d_k}+\sqrt{d_\ell}+ \mathrm{aff}(S_k^*,S^*_\ell)\right) \notag\\
& \ge  \sqrt{d_k}\left(\sqrt{d_\ell}- \mathrm{aff}(S_k^*,S^*_\ell)\right)-2\alpha\left(\sqrt{d_k}+\sqrt{d_\ell}\right) \notag\\
& \ge  \frac{1}{10}\sqrt{d_kd_\ell},
\end{align} 
where the first inequality follows from $\mathrm{aff}(S_k^*,S_\ell^*) \le  \sqrt{d_\ell}$ and the second inequality uses $\mathrm{aff}(S_k^*,S_\ell^*) \le 4\sqrt{d_\ell}/5$ due to \eqref{affi:re} and $d_{\min} \gtrsim \log^3N$. For ease of exposition, let 
\begin{align*}
x_{k\ell} = \frac{\tau(\sqrt{d_k}-\alpha)(\sqrt{d_\ell}-\alpha)}{\mathrm{aff}(S_k^*,S_\ell^*)+\alpha},\quad y_k = \tau(\sqrt{d_k}+\alpha).
\end{align*}
According to \eqref{eq0:lem-p-q-gap} and \eqref{tau}, we have $x_{k\ell} > y_k$ for any $1 \le k \neq \ell \le K$.  For all $1 \le k \neq \ell \le K$ satisfying $\mathrm{aff}(S_k^*,S_\ell^*) \ge \tau(\sqrt{d_k}-\alpha)(\sqrt{d_\ell}-\alpha)/(2c) - \alpha$, we have
\begin{align*}
x_{k\ell} \le 2c
\end{align*}
and
\begin{align*}
x_{k\ell} - y_k  \ge \frac{\tau\sqrt{d_kd_\ell}}{10\left(\mathrm{aff}(S_k^*,S_\ell^*)+\alpha\right)} \ge \frac{c\sqrt{d_kd_\ell}}{10\min\{\sqrt{d_k},\sqrt{d_\ell}\}\sqrt{d_{\max}}} \gtrsim \frac{\sqrt{d_{\min}}}{\sqrt{d_{\max}}}, 
\end{align*}
where the first inequality uses \eqref{eq0:lem-p-q-gap} and the second inequality follows from \eqref{eq:epsilon} and \eqref{affi:re}. This, together with \eqref{p:value}, \eqref{q:value}, and $x_{k\ell} > y_k$, yields that for all $1 \le k \neq \ell \le K$,
\begin{align}\label{eq1:lem-p-q-gap}
p_{kk} - p_{k\ell} & \ge 2 \left( \Phi\left(x_{k\ell}\right)  - \Phi\left(y_k \right) \right) - 6N^{-2}  \ge \sqrt{\frac{2}{\pi}} \exp\left( - \frac{x_{k\ell}^2}{2}\right)\left(  x_{k\ell} - y_k \right) - 6N^{-2} \gtrsim \frac{\sqrt{d_{\min}}}{\sqrt{d_{\max}}}.
\end{align} 
For all $1 \le k \neq \ell \le K$ satisfying $\mathrm{aff}(S_k^*,S_\ell^*) < \tau(\sqrt{d_k}-\alpha)(\sqrt{d_\ell}-\alpha)/(2c) - \alpha$, we have for all $1 \le k \neq \ell \le K$,
\begin{align*}
p_{kk} \ge 2 - 2\Phi\left(   \frac{c(\sqrt{d_k}+\alpha)}{\sqrt{d_{\max}}} \right) - 2N^{-2} \ge 2 - 2\Phi\left(  (1+\varepsilon) c\right) - 2N^{-2},
\end{align*}
where the first inequality is due to \eqref{p:value} and \eqref{tau} and the second inequality uses \eqref{eq:epsilon}, and 
\begin{align*}
p_{k\ell}  & \le 2 -  2 \Phi\left(  2c \right) + 4N^{-2}.
\end{align*}
Then, we have for all $1 \le k \neq \ell \le K$,
\begin{align*}
p_{kk} - p_{k\ell} \ge 2 \Phi\left(  2c \right) - 2\Phi\left(  (1+\varepsilon) c\right) - 6N^{-2},
\end{align*}
which is a constant due to the fact that $c>0$ is a constant. 
This, together with \eqref{eq1:lem-p-q-gap} and $d_{\min} \gtrsim \log N$, implies \eqref{pq:gap}. 
\end{proof}

\subsection{Proof of Lemma \ref{lem:sing-B}}

\begin{proof}
According to \eqref{bkl} and \eqref{tau}, we have for all $1 \le k \neq \ell \le K$,
\begin{align}\label{bkk:bound}
b_{kk} = 2 - 2\Phi\left( \frac{\sqrt{cd_k}}{\sqrt{d_{\max}}} \right) \ge 2 - 2\Phi(\sqrt{c})
\end{align} 
and
\begin{align}\label{bkl:bound}
b_{k\ell} = 2 - 2\Phi\left( \frac{\sqrt{cd_kd_\ell}}{\sqrt{d_{\max}}\mathrm{aff}(S_k^*,S_\ell^*)} \right) \le 2 - 2\Phi\left( \frac{\sqrt{cd_{\min}}}{\kappa\sqrt{d_{\max}}} \right) = 2 - 2\Phi\left( \frac{\sqrt{c}}{\kappa\sqrt{\kappa_d}} \right),
\end{align} 
where the inequality is due to $\mathrm{aff}(S_k^*,S_\ell^*) \le \kappa\min\{\sqrt{d_k},\sqrt{d}_\ell\}$ for all $1\le k \neq \ell \le K$ by Assumption \ref{AS:2} and \eqref{affi:re}. Then, we can decompose the symmetric matrix $\bB$ into $\bB = \bB_1+\bB_2$, where
\begin{align*}
\bB_1 =  \begin{bmatrix}
\frac{b_{11}}{2} & b_{12} & \dots & b_{1K} \\
b_{12} & \frac{b_{22}}{2} & \dots & b_{2K} \\ 
\vdots & \vdots & \ddots & \vdots \\
b_{1K} & b_{2K} & \dots & \frac{b_{KK}}{2} \\ 
\end{bmatrix},\quad  \bB_2 =  \frac{1}{2}\begin{bmatrix}
b_{11} & 0 & \dots & 0 \\
0 & b_{22} & \dots & 0 \\ 
\vdots & \vdots & \ddots & \vdots \\
0 & 0 & \dots & b_{KK} \\ 
\end{bmatrix}.
\end{align*}
According to \eqref{bkk:bound}, \eqref{bkl:bound}, and \eqref{kappa:AS2}, we can verify for all $k \in [K]$,
\begin{align*}
\frac{1}{2}|b_{kk}| - \left| \sum_{\ell:k\neq \ell} b_{k\ell} \right| \ge  1 - \Phi(\sqrt{c}) - 2(K-1) \left(1 -  \Phi\left( \frac{\sqrt{c}}{\kappa\sqrt{\kappa_d}} \right)\right) \ge 0,
\end{align*}
which implies that $\bB_1$ is a symmetric diagonally dominant matrix (see \citet[Section 4.1.1]{golub2013matrix}). Using the result that a symmetric diagonally dominant matrix  with real non-negative diagonal entries is positive semidefinite, we can conclude that $\bB_1$ is positive semidefinite. On the other hand, we can see that $\bB_2$ is a diagonal matrix with all the diagonal elements being larger than $1-\Phi(\sqrt{c})$. Then, we have
\begin{align*}
\min_{\|\bx\|=1}\bx^T\bB\bx \ge \min_{\|\bx\|=1}\bx^T\bB_1\bx + \min_{\|\bx\|=1}\bx^T\bB_2\bx = \lambda_{\min}(\bB_1) + \lambda_{\min}(\bB_2) \ge 1 - \Phi(\sqrt{c}). 
\end{align*}
Then, we complete the proof. 
\end{proof}

\subsection{Proof of Proposition \ref{prop:spec-gap}} 

Before we prove Proposition \ref{prop:spec-gap}, we need estimate the covariance between the random variables $a_{ik}$ and $a_{jk}$ generated by the thresholding procedure \eqref{step:thresholding} for all $1 \le i\neq j \le N$ and $k\in [N]$. 


\begin{lemma}\label{lem:ex-delta-ijk}
Suppose that $\bz_i$, $\bz_j$, and $\bz_k$ are different points generated according to the semi-random UoS model such that $\bz_k \in S_{\ell}$ for some $\ell \in [K]$. Suppose in addition that Assumption \ref{AS:2} holds, the thresholding parameter is set as in \eqref{tau}, $d_{\min} \gtrsim \log N$. 
Then, it holds  for any $1 \le i \neq j \le N$  with probability at least $1-5K^2N^{-2}$ that
\begin{align*}
\left| \E[a_{ik}a_{jk} \vert \bz_i,\bz_j] - \E[a_{ik}\vert \bz_i]\E[a_{jk} \vert \bz_j] \right| \lesssim  \frac{d_{\max}}{d_{\min}\sqrt{\log N}}.
\end{align*}
\end{lemma}
\begin{proof}[Proof of Lemma \ref{lem:ex-delta-ijk}] 
Suppose that \eqref{rst1:eq-norm-Uz} and \eqref{rst2:eq-norm-Uz} hold, which happens with probability at least $1-5K^2/N$ according to Lemma \ref{lem:norm-Uz}. 
To simplify the notations, let
\begin{align}\label{eq:v}
\bv_i := \bU_\ell^{*^T}\bz_i,\ \tilde{\bv}_i:=\frac{\bv_i}{\|\bv_i\|},\ \text{for all}\ i\in [N].
\end{align}
Besides, for any given $\bv_i$ and $\bv_j$, let 
\begin{align}\label{eq:beta} 
\beta_{ij} := \frac{\left(\tau  +  |\langle \bv_i, \tilde{\bv}_j \rangle| \sqrt{\log N}/\sqrt{d_\ell}\right)(\sqrt{d_\ell}+\alpha)}{ \left(\|\bv_i\| - |\langle \bv_i,\tilde{\bv}_j \rangle| \right) \sqrt{1- (\log N)/d_\ell}},\ \beta_{ij}^\prime := \frac{\left(\tau  - |\langle \bv_i, \tilde{\bv}_j \rangle|\sqrt{\log N}/\sqrt{d_\ell}\right)(\sqrt{d_\ell}-\alpha)}{ \|\bv_i\| }.
\end{align}
In addition, suppose that the following inequalities hold:
\begin{align}
 \E[a_{ik}a_{jk} \vert \bz_i,\bz_j] & \ge \left( 2 - 2\Phi\left(  \beta_{ij} \right) - 2N^{-2} \right)\P\left( \tau \le  |\langle \bz_j,\bz_k \rangle| \le  \frac{\|\bv_j\|\sqrt{\log N}}{\sqrt{d_\ell}} \Bigm\vert \bz_j \right)  \label{eq:low-bound-bijk} \\
\E[a_{ik}a_{jk} \vert \bz_i,\bz_j]  & \le \left( 2 - 2\Phi\left(  \beta_{ij}^\prime \right) + 2N^{-2}\right)   \P\left(  \tau \le |\langle \bz_j,\bz_k \rangle| \le \frac{ \|\bv_j\|\sqrt{\log N}}{\sqrt{d_\ell}} \Bigm\vert \bz_j  \right) + \notag\\
&\quad \left(2+2N^{-2}\right) \P\left( \frac{\|\bv_j\| \sqrt{\log N}}{\sqrt{d_\ell}} \le   |\langle \bz_j,\bz_k \rangle| \le 1 \Bigm\vert \bz_j \right).\label{eq:up-bound-bijk}
\end{align}
According to Lemma \ref{lem:prob-cond}, we obtain
\begin{align}\label{eq:bound-bijk}
2 - 2 \Phi\left( \frac{\tau(\sqrt{d_\ell}+\alpha)}{\|\bv_i\|} \right) - 2N^{-2} \le \E[a_{ik}\vert \bz_i]  \le 2 - 2 \Phi\left( \frac{\tau(\sqrt{d_\ell}-\alpha)}{\|\bv_i\|} \right) + 2N^{-2}.
\end{align}
This, together with  \eqref{eq:low-bound-bijk}, yields
\begin{align}\label{eq:low-bound}
&\quad \E[a_{ik}a_{jk} \vert \bz_i,\bz_j] - \E[a_{ik}\vert \bz_i]\E[a_{jk} \vert \bz_j] \notag\\ & \ge  2\left( 1 - \Phi\left(  \beta_{ij} \right) - N^{-2}\right) \P\left( \tau \le  |\langle \bz_j,\bz_k \rangle| \le  \frac{\|\bv_j\|\sqrt{\log N}}{\sqrt{d_\ell}} \Bigm\vert \bz_j \right) - 2\left( 1 - \Phi\left( \frac{\tau(\sqrt{d_\ell}-\alpha)}{\|\bv_i\|} \right) + N^{-2}\right) \notag \\
&\ \P\left( |\langle \bz_j,\bz_k \rangle| \ge \tau  \Bigm\vert \bz_j \right)  = 2\left( \Phi\left( \frac{\tau(\sqrt{d_\ell}-\alpha)}{\|\bv_i\|} \right) - \Phi\left(  \beta_{ij} \right) - 2N^{-2}\right)\P\left( \tau \le  |\langle \bz_j,\bz_k \rangle| \le  \frac{\|\bv_j\|\sqrt{\log N}}{\sqrt{d_\ell}} \Bigm\vert \bz_j \right)\notag\\
& -  2\left( 1 - \Phi\left( \frac{\tau(\sqrt{d_\ell}-\alpha)}{\|\bv_i\|} \right) + N^{-2}\right) \P\left( |\langle \bz_j,\bz_k \rangle| \ge  \frac{\|\bv_j\|\sqrt{\log N}}{\sqrt{d_\ell}} \Bigm\vert \bz_j \right) \gtrsim  -\frac{d_{\max}}{d_{\min}\sqrt{\log N}}, 
\end{align}
where the last inequality is due to Lemma \ref{lem:phi-gap},  Lemma \ref{lem:prob-cond}, and \eqref{tau}. By the similar argument, according to \eqref{eq:up-bound-bijk} and \eqref{eq:bound-bijk}, we have
\begin{align*}
\E[a_{ik}a_{jk} \vert \bz_i,\bz_j] - \E[a_{ik}\vert \bz_i]\E[a_{jk} \vert \bz_j]  & \le 2\left( \Phi\left( \frac{\tau(\sqrt{d_\ell}+\alpha)}{\|\bv_i\|} \right) - \Phi(\beta_{ij}^\prime) + 2N^{-2}\right)\P\left( \tau \le  |\langle \bz_j,\bz_k \rangle| \le  \frac{\sqrt{\log N}}{\sqrt{d_\ell}} \Bigm\vert \bz_j \right) \\ +\ 
& 2\left(  \Phi\left( \frac{\tau(\sqrt{d_\ell}+\alpha)}{\|\bv_i\|} \right) +2 N^{-2}\right)\P\left( |\langle \bz_j,\bz_k \rangle| \ge  \frac{\sqrt{\log N}}{\sqrt{d_\ell}} \Bigm\vert \bz_j \right) \lesssim  \frac{d_{\max}}{d_{\min}\sqrt{\log N}}.  
\end{align*}
According to this and \eqref{eq:low-bound}, we complete the proof.

Then, the rest of the proof is devoted to proving \eqref{eq:low-bound-bijk} and \eqref{eq:up-bound-bijk}.
According to \eqref{step:thresholding} and $\bz_k=\bU_\ell^{*}\ba_k$, we have
\begin{align}\label{eq:e-bijk} 
 \E[a_{ik}a_{jk} \vert \bz_i,\bz_j] & = \P\left( |\langle \bz_i,\bz_k \rangle| \ge \tau, |\langle \bz_j,\bz_k \rangle| \ge \tau \Bigm\vert  \bz_i,\bz_j\right) \notag\\ 
& =  \int_{-\infty}^\infty \P\left( |\langle \bz_i,\bz_k \rangle| \ge \tau, |\langle \bz_j,\bz_k \rangle| \ge \tau  \Bigm\vert  |\langle \bz_j,\bz_k \rangle| = t, \bz_i,\bz_j \right) d\P\left(  |\langle \bz_j,\bz_k \rangle| \le t \Bigm\vert \bz_j\right) \notag\\
& = \int_{\tau}^1 \P\left( |\langle \bz_i,\bz_k \rangle| \ge \tau \Bigm\vert |\langle \bz_j,\bz_k \rangle| = t,\bz_i,\bz_j \right) d\P\left(  |\langle \bz_j,\bz_k \rangle| \le t \Bigm\vert \bz_j \right)  \notag \\
& = \int_{\tau}^1  \P\left( |\langle \bv_i,\ba_k \rangle| \ge \tau \Bigm\vert  |\langle \bv_j,\ba_k \rangle| = t,\bz_i,\bz_j \right)  d\P\left(  |\langle \bz_j,\bz_k \rangle| \le t \Bigm\vert \bz_j \right),
\end{align} 
where the second equality is due to the the law  of total probability, the third equality uses $\tau \le |\langle \bz_j,\bz_k \rangle| \le 1$, and the last equality follows from $\bv_i=\bU_\ell^{*^T}\bz_i$ for all $i\in [N]$. Due to $\|\ba_k\|=1$ and $\| \tilde{\bv}_j\|=1$, we can decompose $\ba_k$ into two parts that are orthogonal:
\begin{align}\label{eq2:decomp-ak}
\ba_k = x \tilde{\bv}_j + \sqrt{1-x^2} \bb_k,
\end{align}
where $x \in \R$ and $\bb_k \in \R^{d_\ell}$ satisfying $\langle \bb_k,\tbv_j \rangle = 0$ and $\|\bb_k\|=1$. This, together with $|\langle \bv_j,\ba_k \rangle| = t$, implies
\begin{align}\label{eq7:x-t}
t = |x|\|\bv_j\|.
\end{align}
Since $\|\tbv_j\|=1$, there exists an orthogonal matrix $\bU \in \mO^{d_\ell}$ such that $\bU\tbv_j=\be_1$. Let 
\begin{align}\label{eq1:bbk}
\tilde{\bb}_k  := \frac{\bU^T\ba_k - x \be_1 }{\sqrt{1-x^2}}
\end{align}
and $\bc_k \in \R^{d_{\ell}-1}$ such that $\bc_k:=(\tilde{b}_{k2},\cdots,\tilde{b}_{kd_{\ell}})$. According to Lemma \ref{lem:unif-dis}, we obtain 
\begin{align}\label{eq8:tb}
\bU\bb_k \sim \tilde{\bb}_k
\end{align}
such that $\tilde{b}_{k1}=0$ and $\bc_k \sim \mathrm{Unif}(\S^{d_{\ell}-2})$. Besides, let 
\begin{align}\label{eq0:tba}
\bw_i := \bv_i - \langle \bv_i,\tilde{\bv}_j \rangle \tilde{\bv}_j.
\end{align}
According to \eqref{eq2:decomp-ak}, we have
\begin{align}\label{eq3:lem-ex-delta-ijk}
\langle \bv_i,\ba_k \rangle & = x \langle \bv_i, \tilde{\bv}_j \rangle  + \sqrt{1-x^2} \langle \bv_i,\bb_k \rangle  = x \langle \bv_i, \tilde{\bv}_j \rangle + \sqrt{1-x^2} \langle \bw_i, \bb_k \rangle,
\end{align}
where the second equality is due to \eqref{eq0:tba} and $\langle \tbv_j, \bb_k \rangle = 0$. According to $\langle \bw_i,\tbv_j \rangle = 0$, we have $\langle \bU^T\bw_i, \bU\tbv_j \rangle = \langle \bU^T\bw_i, \be_1 \rangle = 0$. This, together with letting $\bu_i$ denote the $i$-th column of $\bU \in \mO^{d_\ell}$ and $\bm{d}_i :=\left(\bu_2^T\bw_i,\dots,\bu_d^T\bw_i \right) \in \R^{d_\ell-1}$, we have $\bu_1^T\bw_i=0$ and $\|\bm{d}_i\|=\|\bw_i\|$. It follows from this and \eqref{eq8:tb} that
\begin{align}\label{eq6:lem-ex-delta-ijk}
|\langle \bw_i, \bb_k \rangle| = |\langle \bU^T\bw_i, \bU\bb_k \rangle| \sim |\langle \bU^T\bw_i, \tilde{\bb}_k \rangle| = |\langle \bm{d}_i, \bc_k \rangle|.
\end{align}
Now, we are ready to compute the lower bound of $\E[a_{ik}a_{jk} \vert \bz_i,\bz_j]$. According to \eqref{eq:e-bijk}, \eqref{eq7:x-t}, \eqref{eq3:lem-ex-delta-ijk}, and \eqref{eq6:lem-ex-delta-ijk}, we have
\begin{align}\label{eq4:lem-ex-delta-ijk}
\P\left( |\langle \bv_i,\ba_k \rangle| \ge \tau \Bigm\vert \left| \langle \bv_j,\ba_k \rangle \right| = t,\bz_i,\bz_j \right)& \ge  \P\left( |\langle\bw_i, \bb_k  \rangle| \ge \frac{\tau + |x\langle \bv_i, \tbv_j \rangle|}{\sqrt{1-x^2}}  \Bigm\vert |x| = \frac{t}{\|\bv_j\|}, \bz_i,\bz_j \right) \notag\\
& = \P\left( |\langle \bm{d}_i, \bc_k \rangle| \ge \frac{\tau\|\bv_j\| + t|\langle \bv_i, \tilde{\bv}_j \rangle|}{\sqrt{\|\bv_j\|^2-t^2}} \Bigm\vert  \bz_i,\bz_j \right).
\end{align}
Then, let
\begin{align*}
h(t) :=  \frac{\left(\tau\|\bv_j\| + t|\langle \bv_i, \tilde{\bv}_j \rangle|\right)(\sqrt{d_\ell}+\alpha)}{ \left(\|\bv_i\| - |\langle \bv_i,\tilde{\bv}_j \rangle| \right) \sqrt{\|\bv_j\|^2-t^2}}.
\end{align*}
According to the argument in Lemma \ref{lem:prob-cond} with $\alpha=2\sqrt{\log N}+2$, $\bc_k \sim \mathrm{Unif}(\S^{d_\ell-2})$, and $\|\bm{d}_i\|=\|\bw_i\|$, we obtain
\begin{align}\label{eq5:lem-ex-delta-ijk}
\P\left( |\langle \bm{d}_i, \bc_k \rangle| \ge \frac{\tau\|\bv_j\| + t|\langle \bv_i, \tilde{\bv}_j \rangle|}{\sqrt{\|\bv_j\|^2-t^2}} \Bigm\vert \bz_i,\bz_j \right) & \ge  2 - 2\Phi\left( \frac{\left(\tau\|\bv_j\| + t|\langle \bv_i, \tilde{\bv}_j \rangle|\right)(\sqrt{d_\ell}+\alpha)}{\|\bw_i\|\sqrt{\|\bv_j\|^2-t^2}} \right) - 2N^{-2}  \notag\\
& \ge 2 - 2\Phi\left(  h(t) \right) - 2N^{-2},
\end{align}
where the second inequality is due to $\|\bw_i\| \ge  \|\bv_i\| - |\langle \bv_i,\tilde{\bv}_j \rangle | > 0$ according to \eqref{eq0:tba} and  the triangle inequality. According to  \eqref{eq:e-bijk}, \eqref{eq4:lem-ex-delta-ijk}, and \eqref{eq5:lem-ex-delta-ijk}, we have
\begin{align}\label{eq:low-bijk}
\E[a_{ik}a_{jk} \vert \bz_i,\bz_j] & \ge   \int_{\tau}^1   \left( 2 - 2\Phi\left(  h(t) \right) - 2N^{-2}\right)  d\P\left( |\langle \bz_j,\bz_k \rangle| \le t\Bigm\vert \bz_j \right) \notag\\
& \ge \int_{\tau}^{\frac{\|\bv_j\|\sqrt{\log N}}{\sqrt{d_\ell}}}  \left( 2 - 2\Phi\left(  h(t) \right) - 2N^{-2}\right)  d\P\left( |\langle \bz_j,\bz_k \rangle| \le t \Bigm\vert \bz_j \right) \notag\\
& \ge \int_{\tau}^{\frac{\|\bv_j\|\sqrt{\log N}}{\sqrt{d_\ell}}}  \left( 2 - 2\Phi\left(  \beta_{ij} \right) - 2N^{-2}\right)  d\P\left(  |\langle \bz_j,\bz_k \rangle| \le t \Bigm\vert \bz_j \right)\notag \\
& = \left( 2 - 2\Phi\left( \beta_{ij} \right) - 2N^{-2}\right) \P\left( \tau \le  |\langle \bz_j,\bz_k \rangle| \le  \frac{\|\bv_j\|\sqrt{\log N}}{\sqrt{d}} \Bigm\vert \bz_j \right),
\end{align}
where the last inequality is because $h(t)$ is an increasing function. Next, by letting
\begin{align*}
g(t) := \frac{\left(\tau\|\bv_j\| - t|\langle \bv_i, \tilde{\bv}_j \rangle|\right)(\sqrt{d_\ell}-\alpha)}{ \|\bv_i\|\|\bv_j\|},
\end{align*}
we can obtain the following inequality by the same argument as \eqref{eq4:lem-ex-delta-ijk} and \eqref{eq5:lem-ex-delta-ijk}:
\begin{align*}
 \P\left( |\langle \bv_i,\ba_k \rangle| \ge \tau \Bigm\vert  \langle \bv_j,\ba_k \rangle = t,\bz_i,\bz_j \right) & \le 2 - 2\Phi\left(  \frac{\left(\tau\|\bv_j\| - t|\langle \bv_i, \tilde{\bv}_j \rangle|\right)(\sqrt{d_\ell} - \alpha)}{ \|\bw_i\|  \sqrt{\|\bv_j\|^2-t^2}} \right) + 2N^{-2} \\
& \le  2 - 2\Phi\left(g(t)\right) + 2N^{-2},
\end{align*}
where the last inequality is due to $\|\bw_i\| \le \|\bv_i\|$. 
Besides, it holds for $t\in (\tau,\|\bv_j\|\sqrt{\log N}/\sqrt{d_\ell}]$ that
\begin{align*}
g(t) \ge \beta_{ij}^\prime. 
\end{align*}
These, together with \eqref{eq:e-bijk}, imply
\begin{align}\label{eq:up-bijk}
\E[a_{ik}a_{jk} \vert \bz_i,\bz_j] & \le \int_{\tau}^1   \left( 2 - 2\Phi\left(  g(t) \right) + 2N^{-2}\right)  d\P\left( |\langle \bz_j,\bz_k \rangle| \le t \Bigm\vert \bz_j \right)\notag\\
& = \int_{\tau}^{\frac{\|\bv_j\|\sqrt{\log N}}{\sqrt{d_\ell}}}   \left( 2 - 2\Phi\left(  g(t) \right) + 2N^{-2}\right)  d\P\left( |\langle \bz_j,\bz_k \rangle| \le t \Bigm\vert \bz_j \right) +  \notag\\
&\quad \int^{1}_{\frac{\|\bv_j\|\sqrt{\log N}}{\sqrt{d_\ell}}}   \left( 2 - 2\Phi\left(  g(t) \right) + 2N^{-2}\right)  d\P\left( |\langle \bz_j,\bz_k \rangle| \le t \Bigm\vert \bz_j \right) \notag\\
& \le   \left( 2 - 2\Phi\left(  \beta_{ij}^\prime \right) + 2N^{-2}\right)   \P\left(  \tau \le |\langle \bz_j,\bz_k \rangle| \le \frac{\|\bv_j\|\sqrt{\log N}}{\sqrt{d_\ell}} \Bigm\vert \bz_j  \right) + \notag\\
&\quad  \left(2+2N^{-2}\right) \P\left( \frac{\|\bv_j\|\sqrt{\log N}}{\sqrt{d_\ell}} \le   |\langle \bz_j,\bz_k \rangle| \le 1 \Bigm\vert \bz_j \right).
\end{align} 
\end{proof}

\begin{proof}[Proof of Proposition \ref{prop:spec-gap}]
Suppose that \eqref{rst1:eq-norm-Uz} and \eqref{rst2:eq-norm-Uz} hold, which happens with probability at least $1-5K^2/N$ according to Lemma \ref{lem:norm-Uz}. For ease of exposition, let $\Delta := \bA - \E[\bA]$. Recall the definition of $p_k$ and $q_{k\ell}$ for $1\le k \neq \ell \le K$ in Lemma \ref{lem:p-q}. It follows from \eqref{step:thresholding} and Lemma \ref{lem:p-q} that the $(i,j)$-th element of $\bA$ satisfies $a_{ij} \sim \mathbf{Bern}(p_{ij})$ such that
\begin{align}\label{eq0:lem-spec-gap}
p_{ij} = \begin{cases}
p_k, & \text{if}\ \bz_i,\bz_j \in  S_k, \\
q_{k\ell}, &  \text{if}\ \bz_i\in S_k,\ \bz_j \in S_\ell,\ \text{and}\ k\neq \ell.
\end{cases}
\end{align}
According to \eqref{p:value}, \eqref{tau}, and $d_{\min} \gtrsim \log N$, one can verify that $p_k \in (0,1)$ is a constant for all $k\in [K]$. According to \eqref{q:value} and \eqref{tau}, we have for $1\le k \neq \ell \le K$,
\begin{align}\label{q:upper-bound}
q_{k\ell} & \le 2 - 2\Phi\left(\frac{\max_{k\neq \ell}\mathrm{aff}(S_k,S_\ell) + \alpha}{ \mathrm{aff}(S_k,S_\ell)+\alpha}\frac{(\sqrt{d_k}-\alpha)(\sqrt{d_\ell}-\alpha)}{(\sqrt{d_{\max}}-\alpha)^2}\right) + 4N^{-2} \notag \\
& \le 2-2\Phi\left( \frac{(\sqrt{d_k}-\alpha)(\sqrt{d_\ell}-\alpha)}{(\sqrt{d_{\max}}-\alpha)^2} \right) + 4N^{-2}.
\end{align} 
Now, we are devoted to bounding $\|\Delta\|^2=\|\Delta^2\|$. First, we consider the diagonal elements of $\Delta^2$. According to \eqref{step:thresholding}, we note that $a_{ij}$ for all $j\in [n]$ are mutually independent conditioned on $\bz_i \in \R^N$. This, together with $\delta_{ij}=a_{ij}-\E[a_{ij}] \in \left\{1-p_{ij},-p_{ij}\right\}$ and the Hoeffding's inequality for general bounded random variables (see, e.g., \citet[Theorem 2.2.6]{vershynin2018high}), yields that
\begin{align*}
\P\left( \left| \sum_{j=1}^N \left(\delta_{ij}^2 - \E[\delta_{ij}^2] \right) \right| \ge \sqrt{N\log N} \Bigm\vert \bz_i \right) \le 2\exp\left( -\frac{2N\log N}{N} \right) = 2N^{-2}.
\end{align*}
This, together with the union bound, yields that it holds with probability at least $1-2N^{-1}$ that for all $i\in [N]$,
\begin{align}\label{eq1:lem-spec-gap}
 \left| \sum_{j=1}^N \left(\delta_{ij}^2 - \E[\delta_{ij}^2] \right) \right| \le \sqrt{N\log N}. 
\end{align}
Due to the fact that $p_k \in (0,1)$ is a constant for all $k \in [K]$ and \eqref{q:upper-bound}, we have for all $1\le i < j \le N$,
\begin{align*}
\E[\delta_{ij}^2] = \E[a_{ij}^2] - \E^2[a_{ij}] = p_{ij}(1-p_{ij}) 
\end{align*}
is less than some constant. According to this and \eqref{eq1:lem-spec-gap}, it holds with probability at least $1-2N^{-1}$ that for all $i\in [N]$,
\begin{align}\label{eq2:lem-spec-gap}
\left|(\Delta^2)_{ii}\right| =  \left| \sum_{j=1}^N \delta_{ij}^2 \right|   \le \left|\sum_{j=1}^N \E[\delta_{ij}^2]\right| + \sqrt{N\log N} \lesssim N.
\end{align} 
Next, we consider the off-diagonal elements of $\Delta^2$.  According to \eqref{step:thresholding}, we note that $a_{ik}a_{jk}$ for all $k \neq i$ and $k\neq j$ are mutually independent conditioned on $\bz_i,\bz_j \in \R^N$ for all $1 \le i \neq j \le N$. This, together with $\delta_{ij}=a_{ij}-\E[a_{ij}]$ and the Hoeffding's inequality for general bounded random variables (see, e.g., \citet[Theorem 2.2.6]{vershynin2018high}), yields that
\begin{align*}
\P\left( \left| \sum_{k \neq i, k \neq j} \left(\delta_{ik}\delta_{jk} - \E[\delta_{ik}\delta_{jk}] \right) \right| \ge \sqrt{2N\log N} \Bigm\vert \bz_i,\bz_j \right) \le 2\exp\left( -\frac{4N\log N}{N} \right) = 2N^{-4}.
\end{align*}
According to Jensen's inequality, Lemma \ref{lem:ex-delta-ijk}, and $\E[\delta_{ik}\delta_{jk}\vert \bz_i,\bz_j]=\E[a_{ik}a_{jk} \vert \bz_i,\bz_j] - \E[a_{ik}\vert \bz_i]\E[a_{jk} \vert \bz_j]$, we obtain for $k\neq i, k\neq j$,
\begin{align*}
\left| \E[\delta_{ik}\delta_{jk} ] \right| \le  \E\left[\left|\E[\delta_{ik}\delta_{jk}\vert \bz_i,\bz_j] \right| \right] \lesssim \frac{d_{\max}}{d_{\min}\sqrt{\log N}}.
\end{align*}
These, together with the union bound, yields that it holds with probability at least $1-2N^{-2}$ that for all $1 \le i \neq j \le N$,
\begin{align*}
 \left| \sum_{k \neq i, k \neq j}  \delta_{ik}\delta_{jk} \right| & \le  \left| \sum_{k \neq i, k \neq j} \left(\delta_{ik}\delta_{jk} - \E[\delta_{ik}\delta_{jk}] \right) \right|  + \left| \sum_{k \neq i, k \neq j} \E[\delta_{ik}\delta_{jk}]  \right| \lesssim \sqrt{2N\log N} + \frac{d_{\max}N}{d_{\min}\sqrt{\log N}}. 
\end{align*}
As a result, it holds with probability at least $1-2N^{-2}$ that for any $1 \le i \neq j \le N$,
\begin{align}\label{eq3:lem-spec-gap}
\left|(\Delta^2)_{ij}\right| =  \left| \sum_{k=1}^N \delta_{ik}\delta_{jk}  \right| \lesssim \frac{d_{\max}N}{d_{\min}\sqrt{\log N}},
\end{align} 
Applying the union bound to \eqref{eq2:lem-spec-gap} and \eqref{eq3:lem-spec-gap} yields that
\begin{align*}
\|\Delta^2\| \le \max_{i\in [N]}\left|(\Delta^2)_{ii}\right| + \sqrt{\sum_{i\neq j} \left|(\Delta^2)_{ij}\right|^2} \lesssim  N +  \frac{d_{\max}N^2}{d_{\min}\sqrt{\log N}}
\end{align*}
holds with probability at least $1-6K^2N^{-1}$. This further implies 
\begin{align*}
\|\Delta \| \lesssim \sqrt{\frac{d_{\max}}{d_{\min}}}\frac{N}{\sqrt[4]{\log N}}.
\end{align*}
\end{proof}

\subsection{Proof of Theorem \ref{thm:init}} \label{appenSubsec:pfThmInit}
\begin{proof}
Suppose that \eqref{specgap} holds, which happens with probability at least $1-6K^2N^{-1}$ according to Proposition \ref{prop:spec-gap}. Given $\bz_i \in S_k^*$ and $\bz_j \in S_\ell^*$, recall that $p_{k\ell}=\P\left( |\langle \bz_i,\bz_j \rangle| \ge \tau \right)$ denotes the connection probability between any pair of data points that respectively belong to the subspaces $S_k^*$ and $S_\ell^*$ for all $1\le k,\ell \le K$. Let $\bB:=\{b_{k\ell}\}_{1\le k,\ell \le K}$, $\bC:=\bH^*\bB\bH^{*^T}$, $\bP:=\{p_{k\ell}\}_{1\le k,\ell \le K}$, and $\bm{D}:=\bH^*\bP\bH^{*^T}$, where $b_{k\ell}$ is defined in \eqref{bkl}. In addition, let $\hat{\bU},\bU \in \R^{n\times K}$ be respectively the eigenvectors of $\bA$ and $\bC$ associated with the $K$ leading eigenvalues. According to \eqref{step:thresholding}, one can verify that
\begin{align}\label{E[A]}
\E[\bA] = \bm{D} - \mathrm{diag}(\bm{D}).
\end{align}
We claim that $\bC$ is of rank $K$ and its smallest singular value is larger than $N_{\min}\gamma$, where $\gamma \ge 1 - \Phi(c)$ is given in Lemma \ref{lem:sing-B}. Indeed, let $\bm{\Lambda}=\mathrm{diag}\left(\sqrt{N_1},\dots,\sqrt{N_K}\right)$. Then, we have
\begin{align*}
\bC = \bH^*\bm{\Lambda}^{-1}\bm{\Lambda}\bB\bm{\Lambda}\left(\bH^{*}\bm{\Lambda}^{-1}\right)^T.
\end{align*}
One can verify that $\bH^*\bm{\Lambda}^{-1}$ has orthonormal columns and 
\begin{align*}
\sigma_{\min}(\bC) \ge \sigma_{\min}\left( \bm{\Lambda}\bB\bm{\Lambda} \right) \ge \sigma_{\min}^2\left(\bm{\Lambda} \right)\sigma_{\min}(\bB) = N_{\min}\gamma. 
\end{align*}
According to \eqref{E[A]} and Lemma \ref{lem:p-q}, we have
\begin{align*}
\|\E[\bA] - \bC\| & = \|\bm{D}-\bC-\mathrm{diag}(\bm{D})\| \le \|\bm{D}-\bC \| + \max_{1 \le k \le K}p_{kk} \\
& \le \|\bHs\|^2\|\bB-\bP\| + 1 \le \|\bH^{*^T}\bHs\|\|\bB-\bP\|_F   + 1 \lesssim \frac{ \kappa_d N_{\max}}{\sqrt{\log N}}.
\end{align*}
This, together with \eqref{specgap}, yields that 
\begin{align*}
\|\bA-\bC\|  \le \|\bA - \E[\bA] \| + \|\E[\bA] - \bC\| \lesssim \frac{\sqrt{\kappa_d} N}{\sqrt[4]{\log N}},
\end{align*}
where the last inequality is due to $\kappa_d \lesssim \sqrt{\log N}$. This, together with \citet[Lemma 5.1]{lei2015consistency}, $ \gamma \ge 1 - \Phi(c) $, and the fact that $\kappa_d$ is a constant, yields that there exists a $\bQ \in \mO^K$ such that
\begin{align}\label{eq1:thm-init}
\|\hat{\bU} - \bU\bQ\|_F \le \frac{2\sqrt{2K}}{N_{\min}\gamma}\|\bA-\bC\| \lesssim \frac{\sqrt{\kappa_d} N }{N_{\min} \sqrt[4]{\log N}}.
\end{align} 
According to \citet[Lemma 2.1]{lei2015consistency}, we have $\bU=\bH^*\bX$ for some $\bX \in \R^{K\times K}$ with $\|\bx_k  - \bx_\ell \| = \sqrt{1/N_k+1/N_\ell}$, where $\bx_k$ denotes $k$-th row of $\bX$. By letting $\bX^\prime=\bX\bQ$, we obtain
\begin{align*}
\bU\bQ =  \bH^*\bX^\prime,
\end{align*}  
where  $\|\bx_k^\prime - \bx_\ell^\prime\| = \sqrt{1/N_k+1/N_\ell}$. This, together with setting $\delta_k=1/{\sqrt{N_k}}$ in \citet[Lemma 5.3]{lei2015consistency} and \eqref{eq1:thm-init}, yields that 
\begin{align*}
\min_{ k \in [K]} N_k\delta_k^2 = 1 \gtrsim \frac{\kappa_d N^2}{N^2_{\min} \sqrt{\log N}}   \gtrsim  \|\hat{\bU} - \bU\bQ\|_F^2
\end{align*} 
where the first inequality is due to $\kappa_N^2\kappa_d  \le \sqrt{\log N}$. This implies that there exists a $\bQ \in \mO^K$ such that
\begin{align}\label{eq2:thm-init}
 \frac{\|\hat{\bU} - \bU\bQ\|_F^2}{\delta_k^2} \lesssim N_k,\ \text{for all}\ k \in [K].
\end{align} 
Let $\bar{\bU} = \bH^0\hat{\bX}$, where $(\bH^0,\hat{\bX})$ is an $(1+\varepsilon)$-approximate solution to Problem \eqref{k-means}. Moreover, we define
$T_k=\left\{ i\in \mC_k^*: \bar{\bu}_k - \bar{\bu}_\ell \ge \delta_k/2 \right\}$ for all $k\in [K]$, where $\bar{\bu}_i$ denote the $i$-th row of $\bar{\bU}$.
According to \eqref{eq2:thm-init} and \citet[Lemma 5.3]{lei2015consistency}, we have
\begin{align*}
\sum_{k=1}^K \frac{|T_k|}{N_k} \lesssim  \|\hat{\bU} - \bU\bQ\|_F^2 \lesssim  \frac{\kappa_d N^2}{N^2_{\min} \sqrt{\log N}} \lesssim \frac{\kappa_d\kappa_N^2}{\sqrt{\log N}},
\end{align*}
where the second inequality is due to \eqref{eq1:thm-init}. This  implies
\begin{align*}
\sum_{k=1}^K  |T_k| \lesssim \kappa_d\kappa_N^2\frac{N_{\max}}{\sqrt{\log N}}
\end{align*} 
Note that \citet[Lemma 5.3]{lei2015consistency} ensures that the membership is correctly recovered outside of $\cup_{k\in [K]} T_k$, then we have
\begin{align*}
d_F^2(\bH,\bHs) \lesssim  \kappa_d\kappa_N^2\frac{N_{\max}}{\sqrt{\log N}},
\end{align*}
which implies \eqref{eq:H0}. Then, we complete the proof. 
\end{proof}

\section{Proofs in Section \ref{subsec:update}}\label{appenSec:pfUpdate}

Recall that given an $\bH \in \mM^{N\times K}$, $\mC_k=\{i\in [N]:h_{ik} = 1\}$, $n_k=|\mC_k|$ for all $k \in [K]$, and $n_{k\ell}=|\mC_k \cap \mC_\ell^*|$ for all $k, \ell \in [K]$. We can verify that the number of misclassified points in $\{\mC_1,\dots,\mC_K\}$ represented by $\bH$ with respect to $\{\mC^*_1,\dots,\mC^*_K\}$ represented by $\bH^*$ is 
$
\|\bH - \bH^*\bQ_{\pi^*}\|_F^2/2,
$
where $\bQ_{\pi^*} \in \argmin_{\bQ \in \Pi_K} \|\bH - \bH^*\bQ\|_F$. Moreover, we can verify that for a permutation $\pi:[K] \rightarrow [K]$,
\begin{align}\label{n-B-kl}
\frac{1}{2}\|\bH-\bHs\bQ_{\pi}\|_F^2 = \sum_{k=1}^K \sum_{\ell\neq  \pi^{{-1}}(k)} n_{k\ell} = \sum_{k=1}^K \sum_{\ell \neq \pi(k)} n_{\ell k}.
\end{align}
and
\begin{align}\label{n:kl}
\sum_{\ell:\ell\neq  \pi^{{-1}}(k)} n_{k\ell} = \frac{1}{2}  \sum_{\ell:\ell\neq \pi^{{-1}}(k)} \sum_{i\in \mC_k \cap \mC_\ell^*} \|\bh_i-\bh_i^*\bQ_{\pi}\|^2,\ W_k(\bH) = \max\left\{ \sum_{\ell:\ell\neq \pi^{-1}(k)} n_{k\ell}, \sum_{\ell:\ell\neq \pi(k)} n_{\ell k} \right\},
\end{align}
where $W_k(\bH)$ is defined in \eqref{def:Gk}.
Using Lemma \ref{lem:cov-esti}, we can present a spectral bound on the deviation of the sample covariance of random vectors that follow a uniform distribution over the sphere from its mean. 
\begin{coro}\label{coro:cov-esti}
Suppose that $d_k \ge 4\log\left(Kd_k^2N\right)$ for all $k \in [K]$. For all $k,\ell \in [K]$, it holds with probability at least $1-2K/(Nd_{\min}^2)$ that 
\begin{align}\label{rst:coro-cov-esti}
\left\| \bm{\Psi}_{k\ell} - \frac{1}{d_\ell}\bI_{d_\ell}  \right\| \le \frac{5c_1}{4d_\ell} \left( \sqrt{\frac{d_\ell}{n_{k\ell}}} + \frac{d_\ell}{n_{k\ell}} \right),
\end{align}
where $\bm{\Psi}_{k\ell}$ is defined in \eqref{def:Bkl} and $c_1 > 0$ is an absolute constant.
\end{coro}
\begin{proof}[Proof of Corollary \ref{coro:cov-esti}] 
Since $i\in  \mC_\ell^*$, we have $\ba_i \in \mathrm{Unif}(\S^{d_\ell-1})$ according to the UoS model in Definition \ref{UoS}. Applying Lemma \ref{lem:cov-esti} to \eqref{def:Bkl} with  $u=\log\left( K d_\ell^2 N \right)$ yields that it holds with probability at least $1-2/(Kd_\ell^2N)$ that 
\begin{align*}
\left\| \bm{\Psi}_{k\ell} - \frac{1}{d_\ell}\bI_{d_\ell}  \right\| \le \frac{c_1}{d_\ell} \left( \sqrt{\frac{d_\ell+\log(Kd_\ell^2 N)}{n_{k\ell}}} + \frac{d_\ell+\log(K d_\ell^2 N)}{n_{k\ell}} \right) \le \frac{5c_1}{4d_\ell} \left( \sqrt{\frac{d_\ell}{n_{k\ell}}} + \frac{d_\ell}{n_{k\ell}} \right),
\end{align*}
where the second inequality is due to $d_\ell \ge 4\log(Kd_\ell^2N)$. This, together with the union bound, implies the desired result.  
\end{proof}

\subsection{Proof of Lemma \ref{lem:spec-B-kl}}

\begin{proof}[Proof of Lemma \ref{lem:spec-B-kl}] 
Suppose that \eqref{rst:coro-cov-esti} holds for all $k,\ell \in [K]$, which happens with probability at  least $1-2K/(Nd_{\min}^2)$ according to $N_k \gtrsim d_k \gtrsim \log N$ and Corollary \ref{coro:cov-esti}. According to \eqref{def:Bkl} and \eqref{n-1}, we have for all $k \in [K]$ ,
\begin{align*}
n_{\pi(k)k}  = |\mC_{\pi(k)} \cap \mC_{k}^*| = |\mC^*_k| - \sum_{\ell:\ell \neq {\pi(k)}} |\mC_\ell \cap \mC_{k}^*| = N_k - \sum_{\ell:\ell \neq \pi(k)}n_{\ell k} \ge N_{k} - W_k(\bH) \ge \frac{7}{8}N_{k}. 
\end{align*}
This, together with \eqref{rst:coro-cov-esti}, yields that for all $k\in [K]$,
\begin{align*}
\left\|\bm{\Psi}_{\pi(k)k} - \frac{1}{d_k}\bI_{d_k}\right\| \le \frac{5c_1}{4d_k}\left( \sqrt{\frac{8d_k}{7N_{k}}} + \frac{8d_k}{7N_{k}}\right) \le \frac{5c_1}{2d_k}\sqrt{\frac{8d_k}{7N_{k}}} \le \frac{1}{32d_k},
\end{align*}
where the third and last inequalities are due to $N_k \gtrsim d_k$ for all $k \in [K]$. This, together with Weyl's inequality, yields \eqref{rst1:lem-spec-B-kl}. Again, applying \eqref{rst:coro-cov-esti} to $\bm{\Psi}_{\pi(k)\ell}$ for all $\ell \neq k$ yields 
\begin{align*}
\left\|\bm{\Psi}_{\pi(k)\ell} - \frac{1}{d_\ell}\bI_{d_\ell}\right\| \le \frac{5c_1}{4d_\ell}\left( \sqrt{\frac{d_\ell}{n_{\pi(k)\ell}}} + \frac{d_\ell}{n_{\pi(k)\ell}}\right). 
\end{align*}
This, together with Weyl's inequality, yields \eqref{rst2:lem-spec-B-kl}. Then, the proof is completed. 
\end{proof}

\subsection{Proof of Lemma \ref{lem:contra-U-less}} 

\begin{proof}[Proof of Lemma \ref{lem:contra-U-less}] 
Suppose that \eqref{rst1:lem-spec-B-kl} and \eqref{rst2:lem-spec-B-kl} hold, which happens with probability at least $1-2K /({d_{\min}^2 N})$ according to Lemma \ref{lem:spec-B-kl}. Recall that
\begin{align}\label{grad-U-k}
\bG_{\bU_k}(\bH)=\sum_{i=1}^N h_{ik} \bz_i\bz_i^T\ \text{for all}\ k \in [K]. 
\end{align}
It follows from $\bH \in \mM^{N\times K}$ and $\mC_k=\{i\in [N]:h_{ik} = 1\}$ that $h_{ik}=1$ if $i \in \mC_k$ and $h_{ik}=0$ otherwise for all $i \in [N]$. Then, we note that 
\begin{align*}
\bG_{\bU_{\pi(k)}}(\bH) = \sum_{i\in \mC_{\pi(k)}} \bz_i\bz_i^T = \sum_{\ell=1}^K \sum_{i\in \mC_{\pi(k)} \cap \mC_\ell^*} \bz_i\bz_i^T  = \sum_{\ell=1}^K \sum_{i\in \mC_{\pi(k)} \cap \mC_\ell^*} \bU_\ell^*\ba_i \ba_i^T \bU_\ell^{*^T}   = \sum_{\ell=1}^K n_{{\pi(k)}\ell}  \bU_\ell^* \bm{\Psi}_{{\pi(k)}\ell} \bU_\ell^{*^T},
\end{align*}
where the third equality is due to \eqref{z=Ua} in Definition \ref{UoS} and the last equality follows from \eqref{def:Bkl}. 
To simplify the notations, we define 
\begin{align*}
\bA_\ell = \bU_\ell^* \bm{\Psi}_{{\pi(k)}\ell} \bU_\ell^{*^T}\ \text{for all}\ \ell \in [K],\quad \delta_i =  \sigma_i\left( \bG_{\bU_{\pi(k)}}(\bH) \right) - \sigma_{i+1}\left( \bG_{\bU_{\pi(k)}}(\bH) \right)\ \text{for all}\ i \in [d-1]. 
\end{align*}
On one hand, it follows from \eqref{rst1:lem-spec-B-kl}, $\sigma_{d_k}\left(\bm{\Psi}_{{\pi(k)}k}\right) \le \sigma_{d_k}(\bA_k)$, and $\sigma_1(\bA_k) \le \sigma_1\left(\bm{\Psi}_{{\pi(k)}k}\right) $ that
\begin{align}\label{eq8:lem-contra-U}
\frac{31}{32d_k} \le \sigma_{d_k}(\bA_k) \le \dots \le \sigma_1(\bA_k) \le \frac{33}{32d_k}.
\end{align}
On the other hand, it follows from $\bU_k^* \in \mO^{n\times d_k}$ that 
\begin{align}\label{eq9:lem-contra-U}
\sigma_{d_k+1}(\bA_k) = \dots = \sigma_{n}(\bA_k) = 0.
\end{align}
According to \eqref{def:Bkl}, \eqref{def:Gk}, and \eqref{init:U-noiseless}, we have for all $k\in [K]$,
\begin{align}\label{eq2:lem-contra-U}
n_{\pi(k)k} = |\mC_{k}^*| - \sum_{\ell:\ell \neq \pi(k)} |\mC_\ell \cap \mC_k^*| \ge N_k - W_k(\bH) \ge \frac{7}{8}N_{k},\ \sum_{\ell: \ell \neq k} n_{\pi(k)\ell} \le W_{\pi(k)}(\bH) \le \varepsilon N_{\min}.  
\end{align}
This, together with \eqref{rst2:lem-spec-B-kl}, yields that for all $ k \in [K]$ and $\ell \neq k$, 
\begin{align}\label{eq3:lem-contra-U}
n_{\pi(k)\ell} \sigma_1\left(\bm{\Psi}_{\pi(k)\ell}\right) & \le \frac{n_{\pi(k)\ell}}{d_\ell} + \frac{5c_1}{4}\sqrt{\frac{n_{\pi(k)\ell}}{d_\ell}} + \frac{5c_1}{4}   \le \frac{21n_{\pi(k)\ell}}{16d_\ell} + \frac{5c_1}{4}(c_1+1),  
\end{align}
where the second inequality is due to $2\sqrt{\alpha\beta} \le \alpha/\rho +\rho\beta$ for any $\alpha,\beta \ge 0$ and $\rho > 0$. Summing up \eqref{eq3:lem-contra-U} for all $ \ell \neq  k$ yields that for all $k\in [K]$,
\begin{align}\label{eq4:lem-contra-U}
\sum_{\ell:\ell \neq k} n_{\pi(k)\ell} \sigma_1\left(\bm{\Psi}_{\pi(k)\ell}\right) \le \frac{21}{16}\sum_{\ell:\ell \neq k} \frac{n_{\pi(k)\ell}}{d_\ell} + \frac{5c_1}{4}K(c_1+1)  \le  \frac{21\varepsilon N_{\min}}{16d_{\min}} + \frac{3\varepsilon N_k}{16 d_{\min}} \le \frac{3\varepsilon N_k}{2d_{\min}},
\end{align}
where the second inequality is due to \eqref{eq2:lem-contra-U} and $N_k \gtrsim d_k$. We now show \eqref{rst:dk}. According to $ \bG_{\bU_{\pi(k)}}(\bH) = n_{\pi(k)k}\bA_k + \sum_{\ell \neq k}n_{\pi(k)\ell}\bA_\ell$, we have for all $i\in [d-1]$,
\begin{align}\label{eq0:lem-contra-U}
\delta_i & =  \sigma_i\left( \bG_{\bU_{\pi(k)}}(\bH) \right)  -  \sigma_{i+1}\left(n_{\pi(k)k}\bA_k\right)  +  \sigma_{i+1}\left(n_{\pi(k)k}\bA_k\right) -  \sigma_{i+1}\left( \bG_{\bU_{\pi(k)}}(\bH) \right) \notag\\
& \le \sigma_{i}\left(n_{\pi(k)k}\bA_k\right) - \sigma_{i+1}\left(n_{\pi(k)k}\bA_k\right)  + 2\sigma_1\left( \sum_{\ell:\ell \neq k}n_{\pi(k)\ell}\bA_\ell \right), \notag\\
& \le n_{\pi(k)k}\left( \sigma_{i}\left(\bA_k\right) - \sigma_{i+1}\left(\bA_k\right)\right) + 2\sum_{\ell:\ell \neq k} n_{\pi(k)\ell} \sigma_1\left(\bm{\Psi}_{\pi(k)\ell}\right),
\end{align}
where the first inequality is due to Weyl's inequality. Plugging \eqref{eq8:lem-contra-U}, $n_{\pi(k)k} \le N_k$, and \eqref{eq4:lem-contra-U} into \eqref{eq0:lem-contra-U} yields that for all $i=1,\dots,d_k-1$,
\begin{align}\label{detal:dk-1}
\delta_i  \le \frac{N_k}{16d_k} + \frac{3\varepsilon N_k}{d_{\min}}\le \frac{7N_k}{16d_k},
\end{align}
where the second inequality is due to $\varepsilon \le d_{\min}/(8d_k)$.
Meanwhile, plugging \eqref{eq9:lem-contra-U} and \eqref{eq4:lem-contra-U} into \eqref{eq0:lem-contra-U} yields that for all $i=d_k+1,\dots,d$,
\begin{align}\label{detal:dk+1}
\delta_i \le  \frac{3\varepsilon N_k}{d_{\min}} \le \frac{3N_k}{8d_k},
\end{align}
where the second inequality is due to $\varepsilon \le d_{\min}/(8d_k)$. 
Note that $\bG_{\bU_{\pi(k)}}$ is a positive semidefinite matrix and satisfies
\begin{align}\label{eq1:lem-contra-U}
\sigma_{d_k}\left(\bG_{\bU_{\pi(k)}}(\bH)\right) & \ge \sigma_{d_k} \left(\bU_{k}^* n_{\pi(k)k} \bm{\Psi}_{\pi(k)k} \bU_{k}^{*^T}\right) - \sigma_1 \left( \sum_{\ell:\ell\neq k} \bU_\ell^* n_{\pi(k)\ell}  \bm{\Psi}_{\pi(k)\ell} \bU_\ell^{*^T} \right) \notag\\
& \ge n_{{\pi(k)k}}\sigma_{d_k}\left(\bm{\Psi}_{\pi(k)k}\right) - \sum_{\ell:\ell \neq k}n_{\pi(k)\ell}\sigma_1\left(\bm{\Psi}_{\pi(k)\ell}\right),
\end{align} 
where the first inequality is due to Weyl's inequality and the second inequality follows from $\sigma_d(\bA\bU^T) \ge \sigma_d(\bA)$, $\sigma_1(\bA\bU^T) \le \sigma_1(\bA)$ for $\bA \in \R^{d\times d}$ and $\bU \in \mO^{n\times d}$, and $\sigma_1(\bB+\bC) \le \sigma_1(\bB) + \sigma_1(\bC)$ for $\bB,\bC \in \R^{m\times n}$. 
Plugging \eqref{rst1:lem-spec-B-kl}, \eqref{eq2:lem-contra-U}, and \eqref{eq4:lem-contra-U} into \eqref{eq1:lem-contra-U} yields for all $k \in [K]$,
\begin{align}\label{eq5:lem-contra-U}
\sigma_{d_k}\left(\bG_{\bU_{\pi(k)}}(\bH)\right) \ge \frac{217N_{k}}{256d_k} -\frac{3\varepsilon N_k}{2d_{\min}} \ge \frac{217N_{k}}{256d_k} - \frac{3N_{k}}{16d_k} \ge \frac{169N_k}{256d_k},
\end{align}
where the second inequality is due to $\varepsilon \le d_{\min}/(8d_k)$. Applying Weyl's inequality to $\bG_{\bU_{\pi(k)}}(\bH)$ gives 
\begin{align*}
\sigma_{d_k+1}\left(\bG_{\bU_{\pi(k)}}(\bH)\right)  \le n_{\pi(k)k}\sigma_{d_k+1}\left(\bA_k\right) + \sigma_1\left( \sum_{\ell:\ell \neq k}n_{\pi(k)\ell}\bA_\ell \right) \le  \sum_{\ell:\ell \neq k}n_{\pi(k)\ell}\sigma_1\left( \bm{\Psi}_{\pi(k)\ell} \right) \le \frac{3N_k}{16d_{k}},
\end{align*}
where the second inequality is due to \eqref{eq9:lem-contra-U} and the last inequality follows from \eqref{eq4:lem-contra-U} and $\varepsilon \le d_{\min}/(8d_k)$.
This, together with \eqref{eq5:lem-contra-U}, yields
\begin{align*}
\delta_{d_k} \ge \frac{169N_k}{256d_k} - \frac{3N_k}{16d_{k}} = \frac{121N_k}{256d_k}.  
\end{align*}
This, together with \eqref{dk}, $\lambda_{\pi(k)i} = \sigma_i\left( \bG_{\bU_{\pi(k)}}(\bH) \right)$ for all $i \in [d]$ due to the positive semidefiniteness of $\bG_{\bU_{\pi(k)}}(\bH)$, \eqref{detal:dk-1}, and \eqref{detal:dk+1}, implies \eqref{rst:dk}. 

We next show \eqref{rst:lem-contra-U}. Note that
\begin{align}\label{eq7:lem-contra-U}
\left\| (\bI-\bU_{k}^*\bU_{k}^{*^T})\bG_{\bU_{\pi(k)}}(\bH) \right\| &= \left\| \sum_{\ell:\ell \neq k} (\bI-\bU_{k}^*\bU_{k}^{*^T})\bU_\ell^*n_{\pi(k)\ell}\bm{\Psi}_{\pi(k)\ell} \bU_\ell^{*^T} \right\| \notag \\
& \le \sum_{\ell:\ell \neq k} n_{\pi(k)\ell} \left\| (\bI-\bU_{k}^*\bU_{k}^{*^T})\bU_\ell^* \right\| \|\bm{\Psi}_{\pi(k)\ell}\| \le \sum_{\ell:\ell \neq k} n_{\pi(k)\ell} \sigma_1\left(\bm{\Psi}_{\pi(k)\ell}\right),
\end{align}
where the last inequality is due to $\left\| (\bI-\bU_{k}^*\bU_{k}^{*^T})\bU_\ell^* \right\|  \le 1$ for any $1\le k\neq \ell \le K$. Since $\bU_k$ consists of eigenvectors associated with the top $\bar{d}_k$ eigenvalues of $\bG_{\bU_k}(\bH)$ that is positive semidefinite, then we have $\bG_{\bU_k}(\bH)\bU_k = \bU_k\bm{\Sigma}_k$ by the eigenvalue equation, where $\bm{\Sigma}_k$ is a diagonal matrix with the $i$-th diagonal element being $\sigma_i\left(\bG_{\bU_k}(\bH)\right)$ for all $i\in [\bar{d}_k]$. This, together with \eqref{eq5:lem-contra-U}, implies $\bU_k=\bG_{\bU_k}(\bH)\bU_k\bm{\Sigma}_k^{-1}$ for all $k\in [K]$. According to \eqref{rst:dk}, we have
\begin{align}\label{eq6:lem-contra-U}
d(\bU_{\pi(k)},\bU_{k}^*) & = \left\| (\bI-\bU_{k}^*\bU_{k}^{*^T})\bU_{\pi(k)} \right\| = \left\| (\bI-\bU_{k}^*\bU_{k}^{*^T})\bG_{\bU_{\pi(k)}}(\bH)\bU_{\pi(k)}\bm{\Sigma}_{\pi(k)}^{-1} \right\| \notag \\
& \le \left\| (\bI-\bU_{k}^*\bU_{k}^{*^T})\bG_{\bU_{\pi(k)}}(\bH)\right\|\|\bm{\Sigma}_{\pi(k)}^{-1}\| \le \frac{ 256 d_k\sum_{\ell:\ell \neq k} n_{\pi(k)\ell} \sigma_1(\bm{\Psi}_{\pi(k)\ell})}{169N_{k}},
\end{align}
where the first equality uses \eqref{dist:subsp} and the last inequality is due to \eqref{eq5:lem-contra-U} and \eqref{eq7:lem-contra-U}. Substituting \eqref{eq3:lem-contra-U} into the above inequality and summing up from $k=1$ to $k=K$ yield that for all $k\in [K]$,
\begin{align*}
\sum_{k=1}^K d(\bU_{\pi(k)},\bU_{k}^*) & \le \sum_{k=1}^K \frac{256d_k}{169N_{k}}\cdot\frac{21}{16}\left( K(c_1+1)c_1 + \frac{1}{d_{\min}}\sum_{\ell:\ell \neq k} n_{\pi(k)\ell}   \right) \\
& \le  \frac{2d_{\max}}{N_{\min}}\left( K^2(c_1+1)c_1 + \frac{1}{d_{\min}}\sum_{k=1}^K \sum_{\ell:\ell \neq k} n_{\pi(k)\ell}   \right)  \\
& \le \frac{2d_{\max}}{N_{\min}} \left( \frac{1}{2d_{\min}}\|\bH-\bHs\bQ_{\pi}\|_F^2 + K^2(c_1+1)c_1 \right) \\
& \le  \frac{2d_{\max}}{N_{\min}}  \max\left\{ \frac{1}{d_{\min}} \|\bH-\bHs\bQ_{\pi}\|_F^2, 2K^2(c_1+1)c_1 \right\},
\end{align*}
where the third inequality is due to \eqref{n-B-kl} and $\sum_{k=1}^K \sum_{\ell\neq  \pi^{{-1}}(k)} n_{k\ell} = \sum_{k=1}^K \sum_{\ell\neq  k} n_{\pi(k)\ell}$. Then, we complete the proof. 
\end{proof}

\section{Proofs in Section \ref{subsec:assign}}\label{appenSec:pfAssign}

\subsection{Proof of Lemma \ref{lem:solution-T-h}} 

\begin{proof}[Proof of Lemma \ref{lem:solution-T-h}] 
Note that $\bh$ satisfying $\bh^T\bo_K = 1,\ \bh \in \{0,1\}^K$ is a vector that has exactly one $1$ and $(K-1)$ 0's. We can see that $\mT(\bg)$ is to find the minimum element of $\bg$. Then, the solution follows immediately. This also implies that for $\bQ \in \Pi_K$, $\bv \in \mT(\bg)$ if and only if $\bQ\bv \in \mT(\bQ\bg)$. 
\end{proof}

\subsection{Proof of Lemma \ref{lem:contra-T-h-less}} 

\begin{proof}[Proof of Lemma \ref{lem:contra-T-h-less}] 
According to Lemma \ref{lem:solution-T-h} and $g_\ell > g_k$ for a $k\in [K]$ and all $\ell \neq k$, we see that $\mT(\bg)$ is a singleton and $\{\bv\} = \mT(\bg)$ satisfies $v_k=1$ and $v_\ell = 0$ for all $\ell \neq k$. Let $\bg^\prime \in \R^K$ be arbitrary and $\bv^\prime \in \mT(\bg^\prime)$. It then follows from Lemma \ref{lem:solution-T-h} that $v^\prime_{k^\prime} = 1$ and $v_{\ell^\prime} = 0$ for some $k^\prime \in [K]$ and all $\ell^\prime \neq k^\prime$ satisfying $g^\prime_{k^\prime} \le g^\prime_{\ell^\prime}$. Suppose that $k^\prime = k$. Then, we have $\|\bv - \bv^\prime\|=0$, and thus \eqref{rst:contra-T-h} holds trivially. Suppose to the contrary that $k^\prime \neq k$. Then, we have $\|\bv - \bv^\prime\|=\sqrt{2}$. Moreover, we can compute
\begin{align*}
\|\bg-\bg^\prime\|^2 \ge \left( g_k-g^\prime_{k} \right)^2 +  \left(g_{k^\prime}-g^\prime_{k^\prime} \right)^2 \ge \frac{1}{2} ( \underbrace{g_k - g_{k^\prime}}_{\le -\delta} + \underbrace{g^\prime_{k^\prime} - g^\prime_{k}}_{\le 0} )^2 \ge \frac{1}{2}\delta^2.
\end{align*}
This, together with $\|\bv - \bv^\prime\|=\sqrt{2}$, implies the desired result \eqref{rst:contra-T-h}. 
\end{proof}

\subsection{Proof of Lemma \ref{lem-less:contra-H}} 

\begin{proof}[Proof of Lemma \ref{lem-less:contra-H}] 
Let the row vectors $\bg_i,\ \bg_i^* \in \R^K$ denote the $i$-th row of  $\bG_{\bH}(\bU)$ and $\bG_{\bH}(\bU^*)$, respectively. For all $i\in [N]$, note that $I_i=\{k \in [K]:h_{ik}^*=1\}$. 
According to the semi-random UoS model in Definition \ref{UoS}, we have for all $i\in [N]$ and $\ell \neq I_i$, 
\begin{align*}
g_{i\ell}^*  - g_{iI_i}^* & = \left( \|\bz_i\|^2  - \|\bU_\ell^{*^T}\bz_i\|^2\right) - \left( \|\bz_i\|^2 - \|\bU_{I_i}^{*^T}\bz_i \|^2 \right) =  \|\bU_{I_i}^{*^T}\bU_{I_i}^*\ba_i\|^2 - \|\bU_\ell^{*^T}\bz_i\|^2  \\
& = 1 - \|\bU_\ell^{*^T}\bz_i\|^2   \ge 1 - \max_{\ell \neq I_i} \|\bU_\ell^{*^T}\bz_i\|^2,
\end{align*} 
where the last equality is due to $\|\ba_i\|=1$. This, together with Lemma \ref{lem:solution-T-h} and $\|\bU_{\ell}^{*^T}\bz_i\| < 1$ for all $\ell \neq I_i$, implies
\begin{align}\label{eq2:lem-contra-H}
\{\bHs\}  =  \mT\left(\bG_{\bH}(\bU^*)\right). 
\end{align}
Besides, we note that for each $i\in [N]$ and $\bQ_\pi \in \Pi_K$, 
\begin{align*}
\|\bg_i\bQ_\pi^T-\bg_i^*\|^2 & = \sum_{k=1}^K \left( \|\bU_{\pi(k)}^T\bz_i\|^2 - \|\bU_{k}^{*^T}\bz_i\|^2 \right)^2  \le \sum_{k=1}^K \|\bU_{\pi(k)}\bU_{\pi(k)}^{T} - \bU_k^*\bU_k^{*^T} \|^2 \|\bz_i\|^4  = \sum_{k=1}^K d^2(\bU_{\pi(k)},\bU_k^*),
\end{align*}
where the first equality is due to $\bg\bQ_{\pi}^T = \begin{bmatrix}
g_{\pi(1)} & \dots & g_{\pi(K)}
\end{bmatrix}$. This, together with Lemma \ref{lem:contra-T-h-less} and \eqref{eq2:lem-contra-H}, implies for all $i\in [N]$,
\begin{align*}
\|\bar{\bh}_i - \bh_i^*\bQ_{\pi} \| & = \|\bar{\bh}_i\bQ_{\pi}^T - \bh_i^*\| \le \frac{2\|\bg_i\bQ_{\pi}^T - \bg_i^*\|}{1 - \max_{\ell \neq I_i} \|\bU_\ell^{*^T}\bz_i\|^2} \le \frac{2\sqrt{\sum_{k=1}^K d^2(\bU_{\pi(k)},\bU_{k}^*)}}{1 - \max_{\ell \neq I_i} \|\bU_\ell^{*^T}\bz_i\|^2},
\end{align*}
where the first inequality uses the fact that $\bH\bQ^T \in \mT(\bG_{\bH}(\bU)\bQ^T)$ for $\bQ \in \Pi_K$ if and only if $\bH \in \mT(\bG_{\bH}(\bU))$ due to Lemma \ref{lem:solution-T-h}.  
\end{proof}

With the preparations in Sections \ref{subsec:update} and \ref{subsec:assign}, we can analyze each iteration of the KSS method as follows. 
\begin{prop}\label{prop:basin-attra}
Let $\varepsilon \in \left(0, \frac{d_{\min}}{8d_{\max}}\right]$ be a constant. Suppose that Assumption \ref{AS:1} holds,  $N_{\min} \gtrsim  d_k \gtrsim \log N$ for all $k \in [K]$, and $\bH^t \in \mM^{N\times K}$ satisfies  
\begin{align}\label{prop:init}
\|\bH^t - \bH^*\bQ_\pi\|_F^2 \le 2\varepsilon N_{\min},
\end{align}
where $\bQ_{\pi} \in \argmin_{\bQ \in \Pi_K} \|\bH^t - \bHs\bQ \|_F$. Then, it holds with probability at least $1-2K/(d_{\min}^2N)-5K^2/N$ that $\hat{d}_{\pi(k)} = d_k$ for all $k \in [K]$,
\begin{align}\label{rst:prop-U} 
\sum_{k=1}^K d(\bU^{t+1}_{\pi(k)},\bU_{k}^*) \le   \frac{2d_{\max}}{N_{\min}} & \max\left\{\frac{1}{d_{\min}} \|\bH^t-\bHs\bQ_{\pi}\|_F^2,  2K^2(c_1+1)c_1 \right\},
\end{align}
and for all $i \in [N]$,
\begin{align}\label{rst:prop-H} 
\|\bh_i^{t+1} - \bh_i^*\bQ_{\pi}\|   \le  \frac{2}{1-\kappa}\sqrt{\sum_{k=1}^K d^2(\
\bU_{\pi(k)}^{t+1},\bU_{k}^*)},
\end{align}
where the row vectors $\bh^{t+1}_i, \bh_i^* \in \R^K$ respectively denote the $i$-th row of $\bH^{t+1}$ and $\bHs$. 
\end{prop} 

\begin{proof}[Proof of Proposition \ref{prop:basin-attra}] 
Suppose that \eqref{rst1:lem-spec-B-kl}, \eqref{rst2:lem-spec-B-kl}, and \eqref{rst1:eq-norm-Uz} hold, which happens with probability at least $1-5K^2/N-2K /({d_{\min}^2 N})$ according to  Lemma \ref{lem:spec-B-kl}, Lemma \ref{lem:norm-Uz},and the union bound. According to \eqref{n-B-kl}, we have $W_k(\bH^t) \le \|\bH^t - \bH^*\bQ_\pi\|_F^2/2 \le \varepsilon N_{\min}$. It follows from this and Lemma \ref{lem:contra-U-less} that $\hat{d}^{t+1}_{\pi(k)} = d_k$ for all $k \in [K]$ and \eqref{rst:prop-U}. Next, note that $I_i=\{k \in [K]:h_{ik}^*=1\}$ for all $i\in [N]$. This, together with \eqref{affi:re} in Assumption \ref{AS:1} and \eqref{rst1:eq-norm-Uz}, implies that for all $i\in [N]$ and $\ell \neq I_i$, 
\begin{align}\label{eq1:lem-contra-H}
\|\bU_\ell^{*^T}\bz_i\|^2 \le \left(\frac{\mathrm{aff}(S_{I_i},S_\ell^*)+\alpha}{\sqrt{d_\ell}-\alpha}\right)^2  \le \left(\frac{\kappa\sqrt{d_\ell}+\alpha}{\sqrt{d_\ell}-\alpha}\right)^2 = \left(\kappa + \frac{(1-\kappa)\alpha}{\sqrt{d_\ell}-\alpha}\right)^2 \le 2\kappa^2 \le \kappa, 
\end{align} 
where the third inequality is due to $d_k \gtrsim \log N$ for all $k \in [K]$ and the last inequality is due to $\kappa \le 1/2$. Using this and Lemma \ref{lem-less:contra-H}  yields \eqref{rst:prop-H}. 
\end{proof}

\subsection{Proof of Lemma \ref{lem:one-step}}

\begin{proof}[Proof of Lemma \ref{lem:one-step}] 
Suppose that \eqref{rst:prop-U} and \eqref{rst:prop-H} hold, which happens with probability at least  $1-2K/(d_{\min}^2N)-5K^2/N$ according to \eqref{init:lem:one-step}, $N_{\min} \gtrsim d_k$ for all $k\in[K]$, and Proposition \ref{prop:basin-attra}. According to \eqref{init:lem:one-step} and \eqref{rst:prop-U}, we obtain
\begin{align*}
\sum_{k=1}^K d(\bU^{t+1}_{\pi(k)},\bU_{k}^*) \le \frac{4d_{\max}}{N_{\min}}K^2(c_1+1)c_1.
\end{align*}
 This, together with \eqref{rst:prop-H}, $\kappa \le 1/2$, and $N_{\min} \gtrsim d_{\max}$, yields that for all $i \in [N]$,
\begin{align*}
\|\bh^{t+1}_i - \bh_i^*\bQ_{\pi}\| \le 4 \sum_{k=1}^K d(\bU^{t+1}_{\pi(k)},\bU_{k}^*) \le \frac{16d_{\max}}{N_{\min}}K^2(c_1+1)c_1 < 1.
\end{align*}
Since $\bh^{t+1}_i,\bh_i \in \{0,1\}^K$ for all $i\in [N]$, then we have $\bh^{t+1}_i = \bh_i^*\bQ_{\pi}$ for all $i\in [N]$. Thus, $\bH^{t+1} = \bHs\bQ_{\pi}$. Moreover, due to the fact that $\min_{\bQ \in \Pi_K} \|\bH^{t+1}- \bHs\bQ\|_F \le \|\bH^{t+1} - \bHs\bQ_{\pi}\|_F = 0$, the desired result is implied.  
\end{proof}

\subsection{Proof of Theorem \ref{thm:KSS}} \label{appenSubsec:pfThmKSS}

We should point out that a technical issue occurred in our analysis is that we cannot infinitely use the result in Proposition \ref{prop:basin-attra} infinitely due to the union bound. Then, we study $T = \Theta\left(\log\log N\right)$ iterates. 
\begin{proof}[Proof of Theorem \ref{thm:KSS}] 
For ease of exposition, let $\pi^t:[K] \rightarrow [K]$ be a permutation such that  
\begin{align}\label{min:Q}
\bQ_{\pi^t} \in \argmin_{\bQ \in \Pi_K} \|\bH^t - \bHs\bQ\|_F
\end{align} 
and the row vector $\bh_i \in \R^K$ denote the $i$-th row of $\bH \in \mM^{N\times K}$ for all $i \in [N]$. We first show (i). Suppose that $t \le T$ is a positive integer such that 
\begin{align*}
\|\bH^t  - \bH^*\bQ_{\pi^t}\|_F^2 \le  2K^2(c_1+1)c_1d_{\min}.
\end{align*}
Using Lemma \ref{lem:one-step}, it holds with probability at least $1-2K/(d_{\min}^2N)-5K^2/N$ that 
\begin{align*}
\bH^{t+1} = \bH^*\bQ_{\pi^{t+1}}. 
\end{align*}
Then, it suffices to consider that for all $t \le T$ such that 
\begin{align}\label{eq1:thm:basin-attra}
\|\bH^t  - \bH^*\bQ_{\pi^t}\|_F^2 >  2K^2(c_1+1)c_1d_{\min}.
\end{align} 
We first consider $t=0$. According to \eqref{H:thm:init}, Proposition \ref{prop:basin-attra}, and \eqref{eq1:thm:basin-attra}, it holds with probability at least $1-2K/(d_{\min}^2N)-5K^2/N$ that $\hat{d}^1_{\pi^0(k)} = d_k$ for all $k \in [K]$,
\begin{align}\label{eq2:thm:basin-attra}
\sum_{k=1}^K d(\bU^1_{\pi^0(k)},\bU_{k}^*) \le \frac{2d_{\max}}{N_{\min}d_{\min}} \|\bH^0 - \bHs\bQ_{\pi^0}\|_F^2
\end{align} 
 and for all $i\in [N]$,
\begin{align}\label{eq3:thm:basin-attra}
\|\bh_i^1 - \bh^*_i\bQ_{\pi^0}\| \le  \frac{2}{1-\kappa}\sqrt{\sum_{k=1}^Kd^2(\bU^1_{\pi^0(k)},\bU_{k}^*)}. 
\end{align}
Summing up \eqref{eq3:thm:basin-attra} from $i=1$ to $i=N$ gives
\begin{align}\label{eq4:thm:basin-attra}
\|\bH^1 - \bHs\bQ_{\pi^0}\|_F  \le \frac{2\sqrt{N}}{1-\kappa} \sqrt{\sum_{k=1}^Kd^2(\bU^1_{\pi^0(k)},\bU_{k}^*)} \le  \frac{2\sqrt{N}}{1-\kappa} \sum_{k=1}^Kd(\bU^1_{\pi^0(k)},\bU_{k}^*).
\end{align}
This, together with  \eqref{eq2:thm:basin-attra} and \eqref{H:thm:init}, yields that
\begin{align}\label{eq5:thm:basin-attra}
\|\bH^1 - \bHs\bQ_{\pi^0}\|_F \le  \frac{2\sqrt{N}}{1-\kappa}\frac{2d_{\max}}{N_{\min}d_{\min}} \|\bH^0 - \bHs\bQ_{\pi^0}\|_F^2 = \kappa_1 \|\bH^0 - \bHs\bQ_{\pi^0}\|_F,
\end{align}
where 
\begin{align}\label{kappa1}
\kappa_1 := \frac{2\sqrt{N}}{1-\kappa}\frac{2d_{\max}}{N_{\min}d_{\min}}\|\bH^0 - \bHs\bQ_{\pi^0}\|_F \le \frac{2\sqrt{N}}{1-\kappa}\frac{2d_{\max}}{N_{\min}d_{\min}}\frac{(1-\kappa)d_{\min}N_{\min}}{5d_{\max}\sqrt{N}} = \frac{4}{5}. 
\end{align}
Now, we use mathematical induction to show that it holds for all $t \in [T]$ that 
\begin{align}
& \sum_{k=1}^K d(\bU^{t+1}_{\pi^t(k)},\bU_{k}^*) \le \kappa_1^{2^t}\sum_{k=1}^K d(\bU^t_{\pi^{t-1}(k)},\bU_{k}^*),\label{U:contr} \\
& \|\bH^{t+1} - \bHs\bQ_{\pi^{t+1}}\|_F \le \kappa_1^{2^t}\|\bH^t - \bHs\bQ_{\pi^t}\|_F, \label{H:contr} \\
& \|\bH^t - \bHs\bQ_{\pi^{t-1}}\|_F  \le  \frac{2\sqrt{N}}{1-\kappa} \sum_{k=1}^Kd(\bU^t_{\pi^{t-1}(k)},\bU_{k}^*). \label{H:U}
\end{align}
We first verify \eqref{U:contr}, \eqref{H:contr}, and \eqref{H:U} for $t=1$. Due to \eqref{min:Q} and \eqref{eq5:thm:basin-attra}, we obtain
\begin{align}\label{eq8:thm:basin-attra}
\|\bH^1 - \bHs\bQ_{\pi^1}\|_F \le \|\bH^1 - \bHs\bQ_{\pi^0}\|_F \le \kappa_1 \|\bH^0 - \bHs\bQ_{\pi^0}\|_F. 
\end{align}
According to this, \eqref{kappa1}, Proposition \ref{prop:basin-attra}, and \eqref{eq1:thm:basin-attra}, it holds with probability at least $1-2K/(d_{\min}^2N)-5K^2/N$ that $\hat{d}^2_{\pi^1(k)} = d_k$ for all $k \in [K]$,
\begin{align}\label{eq6:thm:basin-attra}
\sum_{k=1}^K d(\bU^2_{\pi^1(k)},\bU_{k}^*) \le \frac{2d_{\max}}{N_{\min}d_{\min}} \|\bH^1 - \bHs\bQ_{\pi^1}\|_F^2
\end{align} 
 and for all $i\in [N]$,
\begin{align}\label{eq7:thm:basin-attra}
\|\bh_i^2 - \bh^*_i\bQ_{\pi^1}\| \le  \frac{2}{1-\kappa}\sqrt{\sum_{k=1}^Kd^2(\bU^2_{\pi^1(k)},\bU_{k}^*)}. 
\end{align}
Substituting \eqref{eq4:thm:basin-attra} with the first inequality of \eqref{eq8:thm:basin-attra} into \eqref{eq6:thm:basin-attra} yields that
\begin{align*} 
\sum_{k=1}^K d(\bU^2_{\pi^1(k)},\bU_{k}^*) & \le \frac{2d_{\max}}{N_{\min}d_{\min}}\|\bH^1 - \bHs\bQ_{\pi^1}\|_F \cdot \frac{2\sqrt{N}}{1-\kappa} \sum_{k=1}^Kd(\bU^1_{\pi^0(k)},\bU_{k}^*) \\
& \le  \frac{\kappa_1 \sqrt{N}}{1-\kappa}\frac{4d_{\max}}{N_{\min}d_{\min}} \|\bH^0 - \bHs\bQ_{\pi^0}\|_F\sum_{k=1}^Kd(\bU^1_{\pi^0(k)},\bU_{k}^*)  \\
& = \kappa_1^2 \sum_{k=1}^Kd(\bU^1_{\pi^0(k)},\bU_{k}^*),
\end{align*}
where the second inequality follows from \eqref{eq8:thm:basin-attra} and  the equality is due to \eqref{kappa1}. Thus, \eqref{U:contr} holds for $t=1$. According to \eqref{eq7:thm:basin-attra}, repeating the arguments in \eqref{eq4:thm:basin-attra} and \eqref{eq5:thm:basin-attra}, we obtain 
\begin{align*}
\|\bH^2 - \bHs\bQ_{\pi^1}\|_F  \le \frac{2\sqrt{N}}{1-\kappa} \sum_{k=1}^Kd(\bU^2_{\pi^1(k)},\bU_{k}^*)
\end{align*}
and
\begin{align*} 
\|\bH^2 - \bHs\bQ_{\pi^1}\|_F  & \le  \frac{2\sqrt{N}}{1-\kappa}\frac{2d_{\max}}{N_{\min}d_{\min}} \|\bH^1 - \bHs\bQ_{\pi^1}\|_F^2\\
&  \le \frac{2\sqrt{N}}{1-\kappa}\frac{2d_{\max}}{N_{\min}d_{\min}}\kappa_1\|\bH^0 - \bHs\bQ_{\pi^0}\|_F \|\bH^1 - \bHs\bQ_{\pi^1}\|_F  = \kappa_1^2\|\bH^1 - \bHs\bQ_{\pi^1}\|_F,
\end{align*}
where the second inequality is due to the equality in \eqref{eq5:thm:basin-attra}. This, together with \eqref{min:Q}, implies
\begin{align*}
\|\bH^2 - \bHs\bQ_{\pi^2}\|_F \le \kappa_1^2\|\bH^1 - \bHs\bQ_{\pi^1}\|_F. 
\end{align*} 
Thus, \eqref{H:U} and \eqref{H:contr} holds for $t=1$. Next, suppose that \eqref{U:contr}, \eqref{H:contr}, and \eqref{H:U} hold for all $t \ge 1$. Then, we can show that \eqref{eq8:thm:basin-attra},\eqref{eq6:thm:basin-attra}, and \eqref{eq7:thm:basin-attra} also hold for $t+1$ using the same arguments as those of $t=1$. Consequently, we can further show that \eqref{U:contr}, \eqref{H:contr}, and \eqref{H:U} hold for $t+1$ until $t=T$ . Finally, we use mathematical induction and deduce that for all $t\in [T]$, the desired results \eqref{U:contr} and \eqref{H:contr} hold with probability at least $1-(T+1)(2K/(d_{\min}^2N)+5K^2/N)$ according to the union bound. 
It follows from  \eqref{eq5:thm:basin-attra} and \eqref{H:contr} that 
\begin{align*}
d_F\left(\bH^{t+1}, \bH^*\right) \le \kappa_1^{2^t}\kappa_1^{2^{t-1}}\dots\kappa_1^{2^1}\|\bH^{1} - \bH^*\bQ_{\pi^{1}}\|_F  \le \kappa_1^{2^{t+1}-1}\|\bH^{0} - \bH^*\bQ_{\pi^{0}}\|_F = \kappa_1^{2^{t+1}-1} d_F\left(\bH^{0}, \bH^*\right).
\end{align*}

We next show (ii). It follows from $T=\log_2\left( \frac{\log\left((1-\kappa)\sqrt{d_{\min}}N_{\min}\right)-\log(5\sqrt{2}Kc_1\kappa_1d_{\max}\sqrt{N})}{\log(1/\kappa_1)} \right)+1$ that
\begin{align*}
d_F\left(\bH^{T-1}, \bH^*\right) &  \le \kappa_1^{2^{T-1}}\frac{(1-\kappa)d_{\min}N_{\min}}{5\kappa_1 d_{\max}\sqrt{N}} \le  Kc_1 \sqrt{2d_{\min}}.
\end{align*}
This, together with Lemma \ref{lem:one-step}, yields \eqref{HT:thm}. According to Proposition \ref{prop:basin-attra}, we also have $\hat{d}^{T+1}_{\pi^T(k)} = d_k$ for all $k \in [K]$. 
This, together with \eqref{HT:thm} and \eqref{z=Ua}  in Definition \ref{UoS}, yields 
\begin{align}\label{eq10:thm:basin-attra}
\bU^{T+1}_{\pi (k)} \bU^{{T+1}^T}_{\pi (k)} = \bU_k^*\bU_k^{*^T}\ \text{for all}\ k \in[K].
\end{align}
By letting $\bO_k = \bU_k^{*^T}\bU^{T+1}_{\pi (k)}$, we have
\begin{align*}
\bO_k^T\bO_k = \bU^{{T+1}^T}_{\pi (k)} \bU_k^{*}\bU_k^{*^T}\bU^{T+1}_{\pi (k)} = \bI_{d_k},
\end{align*}
where the second equality is due to \eqref{eq10:thm:basin-attra}. This implies $\bQ_k \in \mO^{d_k}$ for all $k \in [K]$. Then, we prove \eqref{UT:thm}. 
\end{proof}

\section{Auxiliary Lemmas}

\begin{lemma}\label{lem:unif-dis}
Suppose that $\ba \sim \mathrm{Unif}(\S^{d-1})$ and $\tbv \in \R^d$ is a fixed vector with $\|\tbv\|=1$. Let $\ba$ be decomposed as
\begin{align}\label{eq:a-decompose}
\ba = x \tilde{\bv} + \sqrt{1-x^2} \bb, 
\end{align}
where $x\in \R$ and $\bb\in \R^d$ satisfying $\langle \tbv, \bb \rangle = 0$ and $\|\bb\|=1$. There exists an orthogonal matrix $\bU \in \mO^d$ such that $\bU\tbv=\be_1$. Let
\begin{align}\label{lem:unif-dist-bbk}
\tilde{\bb}  = \frac{\bU^T\ba  - x \be_1 }{\sqrt{1-x^2}}
\end{align}
and $\bc \in \R^{d-1}$ such that $\bc=(\tilde{b}_{2},\cdots,\tilde{b}_{d})$. Then, it holds that $\bU\bb \sim \tilde{\bb}$, where
\begin{align*}
\tilde{b}_1 = 0,\ \bc \sim \mathrm{Unif}(\S^{d-2}).
\end{align*}
\end{lemma}
\begin{proof}[Proof of Lemma \ref{lem:unif-dis}] 
According to \eqref{eq:a-decompose} and the rotational invariance of a uniform distribution over sphere, we have
\begin{align*}
\bU\bb = \frac{\bU\ba - x\be_1}{\sqrt{1-x^2}} \sim  \tilde{\bb}. 
\end{align*}
Since $\langle \ba ,\tbv  \rangle = x$, then $\langle \bU^T\ba,\bU\tbv \rangle = \langle \bU^T\ba, \be_1 \rangle = x$. This implies $\tilde{b}_{1}=0$. Moreover, since $\ba \sim \mathrm{Unif}(\S^{d-1})$, then $\bU^T\ba \sim \mathrm{Unif}(\S^{d-1})$ due to the rotational invariance of a uniform distribution over sphere. Then, let $\by \sim \mN(\b0,\bI_d)$ such that $\bU^T\ba = \by/\|\by\|$. This, together with $\langle \bU^T\ba, \be_1 \rangle = x$, implies
\begin{align*}
y_1^2 = \frac{x^2\sum_{i\neq 1}y_i^2}{1-x^2}. 
\end{align*}
Then, we have
\begin{align*}
\|\by\|^2 = y_1^2 + \sum_{i\neq 1} y_i^2 = \frac{1}{1-x^2} \sum_{i\neq 1}y_i^2. 
\end{align*}
This, together with \eqref{lem:unif-dist-bbk}, implies that for any $i\neq 1$,
\begin{align*}
\tilde{b}_{i} = \frac{y_i}{\|\by\|\sqrt{1-x^2}} = \frac{y_i}{\sqrt{\sum_{i\neq 1}y_i^2}}. 
\end{align*}
Then, we complete the proof.
\end{proof}

\begin{lemma}\label{lem:phi-gap}
Consider the setting in Lemma \ref{lem:ex-delta-ijk}. Suppose that $\bv_i$, $\tilde{\bv}_i$ for $i \in [N]$ are defined in \eqref{eq:v} and $\beta_{ij}$ is defined as in \eqref{eq:beta} for some given $\bv_i$ and $\tilde{\bv}_j$. 
Then, it holds that
\begin{align}\label{rst1:lem:phi-gap}
\Phi\left(  \beta_{ij}  \right) - \Phi\left( \frac{\tau(\sqrt{d_\ell}-\alpha)}{\|\bv_i\|} \right) \lesssim \frac{1}{\sqrt{\log N}},
\end{align}
and
\begin{align}\label{rst2:lem:phi-gap}
\Phi\left( \frac{\tau(\sqrt{d_\ell}+\alpha)}{\|\bv_i\|} \right) - \Phi\left(  \beta_{ij}^\prime  \right) \lesssim \frac{1}{\sqrt{\log N}},
\end{align}
\end{lemma}
\begin{proof}[Proof of Lemma \ref{lem:phi-gap}] 
Suppose that \eqref{rst1:eq-norm-Uz} and \eqref{rst2:eq-norm-Uz} hold, which happens with probability at least $1-5K^2N^{-2}$ according to Lemma \ref{lem:norm-Uz}. Let $k=\{\ell \in [K]:h_{i\ell}^* = 1\}$. It follows from \eqref{eq:epsilon} that
\begin{align}\label{eq1:lem-phi-gap}
\|\bv_i\| \ge \frac{\mathrm{aff}(S_k^*,S_\ell^*)-\alpha}{\sqrt{d_k}+\alpha} \ge \frac{(1-\varepsilon)\mathrm{aff}(S_k^*,S_\ell^*)}{(1+\varepsilon)\sqrt{d_k}}.
\end{align}
This, together with \eqref{rst2:eq-norm-Uz}, yields that
\begin{align}\label{eq3:lem-phi-gap}
|\langle \bv_i,\tilde{\bv}_j \rangle| \le \frac{2\sqrt{\log N}}{\sqrt{d_k}-\alpha} \le \frac{2\sqrt{\log N}}{(1-\varepsilon)\sqrt{d_k}} \le \varepsilon\|\bv_i\|.
\end{align}
According to \eqref{tau} and \eqref{eq:epsilon}, we have
\begin{align}\label{eq2:lem-phi-gap}
\tau =  \frac{\max_{k\neq \ell}\mathrm{aff}(S_k^*,S_\ell^*)+\alpha}{ (\sqrt{d_{\max}}-\alpha)^2} \le \frac{\left(1+\varepsilon\right)\max_{k\neq \ell}\mathrm{aff}(S_k^*,S_\ell^*)}{(1-\varepsilon)^2d_{\max}} \le \frac{1+\varepsilon}{(1-\varepsilon)^2\sqrt{d_{\max}}}.
\end{align}
We first compute
\begin{align}\label{eq4:lem-phi-gap}
 \frac{ \tau (\sqrt{d_\ell}+\alpha)}{ \left(\|\bv_i\| - |\langle \bv_i,\tilde{\bv}_j \rangle| \right) \sqrt{1 - (\log N)/d_\ell}} - \frac{\tau(\sqrt{d_\ell}-\alpha)}{\|\bv_i\|} & \le  \frac{1+\varepsilon}{(1-\varepsilon)^2\sqrt{d_{\max}}}\left( \frac{1+\varepsilon}{(1-\varepsilon)^2} - (1-\varepsilon) \right)\frac{\sqrt{d_\ell}}{\|\bv_i\|} \notag\\
& \lesssim \frac{1}{\sqrt{\log N}\|\bv_i\|},
\end{align}
where the first inequality is due to \eqref{eq:epsilon} and \eqref{eq2:lem-phi-gap}.
We next compute
\begin{align}\label{eq5:lem-phi-gap}
\frac{ |\langle \bv_i,\tilde{\bv}_j \rangle|\sqrt{(\log N)/d_\ell}(\sqrt{d_\ell}+\alpha)}{ \left(\|\bv_i\| - |\langle \bv_i,\tilde{\bv}_j \rangle| \right) \sqrt{1 - (\log N)/d_\ell}} \le \frac{4\sqrt{\log N}|\langle \bv_i,\tilde{\bv}_j \rangle|}{\|\bv_i\|} \lesssim \frac{1}{\sqrt{\log N}\|\bv_i\|},
\end{align}
where the first inequality is due to \eqref{eq3:lem-phi-gap} and $d_{\min} \gtrsim \log^3N$ and the second one follows from the second inequality of \eqref{eq3:lem-phi-gap} and $d_{\min} \gtrsim \log^3N$.
Then, we obtain
\begin{align}\label{eq6:lem-phi-gap}
\beta_{ij} - \frac{\tau(\sqrt{d_\ell}-\alpha)}{\|\bv_i\|} & = \frac{ \tau (\sqrt{d_\ell}+\alpha)}{ \left(\|\bv_i\| - |\langle \bv_i,\tilde{\bv}_j \rangle| \right) \sqrt{1 - (\log N)/d_\ell}} - \frac{\tau(\sqrt{d_\ell}-\alpha)}{\|\bv_i\|} + \frac{ |\langle \bv_i,\tilde{\bv}_j \rangle|\sqrt{(\log N)/d_\ell}(\sqrt{d}+\alpha)}{ \left(\|\bv_i\| - |\langle \bv_i,\tilde{\bv}_j \rangle| \right) \sqrt{1 - (\log N)/d_\ell}} \notag \\
& \lesssim   \frac{1}{\sqrt{\log N}}  \frac{1}{\|\bv_i\|},
\end{align}
where the first inequality is due to \eqref{eq4:lem-phi-gap} and \eqref{eq5:lem-phi-gap}.
Moreover, we have
\begin{align*}
\Phi\left(  \beta_{ij} \right) - \Phi\left( \frac{\tau(\sqrt{d_\ell}-\alpha)}{\|\bv_i\|} \right) & \le \sqrt{\frac{1}{2\pi}} \exp\left( - \frac{\tau^2(\sqrt{d_\ell}-\alpha)^2}{2\|\bv_i\|^2} \right)\left( \beta_{ij}  - \frac{\tau(\sqrt{d_\ell}-\alpha)}{\|\bv_i\|} \right) \\
& \le  \frac{1}{\sqrt{2\pi\log N}}  \exp\left( - \frac{\left(\kappa\sqrt{d_{\min}}+\alpha \right)^2\left( 	\sqrt{d_{\ell}} - \alpha \right)^2}{2(\sqrt{d_{\max}}-\alpha)^4\|\bv_i\|^2} \right) \frac{1}{ \|\bv_i\|} \\
& \lesssim  \frac{d_{\max}}{d_{\min}\sqrt{\log N}}, 
\end{align*}
where the first inequality is due to the basic inequality for the integral, the second inequality uses the inequality of \eqref{eq2:lem-phi-gap} and \eqref{eq6:lem-phi-gap}, and the last inequality follows from $d_{\min} \gtrsim \log N$ and the fact that $ \exp\left( -cx^2/2 \right)x$ attains the maximum  at $x=1/\sqrt{c}$ when $x \in (0, \infty )$. The proof of \eqref{rst2:lem:phi-gap} follows from the same argument as above. 

\end{proof}

\section{Experiment Setups and Results in Section \ref{subsec:test-real}}\label{appen:expe}

In this section, we provide more implementation details and results for the experiments in Section \ref{subsec:test-real}. We use the real datasets  \emph{COIL-20} \cite{COIL20}, \emph{COIL-100} \cite{COIL100}, the cropped extended \emph{Yale B} \cite{GeBeKr01}, \emph{USPS}~\cite{hull1994database}, and ~\emph{MNIST} \cite{lecun1998mnist}.\footnote{The datasets \emph{COIL-20}, \emph{COIL-100}, the cropped extended \emph{Yale B}, and \emph{USPS} are downloaded from \url{http://www.cad.zju.edu.cn/home/dengcai/Data/data.html}. The dataset \emph{MNIST} is downloaded from \url{https://www.csie.ntu.edu.tw/~cjlin/libsvmtools/datasets/}.}  The information about the used real-world datasets can be found in Table \ref{table-3}. Before using these datasets in the experiments, we normalize them such that each sample has unit length. Note that the MNIST dataset contains $70000$ images of handwritten digits $0$-$9$. Following the preprocessing technique in \citet{you2016scalable,lipor2021subspace}, we represent each image by a feature vector of dimension $3472$ using the scattering convolutional network \cite{bruna2013invariant} and reduce the dimension of each vector to $500$ using PCA. 

\begin{table}[!htbp]
\caption{The parameters for the real datasets: $N$ is the number of samples, $n$ is the dimension of samples, and $K$ is the number of clusters.}
\label{table-3}
\begin{center}
\begin{tabular}{lccccc}
\toprule
Datasets & $N$ & $n$ & $K$  \\
\midrule
\emph{COIL20}& 1440 & 1024 & 20 \\ 
\emph{COIL100} & 7200 & 1024 & 100  \\
\emph{YaleB} & 2414 & 1024 & 38 \\
\emph{USPS}  & 9298 & 256 & 10  \\
\emph{MNIST} & 70000 & 780 & 10 \\
\bottomrule
\end{tabular}
\end{center}
\vskip -0.15in
\end{table}

Since the data points in real datasets generally do not follow the semi-random UoS model in Definition \ref{UoS}, we cannot guarantee good clustering performance if we directly apply the TIPS method for initializing the KSS method. Therefore, in the implementation of the TIPS method, we improve the idea of the thresholding inner product to construct the weight matrix $\bA=\{a_{ij}\}_{1 \le i, j \le N}$ by
\begin{align*}
a_{ij} = \begin{cases}
|\langle \bz_i,\bz_j \rangle|, &\ \text{if}\ |\langle \bz_i,\bz_j \rangle| \ge \tau\  \text{or}\ j \in \mT_i\ \text{and}\ i \neq j, \\
0, &\text{otherwise},
\end{cases}
\end{align*}
where $\mT_i \subseteq [N]\setminus\{i\}$ with $|\mT_i|=2$ satisfies $|\langle \bz_i,\bz_j \rangle| \ge |\langle \bz_i,\bz_k \rangle|$ for all $j \in \mT_i$ and $k \notin \mT_i$. Introducing $\mT_i$ is to ensure that each column of $\bA$ contains at least two non-zero elements.
For the implementation of the KSS method, we simply set $d_1=\dots=d_K=d$, where $d$ is given in Table \ref{table-4}. For all algorithms, we assume that $K$ is known and given in Table \ref{table-3}. We present the parameters of all the tested methods in Table \ref{table-4}. 

\begin{table}[!htbp]
\caption{Parameters setting of the tested methods in the experiments .}
\label{table-4}
\small
\begin{center}
\begin{tabular}{lcccccc}
\toprule
   & \emph{COIL20} & \emph{COIL100} & \emph{YaleB} & \emph{USPS} & \emph{MNIST}   \\
\midrule
KSS & $(d,\tau)=(10,0.98)$& $(d,\tau)=(10,0.98)$ & $(d,\tau)=(8,0.98)$ & $(d,\tau)=(9,0.99)$ & $(d,\tau)=(18,0.98)$   \\ 
SSC & $(\alpha,\rho)=(10,0.8)$ & $(\alpha,\rho)=(10,2)$ & $(\alpha,\rho)=(10,1)$ & $(\alpha,\rho)=(10,0.5)$ & $(\alpha,\rho)=(10,0.8)$ \\
TSC & $q=4$ & $q=3$ & $q=4$ & $q=5$ & $q=6$\\
GSC & $q=25$ & $q=15$ &  $q=20$ & $q=20$ & $q=20$\\
LRR & $\lambda=10^{-2}$ & $\lambda=10^{-3}$ & $\lambda=0.1$ & $\lambda=10^{-3}$ & $\lambda=10^{-2}$\\
LRSSC & $(\sigma,\lambda)=(0.2,0.5)$  & $(\sigma,\lambda)=(1, 2)$  & $(\sigma,\lambda) = (0.1,1)$ & $(\sigma,\lambda)=(10,1)$ & $(\sigma,\lambda) = (0.2,0.5)$ \\
OMP & $q = 2$ &  $q=2$ & $q=5$ & $q=25$ & $q=20$ \\
\bottomrule
\end{tabular}
\end{center}
\vskip -0.15in
\end{table}

To complement the result of recovery accuracy in Table \ref{table-2}, we also report the running time and clustering accuracy for all runs of each method in Table \ref{table-2-com}. 

\begin{table}[!htbp]
\caption{CPU times (in seconds) and the clustering accuracy of the tested methods on real datasets over 10 runs.}
\label{table-2-com}
\begin{center}
\begin{tabular}{lccccccc}
\toprule
 Accuracy & \emph{COIL20} & \emph{COIL100} & \emph{YaleB} & \emph{USPS} & \emph{MNIST}  \\
\midrule
KSS & {\bf 0.9187$\pm$0} & {\bf 0.8050$\pm$0.0040} & 0.6715$\pm$0.0253 & {\bf 0.8120$\pm$0.0164} & {\bf 0.8989$\pm$0.0796}\\ 
SSC & {\bf 0.9075$\pm$0.0164} & 0.6542$\pm$0.0165 & {\bf 0.8179$\pm$0.0074} &  0.6582$\pm$0.0002 & --\\
OMP & 0.5012$\pm$0.0168 & 0.3273$\pm$0.0083 & {\bf 0.7968$\pm$0.0216} & 0.1967$\pm$0.0071 & 0.5749$\pm$0 \\
TSC & 0.8271$\pm$0 & {\bf 0.7164$\pm$0.0093} & 0.4700$\pm$0.0092 & 0.6688$\pm$0.0002 & {\bf 0.8514$\pm$0}\\
GSC & 0.7896$\pm$0 & 0.6445$\pm$0.0084 & 0.6852$\pm$0.0135 & {\bf 0.9522$\pm$0.0001} & 0.5411$\pm$0.0427\\
LRR & 0.7161$\pm$0.0064 & 0.5403$\pm$0.0066 & 0.6534$\pm$0.0146 & 0.7129$\pm$0.0001 & --\\
LRSSC & 0.8194$\pm$0 & 0.5035$\pm$0.0101 & 0.6971$\pm$0.0097 & 0.6440$\pm$0.0005 & -- \\
\midrule
Time (s) & \emph{COIL20} & \emph{COIL100} & \emph{YaleB} & \emph{USPS} & \emph{MNIST}   \\
\midrule
KSS & 1.32$\pm$0.08 & 53.53$\pm$6.78 & 5.94$\pm$0.84 & {\bf 8.85$\pm$0.67} & {\bf 30.5287$\pm$13.15}\\ 
SSC & 55.37$\pm$4.99 & 912.25$\pm$42.12 & 136.36$\pm$13.64 & 1217.88$\pm$27.21 & -- \\
OMP & {\bf 0.62}$\pm$0.04 & {\bf 12.11$\pm$0.54} & {\bf 1.02$\pm$0.06} & 31.12$\pm$0.29 & 398.37$\pm$8.14\\
TSC & {\bf 0.66}$\pm$0.03 & {\bf 29.78$\pm$1.05} & {\bf 3.06$\pm$0.18} & {\bf 2.66$\pm$0.07} & {\bf 154.46$\pm$20.91}\\
GSC & 11.73$\pm$0.54 & 178.15$\pm$7.93 & 24.22$\pm$0.85 & 105.59$\pm$7.22 & 1800.00$\pm$0\\
LRR & 33.63$\pm$2.62 & 144.25$\pm$7.99 & 63.30$\pm$16.06 & 111.56$\pm$9.05 & -- \\
LRSSC & 73.31$\pm$3.45 & 1800.00$\pm$0 & 444.28$\pm$37.95 & 1800.00$\pm$0 & -- \\
\bottomrule
\multicolumn{6}{l}{\footnotesize ``--'' denotes out of memory.}
\end{tabular}
\end{center}
\vskip -0.15in
\end{table}

\end{appendix}

\end{document}